\documentclass[10pt, english,oneside,reqno]{amsart}

\usepackage[utf8]{inputenc}
\usepackage[english]{babel}
\usepackage{hyperref}
\usepackage{color}
\usepackage{enumerate}
\usepackage{fancyhdr}

\usepackage{amsmath}
\usepackage{amscd}
\usepackage{amsfonts}
\usepackage{amsthm}
\usepackage{amstext}
\usepackage{amssymb}
\usepackage{amsbsy}
\usepackage{mathrsfs}
\usepackage{tikz-cd}
\usepackage{mathtools}
\usepackage{enumitem}
\newtheorem*{corollary*}{Corollary}

\usepackage{float}
\usepackage{caption}
\usepackage{subcaption}
\usepackage{graphicx}

\usepackage[alphabetic,abbrev]{amsrefs}

\usepackage{url}

\theoremstyle{plain}
\newtheorem{theorem}{Theorem}[section]
\newtheorem{conjecture}[theorem]{Conjecture}
\newtheorem{definition}[theorem]{Definition}
\newtheorem{proposition}[theorem]{Proposition}
\newtheorem{corollary}[theorem]{Corollary}
\newtheorem{lemma}[theorem]{Lemma}

\theoremstyle{remark}

\newtheorem{remark}[theorem]{Remark}
\numberwithin{equation}{section}
\numberwithin{figure}{section}

\DeclareMathOperator{\var}{Var}

\DeclareMathOperator{\id}{Id}

\newcommand{\eps}{\varepsilon}
\newcommand{\tr}{\mathrm{tr}}

\newcommand{\Z}{\mathbb{Z}}
\newcommand{\Q}{\mathbb{Q}}
\newcommand{\R}{\mathbb{R}}
\newcommand{\C}{\mathbb{C}}

\newcommand{\E}{\mathbb{E}}
\newcommand{\Var}{\mathrm{Var}}

\newcommand{\RWEntropy}{h_{\mu}}

\newcommand{\Haarof}[1]{m_{#1}}

\newcommand{\supp}[1]{\mathrm{supp}(#1)}

\begin{document}
	\title[Absolutely continuous self-similar measures]{On absolute continuity of inhomogeneous and contracting on average self-similar measures}
	\author[sk]{Samuel Kittle}
    \author[ck]{Constantin Kogler}

    \begin{abstract}
        We give a condition for absolute continuity of self-similar measures in arbitrary dimensions. This allows us to construct the first explicit absolutely continuous examples of inhomogeneous self-similar measures in dimension one and two. In fact, for $d\geq 1$ and any given rotations in $O(d)$ acting irreducibly on $\R^d$ as well as any distinct translations, all having algebraic coefficients, we construct absolutely continuous self-similar measures with the given rotations and translations. We furthermore strengthen Varjú's result for Bernoulli convolutions, treat complex Bernoulli convolutions and in dimension $\geq 3$ improve the condition on absolute continuity by Lindenstrauss-Varjú. Moreover, self-similar measures of contracting on average measures are studied, which may include expanding similarities in their support.
    \end{abstract}

    \email{s.kittle@ucl.ac.uk, kogler@ias.edu}

    \address{Samuel Kittle, Department of Mathematics, University College London, 25 Gordon Street, London WC1H 0AY, United Kingdom}

    \address{Constantin Kogler, Institute for Advanced Study, 1 Einstein Dr, Princeton, NJ 08540, United States of America}

    \maketitle

    \tableofcontents
	
    \section{Introduction}

    In the study of self-similar measures it is fundamental to determine their dimension and to find conditions for absolute continuity. For the former problem progress was made by Hochman (\cite{Hochman2014}, \cite{Hochman2017}), relating the dimension of a self-similar measure to the entropy and Lyapunov exponent provided the generating measure satisfies a mild separation condition. While it was shown by Saglietti-Shmerkin-Solomyak \cite{SagliettiShmerkinSolomyak2018} that, under suitable assumptions, generic one-dimensional self-similar measures are absolutely continuous, finding explicit examples remains challenging. It was shown by Varjú \cite{Varju2019} that Bernoulli convolutions are absolutely continuous if their defining parameter is sufficiently close to $1$ in terms of the Mahler measure. In dimension $d\geq 3$, assuming that the rotation part of the self-similar measure is fixed and has an $L^2$ spectral gap on $O(d)$, Lindenstrauss-Varjú \cite{LindenstraussVarju206} showed absolute continuity if all of the contraction rates are sufficiently close to 1. In this paper we strengthen and vastly generalise these two results. Moreover, we give the first explicit examples of absolutely continuous self-similar measures in dimension one and two with non-uniform contraction rates. For instance, consider for $x\in \R$ the similarities 
    \begin{equation}\label{FirstInhomExample}
        g_1(x) = \frac{n}{n + 1}x  \quad \text{ and }   \quad g_2(x) = \frac{n}{n + 2}x + 1.
    \end{equation}
    We then show that the self-similar measure of $\frac{1}{2}\delta_{g_1} + \frac{1}{2}\delta_{g_2}$ is absolutely continuous on $\R$ for any sufficiently large integer $n \geq 1$. Furthermore, our methods allow us to construct several classes of explicit absolutely continuous examples for $g_i(x) = \rho_i U_i x + b_i$ for $x\in \R^d$ in any dimension $d\geq 1$ as well as for every collection of orthogonal matrices $U_i$ acting irreducibly on $\R^d$ and distinct vectors $b_i \in \R^d$, provided they all have algebraic entries. 

    Let $G = \mathrm{Sim}(\R^d)$ be the group of similarities on $\R^d$ and let $O(d)$ be the group of orthogonal $d\times d$ matrices. For each $g \in G$ there exists  a scalar $\rho(g) >  0$, an orthogonal matrix $U(g) \in O(d)$ and a vector $b(g) \in \R^d$ such that $g(x) = \rho(g)U(g)x + b(g)$ for all $x \in \R^d$. A similarity is called contracting if $\rho(g) < 1$ and expanding when $\rho(g) > 1$. 

    The Lyapunov exponent of a probability measure $\mu$ on $G$ is defined, whenever it exists, as $$\chi_{\mu} = \E_{g\sim\mu}[\log \rho(g)].$$ Throughout this paper we use the following terminology.  

    \begin{definition}
        If $\chi_{\mu} < 0$, we call $\mu$ \textbf{contracting on average}. Moreover, if every $g\in \mathrm{supp}(\mu)$ is contracting, we say that $\mu$ is \textbf{contracting}. When $\chi_{\mu} < 0$ and there is $g \in \mathrm{supp}(\mu)$ such that $\rho(g) > 1$, then we call $\mu$ \textbf{contracting only on average}.
    \end{definition}
    
    It is well-known (\cite{Hutchinson1981}, \cite{BarnsleyElton1988}, \cite{BougerolPicard1992}) that when $\mu$ is a finitely supported contracting on average probability measure on $G$, then there exists a unique probability measure $\nu$ on $\R^d$ that is $\mu$-stationary (i.e. $\nu$ satisfies $\mu * \nu = \nu$) and referred to as the self-similar measure of $\mu$. Under these assumptions, it follows from the moment estimates of \cite{GuivarchLePage2016}*{Proposition 5.1} that $\nu$ has a polynomial tail decay in the sense that there exists some $\alpha = \alpha(\mu) > 0$ such that as $R \to \infty$,
    \begin{equation}\label{PolynomialTailDecay}
        \nu(x \in \R^d \,:\, |x| \geq R) \ll_{\mu} R^{-\alpha}
    \end{equation}
    for an implied constant depending only on $\mu$. The authors have given in \cite{KittleKoglerTail} an independent proof of \eqref{PolynomialTailDecay} for contracting on average measures on arbitrary metric spaces. 
    
    Throughout this paper, we denote by $\nu$ the self-similar measure associated to $\mu$. If $\mu$ is (only) contracting on average, we say that $\nu$ is a (only) contracting on average self-similar measure. Moreover, $\mu$ or respectively $\nu$ is called homogeneous if there are $r \in \R_{>0}$ and $U\in O(d)$ such that $r = \rho(g)$ and $U = U(g)$ for all $g \in \mathrm{supp}(\mu)$. When this is not the case, we say that $\mu$ and $\nu$ are inhomogeneous. A particular goal of this paper is to give explicit examples of inhomogeneous as well as contracting only on average self-similar measures which are absolutely continuous.

    To state our main result, we first discuss the Hausdorff dimension of $\nu$, which is defined as $$\dim \nu = \inf\{ \dim E \,:\, E \subset \R^d \text{ measurable and } \nu(E) > 0 \}$$ where $\dim E$ is the Hausdorff dimension of $E$. In order to state the landmark results by Hochman \cite{Hochman2014}, \cite{Hochman2017}, recall that the random walk entropy of a finitely supported measure $\mu$ is defined as  
    \begin{equation}\label{Def:RWEntropy}
        h_{\mu} = \lim_{n \to \infty} \frac{1}{n}H(\mu^{*n}) = \inf_{n \geq 1} \frac{1}{n}H(\mu^{*n}),
    \end{equation}
    where $H(\cdot)$ is the Shannon entropy. Observe that if $\mathrm{supp}(\mu)$ has no exact overlaps, which means that $\mathrm{supp}(\mu)$ generates a free semigroup, then $h_{\mu} = H(\mu) = -\sum_i p_i\log p_i.$ 
    
    Moreover, as in \cite{Hochman2017}, denote by $d(\cdot,\cdot)$ the metric on $G$ defined for $g = \rho_1U_1 + b_1$ and $h = \rho_2U_2 + b_2$ as $$d(g,h) = |\log \rho_1 - \log \rho_2| + ||U_1 - U_2|| + |b_1 - b_2|$$ for $||\cdot||$ the operator norm and $|\cdot|$ the euclidean norm.
    
    To distinguish between the results for dimension and absolute continuity, denote $$\Delta_n = \min\{ d(g, h) \text{ for } g, h \in \mathrm{supp}(\mu^{*n}) \text{ with } g \neq h  \}  $$ and $$M_n = \min\left\{ d(g, h) \text{ for } g, h \in \bigcup_{i = 0}^n \mathrm{supp}(\mu^{*i}) \text{ with } g \neq h  \right\}.$$ Furthermore we set 
    \begin{equation}\label{Def:SepRate}
        S_n = - \frac{1}{n} \log M_n \quad\quad \text{ and } \quad\quad S_{\mu} = \limsup_{n \to \infty} S_n,
    \end{equation}
    where $S_{\mu}$ is referred to as the splitting rate.
    
    We call a subgroup $H$ of $O(d)$ irreducible if $H$ acts irreducibly on $\R^d$, i.e. the only $H$-invariant subspaces of $\R^d$ are $\{ 0 \}$ and $\R^d$. Moreover, we say that a measure $\mu = \sum_{i = 1}^n p_i \delta_{g_i}$ on $G$ or $O(d) \subset G$ is irreducible if the group generated by $\{ U(g_1), \ldots , U(g_n) \}$ is irreducible. When the elements in the support of $\mu$ have a common fixed point $x \in \R^d$, then $\delta_x$ is the self-similar measure of $\mu$. To avoid the latter case, we say that $\mu$ has no common fixed point if the similarities in $\mathrm{supp}(\mu)$ do not. 
    
    It follows by Hochman \cite{Hochman2017}, generalising \cite{Hochman2014}, that if $\mu$ is a finitely supported, contracting and irreducible probability measure on $G$ without a common fixed point such that $\Delta_n \geq e^{-cn}$ for some $c > 0$ and infinitely many $n \geq 1$, then $\dim \nu = \min\{ d, \frac{h_{\mu}}{|\chi_{\mu}|} \}.$ 
    
    In the accompanying paper \cite{KittleKoglerDimension} we use the techniques of this paper to generalise Hochman's result to contracting on average measures. Moreover, we show that a weaker requirement than exponential separation at all scales is sufficient (see \cite{KittleKoglerDimension} for a discussion). We work with $M_n$ instead of $\Delta_n$ for convenience only and in order to apply the general entropy gap results from \cite{KittleKoglerEntropy}.

    \begin{theorem}(\cite{KittleKoglerDimension}*{Theorem 1.2 and Theorem 1.3})\label{GeneralHochman}
        Let $\mu$ be a finitely supported, contracting on average and irreducible probability measure on $G$ without a common fixed point. Assume that either of the following two properties holds:
        \begin{enumerate}[label = (\roman*)]
            \item For some $c > 0$, $M_n \geq e^{-cn}$ for infinitely many $n\geq 1$,
            \item For some $\eps > 0$, $\log M_n \geq -n \exp((\log n)^{1/3 - \eps})$ for all sufficiently large $n\geq 1$.
        \end{enumerate} Then $$\dim \nu = \min\left\{ d, \frac{h_{\mu}}{|\chi_{\mu}|} \right\}.$$
    \end{theorem}

    It is well-established that $\dim \nu \leq \{ d, \frac{h_{\mu}}{|\chi_{\mu}|} \}$. Therefore, $\nu$ can only be absolutely continuous if $ h_{\mu} \geq d\,|\chi_{\mu}|$. The following folklore conjecture is expected to hold, which to the authors knowledge has not been explicitly stated before. 
    
    \begin{conjecture}\label{MainConjecture}
        Let $\mu$ be a finitely supported, contracting on average and irreducible probability measure on $G$ without a common fixed point. Then $\nu$ is absolutely continuous if $$\frac{h_{\mu}}{|\chi_{\mu}|} > d.$$
    \end{conjecture}
    
    The latter conjecture is completely open and not known for any class of self-similar measures. Our main result establishes a weakening of the conjecture. Indeed, we show Conjecture~\ref{MainConjecture} with the $d$ being replaced by a large constant depending on the given rotations multiplied by a term involving the logarithmic separation rate $\log S_{\mu}$.
    

    More precisely, our result is uniform in perturbations of the orthogonal matrices. To capture this, denote for every $k \geq 2$ by $$\mathcal{I}_{k} = \{ (U_1, \ldots , U_{k}) \in O(d)^{k} \,:\, \text{the group generated by } U_1,\ldots, U_{k} \text{ is irreducible}  \}.$$ As shown in Lemma~\ref{I_kOpen}, the set $\mathcal{I}_k$ is an open subset of $O(d)^{k}$. We now state our main result, for which we are allowed to choose the rotations of $\mu$ from a given compact subset of $\mathcal{I}_k$.

    \begin{theorem}\label{MainResultFixU}
        Let $d \geq 1$, $k \geq 2$ and $\eps \in (0,1)$. Given a compact subset $\mathcal{I} \subset \mathcal{I}_{k}$ there exists a constant $C\geq 1$ depending on $d,\eps$ and $\mathcal{I}$ such that the following holds. Let $\mu = \sum_{i = 1}^k p_i\delta_{g_i}$ be a contracting probability measure on $G$ without a common fixed point satisfying $(U(g_1), \ldots, U(g_k)) \in \mathcal{I}$ and $p_i \geq \eps$ for all $1 \leq i \leq k$ as well as $S_{\mu}<\infty$. Then the self-similar measure $\nu$ is absolutely continuous if 
        \begin{equation}\label{MainCondition}
            \frac{h_{\mu}}{|\chi_{\mu}|} > C \left( \max\left\{ 1, \log \frac{S_{\mu}}{h_{\mu}} \right\} \right)^2.
        \end{equation}
    \end{theorem}


    \begin{remark}\label{rhoRemark}
    We observe that by the assumptions on $\mu$ in Theorem~\ref{MainResultFixU} it follows that $h_{\mu}\leq \log k \leq \log \eps^{-1}$ and therefore $|\chi_{\mu}| < C^{-1}\log \eps^{-1}$. Thus, since $\mu$ is contracting and $p_i \geq \eps$ for all $1\leq i \leq k$, we can deduce that $\rho(g_i) \in (\tilde{\rho},1)$ for all $1 \leq i \leq k$ and $\tilde{\rho} \in (0,1)$ a constant sufficiently close to $1$ depending on $C$ and $\eps$.
    \end{remark}

    Theorem~\ref{MainResultFixU} is a special case of the more general Theorem~\ref{MainResult}, which requires a few new definitions we state in Section~\ref{subsection:MainResult}. When $d = 1$ we note that every probability measure on $O(1) = \{\pm 1 \}$ is irreducible and therefore we can set $\mathcal{I} = O(1)^k$.
    
    While Theorem~\ref{MainResultFixU} applies in the case when the $L^2(O(d))$-spectral gap of the measure $\sum_{i = 1}^k p_i\delta_{U(g_i)}$ is zero, the dependence of $C$ can be made more uniform in the presence of a spectral gap. To introduce notation, given a closed subgroup $H \subset G$ and assuming that $\mu_U$ is a probability measure on $O(d)$ with $\mathrm{supp}(\mu_U) \subset H$, we denote by $\mathrm{gap}_H(\mu_U)$ the $L^2$-spectral gap of $\mu_U$ in $H$ as defined in \eqref{SpectralGapDef}. Given a measure $\mu$ on $G$ we denote by $U(\mu)$ the pushforward of $\mu$ under the map $g \mapsto U(g)$.

    \begin{theorem}\label{SpectralGapTheorem}
        Let $d\geq 1$ and $\eps \in (0,1)$. Then there exists a constant $C \geq 1$ depending only on $d$ and $\eps$ such that the following holds. Let $\mu = \sum_{i = 1}^k p_i\delta_{g_i}$ be a contracting probability measure on $G$ without a common fixed point satisfying $p_i \geq \eps$ for all $1 \leq i \leq k$ as well as $S_{\mu}<\infty$. Assume further that $\mathrm{gap}_H(U(\mu)) \geq \eps > 0$ for $H$ the closure of the subgroup generated by the support of $U(\mu)$. Then the conclusion of Theorem~\ref{MainResultFixU} holds. 
    \end{theorem}
    
    We point out that in Theorem~\ref{SpectralGapTheorem} the constant $C$ is independent of the subgroup $H$ and the statement applies when $H$ is a finite irreducible subgroup of $O(d)$ as well as when $H$ is a positive dimensional irreducible Lie subgroup of $O(d)$. As is shown in section~\ref{section:mixnondeg}, this observation relies on uniform convergence of $\mu_U^{*n}$ towards the Haar probability measure $\Haarof{H}$ and on Schur's lemma implying that $\E_{h \sim \Haarof{H}}[|x \cdot hy|^2] = d^{-1}$ for any unit vectors $x,y \in \R^d$ and any irreducible subgroup $H\subset O(d)$.

    To construct explicit examples of absolutely continuous self-similar measures on $\R^d$, Theorem~\ref{MainResultFixU} requires to estimate $h_{\mu}, |\chi_{\mu}|$ and $S_{\mu}$. It is straightforward to deal with $|\chi_{\mu}|$ as it can be explicitly computed. Lower bounds on the random walk entropy follow in many cases (see Section~\ref{subsection:Entropy}) by the ping-pong lemma or by Breuillard's strong Tits alternative \cite{Breuillard2008}. It also holds that $h_{U(\mu)} \leq h_{\mu}$, so when $h_{U(\mu)} > 0$, we only need to control $|\chi_{\mu}|$ and $S_{\mu}$. With current methods we can usually only bound $S_{\mu}$ if all of the coefficients of the elements in the support of $\mu$ are algebraic. In the latter case, as shown in Section~\ref{HeightSeparationSection}, when all of the coefficients of elements in the support of $\mu$ lie in a number field $K$ and have logarithmic height at most $L$ (see \eqref{def:logheight}), then $S_{\mu} \ll_d L\cdot[K:\Q]$. We observe that $\log S_{\mu}$ is usually very small as it is double logarithmic in the arithmetic complexity of the coefficients. All this information makes it straightforward to find explicit examples of absolutely continuous self-similar measures. The constants $C$ in Theorem~\ref{MainResultFixU} can be computed from the involved terms, yet we do not make the dependence explicit in this work.

    The proof of Theorem~\ref{MainResultFixU} and Theorem~\ref{MainResult} builds on new techniques initiated by the first-named author in \cite{Kittle2023} and further developed in this paper, while being inspired by ideas from \cite{Hochman2014}, \cite{Hochman2017}, \cite{Varju2019} and \cite{Kittle2021}. We give an outline of our proof in Section~\ref{section:Outline} and note that the main novelties exploited are strong product bounds for detail at scale $r$ (a notion introduced in \cite{Kittle2021}) and a decomposition theory for stopped random walks to capture the amount of variance we can gain at a given scale, a technique we call the variance summation method. \cite{Kittle2023} is concerned with constructing absolutely continuous Furstenberg measures of $\mathrm{SL}_2(\R)$ on 1-dimensional projective space $\mathbb{P}^1(\R) = \R^2/\sim$ and an analogue of Theorem~\ref{MainResult} is shown. However, we currently can't deduce a result similar to Theorem~\ref{MainResultFixU} for Furstenberg measures of $\mathrm{SL}_2(\R)$ as the dynamics of the $\mathrm{SL}_2(\R)$ action on $\mathbb{P}^1(\R)$ are more difficult to control than the one of the $\mathrm{Sim}(\R^d)$ action on $\R^d$. Indeed, we exploit that one can rescale and translate self-similar measures without changing the Lyapunov exponent, the separation rate, the random walk entropy or the spectral gap of the generating measure. Moreover, an analogue of Theorem~\ref{MainResult} as well as Theorem~\ref{GeneralHochman} for Furstenberg measures of arbitrary dimensions is presently out of reach since the current methods cannot deal with non-conformal measures.

    To also treat contracting on average measures, we state the following version of Theorem~\ref{MainResultFixU}. We require some control on the scaling rate of the expanding similarities.

    \begin{theorem}\label{MainResultContAvFixU}
        Let $d\geq 2$, $k\geq 2$, $\eps \in (0,1)$ and $\mathcal{I} \subset \mathcal{I}_k$ be a compact subset and let $R > 1$. Let $\mu = \sum_{i = 1}^k p_i\delta_{g_i}$ be a contracting on average probability measure on $G$ without a common fixed point satisfying $(U(g_1), \ldots , U(g_k)) \in \mathcal{I}$ and $p_i \geq \eps$ as well as $\rho(g_i) \in [R^{-1},R]$ for all $1\leq i \leq k$. Then there is some $\tilde{\rho} \in (0,1)$ and $C > 1$ depending on $d,R, \eps$ and $\mathcal{I}$ such that the conclusion of Theorem~\ref{MainResultFixU} holds provided that for some $\hat{\rho} \in (\tilde{\rho},1)$ we have $$\frac{\E_{\gamma\sim \mu}[|\hat{\rho} - \rho(\gamma)|]}{1 - \E_{\gamma\sim\mu}[\rho(\gamma)]} < 1 - \varepsilon.$$ 
    \end{theorem}

    In the presence of a spectral gap, the analogue of Theorem~\ref{SpectralGapTheorem} also holds for Theorem~\ref{MainResultContAvFixU}. Using Theorem~\ref{MainResultFixU}, Theorem~\ref{MainResultContAvFixU} and Theorem~\ref{MainResult} one can construct a versatile collection of  explicit absolutely continuous self-similar measures. We give a few cases below and encourage the reader to find further examples.  Indeed, as shown in Corollary~\ref{QuiteGeneralExample} and Corollary~\ref{QuiteGeneralContAvExample}, for any given irreducible probability measure $\mu_U$ on $O(d)$ supported on matrices with algebraic entries and algebraic vectors $b_1, \ldots , b_k$ with $b_1 \neq b_2$, we can find explicit contracting as well as contracting only on average measures $\mu = \sum_{i = 1}^k p_i\delta_{g_i}$ on $G$ with $U(\mu) = \mu_U$ and $b(g_i) = b_i$ for $1 \leq i \leq k$ and having absolutely continuous self-similar measure. 

    \subsubsection*{\normalfont\textbf{Inhomogeneous Self-Similar Measures in Dimension $1$}}

    As a first example, we present results for self-similar measures supported on two similarities in dimension one. Upon conjugating, we can assume without loss of generality that our generating measure is supported on $x \mapsto \lambda_1x$ and $x \mapsto \lambda_2 x + 1$ for $\lambda_1, \lambda_2 \in (0,1)$.

    We recall the definition of the height of algebraic numbers, which measures the arithmetic complexity. For a number field $K$ and an algebraic number $\alpha \in K$ one defines the absolute height as  \begin{equation}\label{def:height}
        \mathcal{H}(\alpha) = \left(\prod_{v \in M_K} \max(1, |\alpha|_v)^{n_v} \right)^{1/[K:\Q]}
    \end{equation}
    where $M_K$ is the set of places of $K$, $n_v = [K_v : \mathbb{Q}_v]$ is the local degree at $v$ and $|\cdot|_v$ is the absolute value associated with the place $v$. We refer to \cite{MasserPolynomialsBook} for basic properties of heights and note that the height of $\alpha$ is independent of the number field $K$. We will also work with the logarithmic height 
    \begin{equation}\label{def:logheight}
        h(\alpha) = \log \mathcal{H}(\alpha).
    \end{equation}

    \begin{corollary}\label{Dim1Imhomogen}
        For every $\eps > 0$ there exists a small constant $c = c(\eps) > 0$ such that the following holds. Let $K$ be a number field and $\lambda_1,\lambda_2 \in K \cap (0,1)$ and write $h(\lambda_1,\lambda_2) = \max\{ h(\lambda_1), h(\lambda_2) \}$. Consider the similarities given for $x \in \R$ as 
        \begin{equation}\label{EasySimilarities}
            g_1(x) = \lambda_1 x \quad\quad \text{ and } \quad\quad g_2(x) = \lambda_2 x + 1.
        \end{equation}
        Then the self-similar measure of $\mu = \frac{1}{2}\delta_{g_1} + \frac{1}{2}\delta_{g_2}$ is absolutely continuous if $$ h(\lambda_1,\lambda_2) \geq \eps \quad\quad \text{ and } \quad\quad  |\chi_{\mu}| \max\{ 1, \log ([K:\Q] h(\lambda_1,\lambda_2) ) \} ^2 < c.$$
    \end{corollary}

     Corollary~\ref{Dim1Imhomogen} can be viewed as an inhomogeneous version of our strengthening of Varjú's result for Bernoulli convolutions (Corolarry~\ref{BernoulliCorollary}), yet with an additional dependence on the number field $K$ and on the lower bound of $\max\{ h(\lambda_1), h(\lambda_2) \}$. We further note that Lehmer's conjecture states the existence of an absolute $\eps_0 > 0$ such that $h(\lambda) \geq \eps_0/[K:\Q]$ for all $\lambda \in K$ and any number field $K$. Concretely, generalising the example discussed in \eqref{FirstInhomExample}, we can make the following conclusion.
    
    \begin{corollary}\label{Dim1ImhomogenEasy}
        There exists an absolute constant $c > 0$ such that the following holds. Let $\lambda_i = 1 - p_i/q_i \in (0,1)$ be rational for $i \in \{1,2 \}$ with coprime integers $p_i, q_i \geq 1$ and let $g_1$ and $g_2$ be the similarities from \eqref{EasySimilarities}. Then the self-similar measure of $\frac{1}{2}\delta_{g_1} + \frac{1}{2}\delta_{g_2}$ is absolutely continuous if for $i \in \{ 1,2 \}$, $$\frac{p_i}{q_i}\leq \frac{c}{(\log \log \max\{q_1, q_2\})^2}.$$ 
    \end{corollary}

    It is straightforward to adapt Corollary~\ref{Dim1Imhomogen} and Corollary~\ref{Dim1ImhomogenEasy} to multiple maps and also to include contracting on average measures. In the next section we discuss such examples in arbitrary dimensions and just mention the following simple example here. 

    \begin{corollary}\label{Dim1ContAverageEasy}
        For $n\geq 1$ consider the similarities given for $x\in \R$ as $$g_1(x) = \frac{n + 1}{n} x \quad\quad \text{ and } \quad\quad g_2(x) = \frac{n}{n+ 1} x + 1.$$ Then the self-similar measure of $\frac{1}{5}\delta_{g_1} + \frac{4}{5}\delta_{g_2}$ is absolutely continuous for sufficiently large $n$. 
    \end{corollary}
    
    \subsubsection*{\normalfont\textbf{Self-similar measures on $\R^d$}}

    With Theorem~\ref{MainResultFixU} and Theorem~\ref{MainResultContAvFixU} numerous explicit classes of absolutely continuous self-similar measures in $\R^d$ can be constructed. In order to apply these results we need to estimate $h_{\mu}$. In the following examples we have used the ping-pong lemma (see section~\ref{section:examples}) in two ways in order to establish lower bounds on $h_{\mu}$. For the first class of examples we have applied $p$-adic ping-pong as in Lemma~\ref{lemm:p_adic_ping_pong}.
            
    \begin{corollary}\label{QuiteGeneralExample}
    Let $d \geq 1$, $\eps \in (0,1)$ and $\mu_U = \sum_{i = 1}^k p_i\delta_{U_i}$ be an irreducible probability measure on $O(d)$ with $p_i \geq \eps$ and let $b_1, \ldots , b_k \in \R^d$ with $b_1 \neq b_2$. Assume that $U_1, \ldots , U_k$ and $b_1, \ldots , b_k$ have algebraic coefficients. Let $q$ be a prime number and for $1 \leq i \leq k$ consider $$g_i(x) = \frac{q}{q+a_{i,q}}U_ix + b_i \quad \quad \text{ for any integer } \quad \quad a_{i,q} \in [1, q^{1-\eps}].$$ Assume that $g_1, \ldots , g_k$ do not have a common fixed point and consider $\mu = \sum_{i = 1}^k p_i\delta_{g_i}$. Then the self-similar measure of $\mu$ is absolutely continuous for $q$ a sufficiently large prime depending on $d,\eps,U_1, \ldots , U_k$ and $b_1, \ldots , b_k$. 
    \end{corollary}

    We point out that any choice of integers $a_{i,q}$ works and that the necessary size of $q$ to derive absolute continuity does not depend on this choice, leading to a vast number of examples. Moreover, we can adapt Corollary~\ref{QuiteGeneralExample} to give contracting only on average examples. In order to satisfy the assumption from Theorem~\ref{MainResultContAvFixU}, we require that $\mu = \sum_{i = 1}^k p_i \delta_{g_i}$ satisfies that $p_k \leq \frac{1}{3}$. This nonetheless leads to absolutely continuous examples with $U(\mu) = \mu_U$ for any given irreducible probability measure $\mu_U = \sum_{i = 1}^k p_i\delta_{U_i}$ on $O(d)$ as we do not require that the $U_i$ are distinct.

    \begin{corollary}\label{QuiteGeneralContAvExample}
    Let $d, \eps$ and $\mu_U = \sum_{i = 1}^k p_i \delta_{U_i}$  as well as $b_1, \ldots , b_k$ be as in Proposition~\ref{QuiteGeneralExample}. Let $q$ be a prime number and  consider for $1 \leq i \leq k-1$ $$g_i(x) = \frac{q}{q+3}U_ix + b_i \quad \quad \text{ and } \quad \quad g_k(x) = \frac{q}{q-1}U_kx + b_k.$$ Assume that $g_1, \ldots , g_k$ do not have a common fixed point and further that $$p_k \leq \frac{1}{3}.$$ Then the self-similar measure of $\mu = \sum_{i = 1}^k p_i\delta_{g_i}$ is absolutely continuous for $q$ a sufficiently large prime depending on $d,\eps,U_1, \ldots , U_k$ and $b_1, \ldots , b_k$.
    \end{corollary}
    
    We give a second class of examples that rely on Galois ping-pong in as Lemma~\ref{lemm:p_adic_ping_pong}.

    \begin{corollary}\label{MostGeneralExample}
        Let $d \geq 1$, $\eps \in (0,1)$ and $\mu_U = \sum_{i = 1}^k p_i\delta_{U_i}$ an irreducible probability measure on $O(d)$ with $p_i \geq \eps$ for all $1\leq i \leq k$. Assume furthermore that $U_1, \ldots , U_k$ have algebraic entries. Let $C > 1$ be sufficiently large in terms of $d,\eps$ and $\mu_U$.  
        
        Suppose that $g_i(x) = \frac{a_i + b_i\sqrt{q}}{c_i}U_ix + d_i$ with $a_i, b_i, c_i \in \Z$ and $d_i \in \Z^d$ for $1 \leq i \leq k$ and a prime number $q$ are contracting and do not have a common fixed point. Then the self-similar measure associated to $\mu = \sum_{i = 1}^k p_i\delta_{g_i}$ is absolutely continuous if the following properties are satisfied: 
        \begin{enumerate}[label = (\roman*)]
            \item For $j = 1$ and for $j = 2$ we have $$\bigg|\frac{a_j - b_j\sqrt{q}}{c_j}\bigg| < \frac{1}{3},$$
            \item For $L = \max(\sqrt{q}, |a_i|, |b_i|, |c_i|, |d_i|_{\infty})$ we have $$C|\chi_{\mu}| \leq \frac{1}{(\log(\log L))^2}.$$
        \end{enumerate}
    \end{corollary}

    As a particular case of Corollary~\ref{MostGeneralExample}, we can consider as shown in Lemma~\ref{MostGeneralParticualr} the maps $$g_i(x) = \frac{\lceil \sqrt{q} \rceil - m_{i,q} + 2\sqrt{q}}{3\lceil \sqrt{q} \rceil}U_ix + d_i$$ for any $m_{i,q} \in \Z$ and $d_i \in \Z^d$ satisfying for some $\eps > 0$ that $$m_{i,q} \in [0,q^{1/2-\eps}] \quad\quad \text{ and } \quad\quad |d_i|_{\infty} \leq \exp(\exp(q^{\eps/3})).$$ Then the self-similar measure of $\mu = \sum_{i = 1}^n p_i\delta_{g_i}$ is absolutely continuous for sufficiently large primes $q$ depending on $d, \mu_U$ and $\eps$, provided that $g_1, \ldots , g_k$ do not have a common fixed point. We note that since we have a double exponential range for $d_i$, we get abundantly many examples. 

    \subsubsection*{\normalfont\textbf{Dimension $d\geq 3$}}

    The case $d\geq 3$ is discussed next. We recall that Lindenstrauss-Varjú \cite{LindenstraussVarju206} proved the following. Let $d\geq 3$, $\eps \in (0,1)$ and $\mu_U$ a finitely supported probability measure on $SO(d)$, whose support generates a dense subgroup of $SO(d)$ and with $\mathrm{gap}_{SO(d)}(\mu_U) \geq \eps$. Then there exists a constant $\widetilde{\rho}  \in (0,1)$ depending on $d$ and $\eps$ such that every finitely supported contracting probability measure $\mu = \sum_{i = 1}^k p_i \delta_{g_i}$ on $G$ with $U(\mu) = \mu_U$ and
    \begin{equation}\label{LindenstraussVarju}
    p_i \geq \eps \quad \text{ as well as } \quad \rho(g_i) \in (\tilde{\rho},1) \quad \quad  \text{ for all } \quad  1 \leq i \leq k
    \end{equation} has an absolutely continuous self-similar measure $\nu$. Moreover, \cite{LindenstraussVarju206} show that $\nu$ has a $C^k$-density if the constant $\widetilde{\rho}$ is in addition sufficiently close to $1$ in terms of $k$. By current methods (\cite{BourgainGamburd2008Invent}, \cite{BenoistDeSaxce2016}) spectral gap of $U(\mu)$ is only known when $\mathrm{supp}(U(\mu))$ generates a dense subgroup and all of the entries of elements in $\mathrm{supp}(U(\mu))$ are algebraic. 

    Instead of using the entropy and separation rate on $G$, we can use the same quantities on $O(d)$ and apply similar methods to Theorem~\ref{MainResultFixU} to establish absolute continuity of self-similar measures. Denote by $h_{U(\mu)}$ the random walk entropy of $U(\mu)$ as defined in \eqref{Def:RWEntropy} and by $S_{U(\mu)}$ the separation rate on $O(d)$ as in \eqref{Def:SepRate} yet with the metric $||A-B||$ for $A,B \in O(d)$ and $||\cdot||$ the operator norm. 
    
    We note that $h_{U(\mu)} \leq h_{\mu}$ yet we do not have in general that $S_{U(\mu)} \geq S_{\mu}$. In the case when $S_{U(\mu)} \geq S_{\mu}$, which for example holds when the support of $U(\mu)$ generates a free group, \eqref{LindenstraussVarju} follows from Theorem~\ref{MainResultFixU}. Moreover, our method can be adapted to work with $S_{U(\mu)}$ and $h_{U(\mu)}$ instead of $S_{\mu}$ and $h_{\mu}$ and we establish a generalisation of \eqref{LindenstraussVarju} (in the case when $\mathrm{supp}(\mu_U)$ consists of matrices with algebraic coefficients) that we state next. In the case that $U(g_1), \ldots , U(g_k)$ does not generate a virtually solvable subgroup, uniform bounds on $h_{U(\mu)}$ follow from the strong Tits alternative \cite{Breuillard2008}, which make our results particularly strong as stated in \eqref{MainTheoremHighDimUnifTits}.

    \begin{theorem}\label{MainTheoremHighDim}
        Let $d \geq 3$ and let $k, \eps, \mathcal{I}$ and $\mu$ be as in Theorem~\ref{MainResultFixU}. Then there exists a constant $C \geq 1$ depending on $d,\eps$ and $\mathcal{I}$ such that $\nu$ is absolutely continuous if 
        \begin{equation}\label{Eq:MainTheoremHighDim}
            \frac{h_{U(\mu)}}{|\chi_{\mu}|} \geq C\left( \max\left\{  1,\log \frac{S_{U(\mu)}}{h_{U(\mu)}} \right\} \right)^2.
        \end{equation}
    
        If in addition the group generated by $U(g_1), \ldots , U(g_k)$ is not virtually solvable, there exists a constant $c\leq 1$ depending on $d,\eps$ and $\mathcal{I}$ such that $\nu$ is absolutely continuous if
        \begin{equation}\label{MainTheoremHighDimUnifTits}
            |\chi_{\mu}| \leq \frac{c}{\max\{  1, \log S_{U(\mu)}  \}^2}.
        \end{equation}
    \end{theorem}

    The result \eqref{MainTheoremHighDimUnifTits} implies and generalises \eqref{LindenstraussVarju} (in the case when $\mathrm{supp}(\mu_U)$ consists of matrices with algebraic coefficients) as we do not require $\mathrm{supp}(\mu_U)$ to generate a dense subgroup of $O(d)$ or $SO(d)$. If, moreover, all of the coefficients of the matrices in $\mathrm{supp}(U(\mu))$ lie in the number field $K$ and have logarithmic height at most $L \geq 1$, then as shown in Section~\ref{HeightSeparationSection} it holds that $S_{U(\mu)} \ll_d L\cdot[K:\Q]$, which makes \eqref{MainTheoremHighDimUnifTits} particularly effective. We finally mention that it is straightforward to adapt Theorem~\ref{MainTheoremHighDim} to contracting on average measures, analogously to Theorem~\ref{MainResultContAvFixU}.

    As in \eqref{LindenstraussVarju}, to apply Theorem~\ref{MainTheoremHighDim}, we do not require that all the entries of elements in $\mathrm{supp}(\mu)$ are algebraic and only require the latter for $U(\mu)$, leading to parametric families of examples. The advantage of Theorem~\ref{MainTheoremHighDim} over \eqref{LindenstraussVarju} is that not only is our result somewhat uniform when the rotations are perturbed, yet it is also particularly effective when $U(\mu)$ has high entropy (for example when $\mathrm{supp}(U(\mu))$ generates a free semigroup).

    \subsubsection*{\normalfont\textbf{Real and Complex Bernoulli Convolutions}} We finally discuss Bernoulli convolutions. Although Theorem~\ref{MainResultFixU} applies to arbitrary self-similar measures, it gives new results for Bernoulli convolutions. Let $\lambda \in (1/2,1)$ and denote by $\nu_{\lambda}$ the unbiased Bernoulli convolution of parameter $\lambda$, i.e. the law of the random variable $\sum_{n = 0}^{\infty} \xi_n\lambda^n$ with $\xi_0,\xi_1, \ldots$ independent Bernoulli random variables with $\mathbb{P}[\xi_i = 1] = \mathbb{P}[\xi_i = -1] =  1/2$. It was shown by Solomyak \cite{Solomyak1995} that for almost all $\lambda \in (1/2,1)$ the Bernoulli convolution $\nu_{\lambda}$ has a density in $L^2(\R)$, while Erd\H{o}s \cite{Erdos1939} proved that $\nu_{\lambda}$ is singular whenever $\lambda^{-1}$ is a Pisot number.

    The Mahler measure of an algebraic number $\lambda$ is defined as $$M_{\lambda} = |a|\prod_{|z_j| > 1} |z_j|$$ with $a(x-z_1)\cdots(x-z_{\ell})$ the minimal polynomial of $\lambda$ over $\Z$. We note that as in Corollary 5.9 of \cite{Kittle2023} it holds that 
    \begin{equation}\label{SeparationMahlerBound}
        S_{\nu_{\lambda}} \leq \log M_{\lambda}.
    \end{equation}
    Garsia \cite[Theorem 1.8]{Garsia1962} showed that $\nu_{\lambda}$ is absolutely continuous for algebraic $\lambda$ with $M_{\lambda} = 2$, while the first-named author \cite{Kittle2021} established that $\nu_{\lambda}$ is absolutely continuous if $M_{\lambda} \approx 2$. In landmark work, Varjú \cite{Varju2019} proved for every $\eps > 0$ there is a constant $C>1$ such that that $\nu_{\lambda}$ is absolutely continuous if 
    \begin{equation}\label{Varju}
        \lambda > 1 - C^{-1}\min\{\log M_{\lambda},(\log M_{\lambda})^{-1 + \eps} \}.
    \end{equation}
    When applying Theorem~\ref{MainResultFixU} to Bernoulli convolutions we deduce the following strengthening of \eqref{Varju}, exploiting the comparison between the entropy and the Mahler measure for Bernoulli convolution due to \cite{BreuillardVarju2020}.

    \begin{corollary}\label{BernoulliCorollary}
        There is an absolute constant $C > 1$ such that the following holds. Let $\lambda \in (1/2,1)$ be a real algebraic number. Then the Bernoulli convolution $\nu_{\lambda}$ is absolutely continuous on $\R$ if 
        \begin{equation}\label{BernoulliCond}
            \lambda > 1 - C^{-1}\min\{\log M_{\lambda},(\log \log M_{\lambda})^{-2} \} .
        \end{equation}
    \end{corollary}

    We estimate that a direct application of our method would lead to $C \approx 10^{10}$ in Corollary~\ref{BernoulliCorollary}. It would be an interesting further direction to try to optimise $C$ for Bernoulli convolutions and in particular for the case $\lambda = 1 - \frac{1}{n}$. 
    
    Our most general result, Theorem~\ref{MainResult}, also applies to complex Bernoulli convolutions, which are defined analogously for $\lambda  \in \mathbb{D} = \{ \lambda \in \C \,:\, |\lambda| < 1 \}$. When $|\lambda| \in (0,2^{-1/2})$, then $\dim \nu_{\lambda} \leq \frac{\log 2}{|\log \lambda|} < 2$ and $\nu_{\lambda}$ is singular to the Lebesgue measure on $\C$. It was shown by Shmerkin-Solomyak \cite{ShmerkinSolomyak2016} that the set of $\lambda \in \C$ with $|\lambda| \in (2^{-1/2}, 1)$ and $\nu_{\lambda}$ is singular has Hausdorff dimension zero, whereas Solomyak-Xu \cite{SolomyakXu2003} showed that $\nu_{\lambda}$ is absolutely continuous on $\C$ for a non-real algebraic $\lambda \in \mathbb{D}$ with $M_{\lambda} = 2$. We extend Corollary~\ref{BernoulliCorollary} to complex parameters while assuming \eqref{eq:SomeImaginaryPart} in order to ensure that the rotation part of $\lambda$ mixes fast enough and so that our measure is sufficiently non-degenerate (see section~\ref{subsection:MainResult}). 

    \begin{corollary}\label{ComplexBernoulliCorollary}
    For every $\varepsilon > 0$ there is a constant $C \geq 1$ such that the following holds. Let $\lambda \in \mathbb{C}$ be a complex algebraic number such that $|\lambda| \in (2^{-1/2}, 1)$ and 
    \begin{equation}
        |\mathrm{Im}(\lambda)| \geq \varepsilon. \label{eq:SomeImaginaryPart}
    \end{equation} Then the Bernoulli convolution $\nu_{\lambda}$ is absolutely continuous on $\C$ if 
    \begin{equation*}
        |\lambda| > 1 - C^{-1}\min\{\log M_{\lambda},(\log \log M_{\lambda})^{-2} \} .
    \end{equation*}
    \end{corollary}

    \subsubsection*{\normalfont\textbf{Discussion of other work}}

    In addition to the discussed above \cite{Garsia1962}, \cite{SolomyakXu2003}, \cite{LindenstraussVarju206}, \cite{Varju2019} and \cite{Kittle2021} there is little known about explicit examples of absolutely continuous self-similar measures. To the authors knowledge, the only other papers addressing this topic are \cite{DaiFengWang2007} and \cite{Streck24}, which are concerned with homogeneous self-similar measures on $\R$ whose contraction rate $\lambda$ satisfies that all of its Galois conjugates have absolute value $< 1$. 

    A related problem is to study the Furstenberg measure of $\mathrm{SL}_2(\R)$ or of arbitrary simple non-compact Lie groups. The first examples of absolutely continuous Furstenberg measures arising from finitely supported generating measures were established by \cite{Bourgain2012}, giving an intricate number theoretic construction and also providing examples with a $C^k$-density for any $k\geq 1$. Bourgain's methods were generalised and further used by \cite{BoutonnetIoanaSalehiGolsefidy2017}, \cite{Lequen2022} and \cite{Kogler2022}. Moreover, numerous new examples we recently given by \cite{Kittle2023}.

    Returning to self-similar measures, we observe that the behavior of generic self-similar measures on $\R$ or $\C$ is better understood. \cite{Shmerkin2014} showed, thereby improving \cite{Solomyak1995}, that the set of $\lambda \in (1/2,1)$ such that the Bernoulli convolution $\nu_{\lambda}$ is singular has Hausdorff dimension zero. In \cite{SagliettiShmerkinSolomyak2018} it was shown that when the translation part (with distinct translations) and the probability vector is fixed, then generic one-dimensional self-similar measures on $\R$ are almost surely absolutely continuous in the range where the similarity dimension $>1$. This was generalised to $\C$ by \cite{SolomyakSpiewak2023}. A further line of research is to show that certain parametrized families of self-similar measures or other types of invariant function systems are generically absolutely continuous, see for example \cite{Hochman2014}, \cite{Hochman2017}, \cite{ShmerkinSolomyak2016Second} and \cite{BaranySimonSolomyak2022}.

    We finally mention that Fourier decay of self-similar measures was studied by numerous authors recently. The interested reader is referred to \cite{LiSahlsten2020}, \cite{Bremont2021}, \cite{LiSahlsten2022}, \cite{Rapaport2022}, \cite{Solomyak2022}, \cite{VarjuYu2022} and \cite{BakerKhalilSahlsten2024} and as well as \cite{AlgomRHWang} and \cite{BakerSahlsten2023} for self-conformal measures.
    
    \subsection*{Acknowledgement} The first-named author gratefully acknowledges support from the Heilbronn Institute for Mathematical Research. This work is part of the second-named author's PhD thesis conducted at the University of Oxford. We  thank Emmanuel Breuillard and Péter Varjú for comments on a preliminary draft and Timothée Bénard for pointing out that \eqref{PolynomialTailDecay} follows from \cite{GuivarchLePage2016}. 

    \section{Main Result and Outline}
    \label{OutlineNotation}

In this section we first state our main results and give an outline of the proof of the main theorem in section~\ref{section:Outline}. Then we collect for the convenience of the reader some notation used throughout this paper in section~\ref{section:Notation} and comment on the organisation of the paper in section~\ref{section:organisation}.

\subsection{Main Result} \label{subsection:MainResult}

Let $\mu$ be a probability measure on $G = \mathrm{Sim}(\R^d)$. To state our main results in full generality we introduce notions that capture how well $U(\mu)$ mixes on $O(d)$ and how degenerate $\nu$ is. 

Denote by $\gamma_1, \gamma_2, \dots$ independent samples from $\mu$, write $$q_n := \gamma_1 \gamma_2 \dots \gamma_n$$ and given $\kappa > 0$ let $\tau_{\kappa}$ be the stopping time defined by $$\tau_{\kappa} := \inf \{ n \geq 1 : \rho(q_{n}) \leq \kappa \}.$$ We then have the following definitions.
    \begin{definition}
        Let $\mu$ be a probability measure on $G$ generating a self-similar measure $\nu$. \begin{enumerate}[label = (\roman*)]
            \item We say that $\mu$ is \textbf{$(\alpha_0,\theta,A)$-non-degenerate} for $\alpha_0 \in (0,1)$ and $\theta,A > 0$ if for any proper subspace $W \subset \R^d$ and $y\in \R^d$, $$\nu(\{ x \in \R^d \,:\, |x - (y + W)| < \theta \text{ or } |x| \geq A \}) \leq \alpha_0.$$ 
            \item We say that $\mu$ is \textbf{$(c,T)$-well-mixing} for $c \in (0,1)$ and $T \geq 0$ if for any unit vectors $x, y \in \R^d$ we have $$\mathbb{E}[|x \cdot U(q_{F})y|^2] \geq c,$$ 
            where $F$ is a uniform random variable on $[0, T]$ which is independent of the $\gamma_i$.
        \end{enumerate}
    \end{definition}

    For $d=1$ the measure $\mu$ will always be $(1, 1)$-well-mixing. As we show in section~\ref{MixingSection}, when $U(\mu)$ is fixed and irreducible, there exists $(c,T)$ depending only on $U(\mu)$ such that $\mu$ is $(c,T)$-well-mixing. This follows as $U(q_F) \to \Haarof{H}$ in distribution as $T \to \infty$, where $H$ is the closure of the subgroup generated by $\mathrm{supp}(U(\mu))$ and $\Haarof{H}$ the Haar probability measure on $H$. The latter would not be true if we fix $F$ to be a deterministic random variable and therefore we have introduced the above definition. Moreover, it is easy to see that $(c,T)$-well-mixing in continuous in $\mu_U$.

    Dealing with non-degeneracy is more involved and uniform results for many classes of self-similar measures do not hold. However, instead of our given measure we can consider a conjugated measure to establish uniform non-degeneracy results. Indeed, for $\mu = \sum_{i = 1}^k p_i\delta_{g_i}$ a measure on $G$ and $h\in G$ we denote $$\mu_h = \sum_{i = 1}^k p_i\delta_{hg_ih^{-1}} \quad \text{ and } \quad \mu_h' =\frac{1}{2}\delta_e + \frac{1}{2}\sum_{i = 1}^k p_i\delta_{hg_ih^{-1}}.$$ Then as we show in Lemma~\ref{ConjugationLemma}, absolute continuity of any of the self-similar measures of $\mu,\mu_h$ or $\mu_h'$ is equivalent and all relevant quantities such as $h_{\mu}$, $S_{\mu}$ and $|\chi_{\mu}|$ are the same or comparable. 
    
    Towards Theorem~\ref{MainResultFixU}, Theorem~\ref{SpectralGapTheorem} and Theorem~\ref{MainResultContAvFixU}, as we state in Proposition~\ref{UnifNonDeg} and Proposition~\ref{UnifNonDegContAv} we have essentially uniform $(c,T)$-mixing and uniform $(\alpha_0,\theta, A)$-non-degeneracy as long as we assume $(U_1, \ldots , U_k) \in \mathcal{I}$. We first state a uniform mixing result adapted for Theorem~\ref{MainResultFixU} and Theorem~\ref{SpectralGapTheorem} in the contracting case. 

    \begin{proposition}\label{UnifNonDeg}
    Let $d\geq 1$, $\eps \in (0,1)$ and let $\mathcal{I} \subset \mathcal{I}_k$ be a compact subset. Then there exists $\tilde{\rho} \in (0,1)$, $(c,T)$ and $(\alpha_0,\theta, A)$ depending on $d, \eps$ and $\mathcal{I}$ such that the following holds. Let $\mu = \sum_{i = 1}^k p_i \delta_{g_i}$ be a contracting probability measure on $G$ without a common fixed point and with $(U_1, \ldots , U_k) \in \mathcal{I}$ and $$p_i \geq \eps \quad \text{ as well as } \quad \rho(g_i) \in (\tilde{\rho},1) \quad \quad  \text{ for all } \quad  1 \leq i \leq k.$$ Then there is $h\in G$ such that $\mu_h' =  \frac{1}{2}\delta_e + \frac{1}{2}\sum_{i = 1}^k p_i\delta_{hg_ih^{-1}}$ is $(c,T)$-well-mixing and $(\alpha_0,\theta, A)$-non-degenerate. 

    Moreover, if $\mathrm{gap}_H(U(\mu)) \geq \eps > 0$ for $H$ the closure of the subgroup generated by the support of $U(\mu)$ (without assuming $(U_1, \ldots , U_k) \in \mathcal{I}$), then there exist $\tilde{\rho}$, $(c,T)$ and $(\alpha_0,\theta, A)$ depending only on $d$ and $\eps$ such that the above conclusion holds.
    \end{proposition}

    For Theorem~\ref{MainResultContAvFixU} we state a similar result for contracting on average measures.

    \begin{proposition}\label{UnifNonDegContAv}
     Let $d\geq 1$, $\eps \in (0,1)$ and let $\mathcal{I} \subset \mathcal{I}_k$ be a compact subset. Then there exists $\tilde{\rho} \in (0,1)$, $(c,T)$ depending on $d, \eps$ and $\mathcal{I}$ and $(\alpha_0,\theta, A)$ such that the following holds. Let $\mu = \sum_{i = 1}^k p_i \delta_{g_i}$ be a contracting on average probability measure on $G$ without a common fixed point satisfying $(U_1, \ldots , U_k) \in \mathcal{I}$ and $p_i \geq \eps$ for all $1\leq i \leq k$. Assume that for some $\hat{\rho} \in (\tilde{\rho},1)$ we have $$\frac{\E_{\gamma\sim \mu}[|\hat{\rho} - \rho(\gamma)|]}{1 - \E_{\gamma\sim\mu}[\rho(\gamma)]} < 1 - \varepsilon.$$ Then there is $h\in G$ such that $\mu_h' =  \frac{1}{2}\delta_e + \frac{1}{2}\sum_{i = 1}^k p_i\delta_{hg_ih^{-1}}$ is $(c,T)$-well-mixing and $(\alpha_0,\theta, A)$-non-degenerate.
    \end{proposition}

    Proposition~\ref{UnifNonDeg} and Proposition~\ref{UnifNonDegContAv} are proved in section~\ref{section:mixnondeg}. We are now in a suitable position to state our main result. Theorem~\ref{MainResultFixU}, Theorem~\ref{SpectralGapTheorem} and Theorem~\ref{MainResultContAvFixU} follow from the main result Theorem~\ref{MainResult} by applying Proposition~\ref{UnifNonDeg}, Remark~\ref{rhoRemark} and Proposition~\ref{UnifNonDegContAv} as well as Lemma~\ref{ConjugationLemma}.

	\begin{theorem}\label{MainResult} For every $d \in \Z_{\geq 1}$ and $R, c, T, \alpha_0, \theta, A > 0$ with $c, \alpha_0 \in (0,1)$ and $T \geq 0$ there is a constant $C = C(d,R, c, T, \alpha_0, \theta,A)$ depending on $d, R, c, T, \alpha_0, \theta$ and $A$ such that the following holds. Let $\mu$ be a finitely supported, contracting on average, $(c,T)$-well-mixing and $(\alpha_0 , \theta,A)$-non-degenerate probability measure on $G$ with $\mathrm{supp}(\mu) \subset \{ g \in G \,:\, \rho(g) \in [R^{-1},R] \}$ and $S_{\mu} < \infty$. Then the associated self-similar measure $\nu$ is absolutely continuous if $$\frac{h_{\mu}}{|\chi_{\mu}|} > C \left( \max\left\{ 1, \log \frac{S_{\mu}}{h_{\mu}} \right\} \right)^2.$$  
    \end{theorem}

    For convenience of the reader, we formally deduce Theorem~\ref{MainResultFixU}, Theorem~\ref{SpectralGapTheorem} and Theorem~\ref{MainResultContAvFixU} from these results.

    \begin{proof}(of Theorem~\ref{MainResultFixU}, Theorem~\ref{SpectralGapTheorem} and Theorem~\ref{MainResultContAvFixU})
       We only prove Theorem~\ref{MainResultFixU} formally as the other deductions are analogous. Let $d,k,\eps, \mathcal{I},C$ and $\mu$ be as in Theorem~\ref{MainResultFixU}. Then by Remark~\ref{rhoRemark} there is $\tilde{\rho} \in (0,1)$ sufficiently close to $1$ in terms of $C$ and $\eps$ such that $\rho(g_i) \in (\tilde{\rho},1)$ for all $1\leq i \leq k$. Therefore, the assumptions of Proposition~\ref{UnifNonDeg} hold and so there is $(c,T)$ and $(\alpha_0,\theta, A)$ depending only on $d, \eps$ and $\mathcal{I}$ such that a conjugate $\mu'$ of $\frac{1}{2} + \frac{1}{2}\mu$ is $(c,T)$-well-mixing and $(\alpha_0,\theta, A)$-non-degenrate. As by Lemma~\ref{ConjugationLemma} it holds that $2\chi_{\mu} = \chi_{\mu'}$, $2h_{\mu} = h_{\mu'}$ and $S_{\mu} = S_{\mu'}$, Theorem~\ref{MainResultFixU} follows from Theorem~\ref{MainResult}.
    \end{proof}

    A similar result to Theorem~\ref{MainResult} for Furstenberg measures of $\mathrm{PSL}_2(\R)$ was established by the first-named author \cite{Kittle2023}. However in \cite{Kittle2023} it is necessary to assume that $\alpha_0 \in (0,1/3)$ and we currently can't prove an analogue of Proposition~\ref{UnifNonDeg} for Furstenberg measures. Therefore Theorem~\ref{MainResultFixU} can be deduced in the case of self-similar measures and we also note that the examples of absolutely continuous Furstenberg measures in \cite{Kittle2023} are more intricate.

    Next, we state a version of our main theorem for $d\geq 3$. Note that when $d\leq 2$, $h_{U(\mu)} = 0$ since $O(1)$ and $O(2)$ are abelian, so the following result only makes sense when $d \geq 3$. 

    \begin{theorem}\label{HighdimMainTheorem} Let $d \geq 3$ and $R, c, T, \alpha_0, \theta, A > 0$ with $c, \alpha_0 \in (0,1)$ and $T \geq 1$. Then there is a constant $C = C(d,R, c, T, \alpha_0, \theta,A)$ such that the following holds. Let $\mu$ be a finitely supported, contracting on average, $(c,T)$-well-mixing and $(\alpha_0 , \theta,A)$-non-degenerate probability measure on $G$ with $\mathrm{supp}(\mu) \subset \{ g \in G \,:\, \rho(g) \in [R^{-1},R] \}$. Moreover, assume that $S_{\mu} < \infty$. Then $\nu$ is absolutely continuous if $$\frac{h_{U(\mu)}}{|\chi_{\mu}|} \geq C \max\left\{ 1, \log \left( \frac{ S_{U(\mu)}}{h_{U(\mu)}} \right)  \right\}^2. $$    
    \end{theorem}

    We again formally deduce Theorem~\ref{MainTheoremHighDim} from Theorem~\ref{HighdimMainTheorem}.

    \begin{proof}(of Theorem~\ref{MainTheoremHighDim})
        As in the above proof of Theorem~\ref{MainResultFixU}, the result \eqref{Eq:MainTheoremHighDim} follows from Theorem~\ref{HighdimMainTheorem} and Proposition~\ref{UnifNonDeg}, Lemma~\ref{ConjugationLemma} and Remark~\ref{rhoRemark}. To deduce \eqref{MainTheoremHighDimUnifTits} from \eqref{Eq:MainTheoremHighDim}, it suffices to show that there is a constant $c$ depending on $d$ and $\eps$ such that $h_{U(\mu)} \geq c$, which follow as Corollary 6.9 of \cite{HochmanSolomyak2017} using Breuillard's uniform Tits alternative \cite{Breuillard2008}.
    \end{proof}

\subsection{Outline} \label{section:Outline}

We give a sketch for the proof of Theorem~\ref{MainResult}. Our proof extends the strategy of \cite{Kittle2023} to self-similar measures and generalises it to higher dimensions, which in turn is inspired by ideas and techniques developed in \cite{Hochman2014}, \cite{Hochman2017}, \cite{Varju2019} and \cite{Kittle2021}. Proposition~\ref{UnifNonDeg} will be discussed and proved in section~\ref{section:mixnondeg}. An entropy theory for random walks on general Lie groups was developed in \cite{KittleKoglerEntropy} and will be used in this paper.

Let $\mu$ be a measure on $G = \mathrm{Sim}(\R^d)$ and let $\gamma_1, \gamma_2, \ldots$ be independent $\mu$-distributed random variables. For a stopping time $\tau$ write $q_{\tau} = \gamma_1 \cdots \gamma_{\tau}$. Note that if $x$ is a sample of $\nu$ then so is $q_{\tau} x$. The basic idea of our proof is to decompose $q_{\tau} x$ as a sum 
\begin{equation}\label{BasicDecomposition}
    q_{\tau} x = X_1 + \dots + X_{n}
\end{equation}
with $X_1, \ldots , X_n$ independent random variables. We aim to show that for each scale $r > 0$ and a suitable stopping time $\tau$ that we can find a decomposition \eqref{BasicDecomposition} such that for all $i \in [n]$, 
\begin{equation}\label{Outline:DecompositionGoal}
    |X_i| \leq C^{-1}r \quad \text{ and } \quad  \sum_{i = 1}^n \mathrm{Var} \, X_i \geq C(\log \log r^{-1}) r^2I 
\end{equation} for a sufficiently large fixed constant $C = C(d) > 0$ only depending on $d$, where $\Var \, X_j$ is the covariance matrix of $X_j$ and we denote by $\geq$ the partial order defined in \eqref{MatrixPartialOrder}. The proof of Theorem~\ref{MainResult} comprises to establish \eqref{Outline:DecompositionGoal} and to deduce from \eqref{Outline:DecompositionGoal} that $\nu$ is absolutely continuous. For the former we use adequate entropy results and for the latter we work with the detail of a measure. The constant $C$ is be closely related to the one from Theorem~\ref{MainResult}.

\subsubsection*{\normalfont\textbf{From Decomposition to Absolute Continuity}} The notion of detail $s_{r}(\nu)$ at scale $r > 0$ of a measure $\nu$ is a tool introduced in \cite{Kittle2021} measuring how smooth $\nu$ is at scale $r$. Detail is an analogue of the entropy between scales $1 - H(\nu; r|2r)$ used by \cite{Varju2019}, yet with better properties. Our goal is to deduce from \eqref{Outline:DecompositionGoal} that our self-similar measure $\nu$ satisfies for $r$ sufficiently small,  
\begin{equation}\label{Outline:DetailDecay}
    s_{r}(\nu) \leq (\log r^{-1})^{-2},
\end{equation}
which implies that $\nu$ is absolutely continuous, as shown in \cite{Kittle2021}.  

A novelty introduced in \cite{Kittle2023} is a strong product bound for detail on $\R$, which we prove for $\R^d$ in this paper. Indeed, if $\lambda_1, \ldots, \lambda_k$ are measures on $\R^d$, $a < b$ and $r > 0$ with $s_r(\lambda_i) \leq \alpha$ for some $\alpha > 0$ and all $r \in [a,b]$ and $1\leq i \leq k$, then, as shown in Corollary~\ref{StrongProduct}, 
\begin{equation}\label{Outline:StrongProduct}
    s_{a\sqrt{k}}(\lambda_1 * \cdots * \lambda_k) \leq Q'(d)^{k-1}(\alpha^k + k! ka^2b^{-2})
\end{equation}
for some constant $Q'(d)$ depending only on $d$. To prove \eqref{Outline:StrongProduct}, \cite{Kittle2023} introduced $k$ order detail, which we generalise to $\R^d$. We note that \eqref{Outline:StrongProduct} is stronger than the product bounds \cite[Theorem 1.17]{Kittle2021} and \cite[Theorem 3]{Varju2019} and is required in our proof. 

To convert \eqref{Outline:DecompositionGoal} into \eqref{Outline:DetailDecay}, we first partition $[n]$ as $J_1 \sqcup \ldots \sqcup J_k$ for $k\asymp \log \log r^{-1}$ such that the random variables $Y_j = \sum_{i \in J_j} X_i$ satisfy $\mathrm{Var}\, Y_j \gg_d C$. Then we apply a Berry-Essen type result to deduce that each $Y_j$ is well-approximated by a Gaussian random variable and therefore that $s_r(Y_j) \leq \alpha$ for some constant $\alpha$ depending on $C$, with $\alpha$ tending to zero as $C$ tends to $\infty$. Finally we conclude by \eqref{Outline:StrongProduct} that we roughly get $s_r(\nu) \leq Q'(d)^k\alpha^k = e^{k(\log Q'(d) + \log \alpha)}$. We choose $k \asymp \log \log r^{-1}$ and therefore deduce \eqref{Outline:DetailDecay} provided that $\alpha$ is sufficiently small in terms of $d$ or equivalently $C$ is sufficiently large. This proves that $\nu$ is absolutely continuous.

\subsubsection*{\normalfont\textbf{From Decomposition on $\R^d$ to Decomposition on $G$}} It remains to explain how to establish \eqref{Outline:DecompositionGoal}, which we first translate into an analogous question on $G$. Indeed, we will make a decomposition of $q_{\tau}$ into 
\begin{equation}\label{Outline:Decomposition}
    q_{\tau} = g_1\exp(U_1) g_2\exp(U_2) \cdots g_n \exp(U_n)
\end{equation}
for random variables $g_1, \ldots , g_n$ on $G$ and $U_1, \ldots , U_n$ on the Lie algebra $\mathfrak{g}$ of $G$. In order to express $q_{\tau}v$ as a sum of random variables using \eqref{Outline:Decomposition}, we apply Taylor's theorem in Proposition~\ref{MainTaylorBound} to deduce 
\begin{equation}\label{Outline:TaylorBound}
    q_{\tau}v \approx g_1\cdots g_n v + \sum_{i = 1}^n \zeta_i(U_i),
\end{equation}
where $$\zeta_i = D_u(g_1g_2\cdots g_i \exp(u)g_{i + 1}g_{i + 2}\cdots g_n v)|_{u = 0}.$$

For notational convenience we write in this outline of proofs $$g_i' = g_1\cdots g_i \quad \text{ and } \quad g_i'' = g_{i + 1} \cdots g_n$$ and denote $$\rho_x = D_u(\exp(u)x)|_{u = 0}.$$ Then by the chain rule, as shown in Lemma~\ref{ZetaVariance}, $$\var(\zeta_i(U_i)) = \rho(g_i')^2 \, U(g_i')\var(\rho_{g_i''x}(U_i)) U(g_i')^T.$$

We will use the $(c,T)$-well-mixing and $(\alpha_0,\theta, A)$-non-degeneracy condition to ensure that 
\begin{equation}\label{Outline:VarTraceBound}
    \var(\zeta_i(U_i)) \geq  c_1 \rho(g_i')^2\tr(U_i)I = c_1 \tr(\rho(g_i')U_i)I
\end{equation}
for some constant $c_1 > 0$ depending on $d,c,T,\alpha_0,\theta$ and $A$ and where $\tr(U_i)$ is the trace of the covariance matrix of $U_i$. This will be shown in Proposition~\ref{InitialDecompositionV} by ensuring that each of the $g_i$ is a product of sufficiently many $\gamma_j$ such that we can apply well-mixing and non-degeneracy as $g_ix$ is close in distribution to $\nu$. In fact, we exploit suitable properties of the derivative of $\rho_x$ and use a principal component decomposition.

So in order to achieve \eqref{Outline:DecompositionGoal}, we require that 
\begin{equation}\label{Outline:DecompositionNewGoal}
    |U_i| \leq \rho(g_i')^{-1}r \quad\quad \text{ and } \quad\quad \sum_{i = 1}^n \tr(\rho(g_i')U_i) \geq C^3 c_1^{-1} ( \log \log r^{-1} ) r^2
\end{equation}
for the constant $C$ from \eqref{Outline:DecompositionGoal}. Note that to arrive at \eqref{Outline:DecompositionGoal} we replace $U_i$ by $C^{-1}U_i$ and use \eqref{Outline:VarTraceBound}.

\subsubsection*{\normalfont\textbf{Entropy Gap and Trace Bounds for Stopped Random Walk}} We prove \eqref{Outline:DecompositionNewGoal} by establishing suitable entropy bounds on $G$ and then translate them to the necessary trace bounds. We use the following notation. For a random variable $g$ on $G$ and $s > 0$, we define $\tr(g; s)$ to be the supremum of all $t\geq 0$ such that we can find some $\sigma$-algebra $\mathscr{A}$ and some $\mathscr{A}$-measurable random variable $h$ taking values in $G$ such that $$|\log(h^{-1}g)| \leq s \quad \text{ and } \quad \E[\tr(\log(h^{-1}g)|\mathscr{A})] \geq ts^2,$$ where $\log: G \to \mathfrak{g}$ is the Lie group logarithm and we assume that $h^{-1}g$ is supported on a small ball around the identity. The reason we need to work with the conditional trace is to use \eqref{Outline:EntropyTraceBound}. 

To establish \eqref{Outline:DecompositionNewGoal} we therefore need to find a collection of scales $s_i = \rho(g_i')^{-1}r$ such that 
\begin{equation}\label{Outline:DecompositionFinalGoal}
    \sum_{i = 1}^n \mathrm{tr}(q_{\tau}; s_i) \geq C c_1^{-1} \log \log r^{-1}
\end{equation} for $C$ an absolute constant depending only on $d$.

To show \eqref{Outline:DecompositionFinalGoal} one converts entropy estimates for $q_{\tau}$ into trace estimates, using in essence that for an absolutely continuous random variable $Z$ on $\R^{\ell}$ we have
\begin{equation}\label{Outline:RdEntropyTraceBound}
    H(Z) \leq \frac{\ell}{2}\log\left(\frac{2\pi e}{\ell} \cdot \tr(Z)\right),
\end{equation}
where $H$ is the differential entropy and $\tr(Z)$ is the trace of the covariance matrix of $Z$. Equality holds in \eqref{Outline:RdEntropyTraceBound} if and only if $Z$ is a spherical Gaussian.

We will work with entropy between scales on $G$. Precise definitions are given in section~\ref{section:vartrdef}. For the purposes of this outline consider the entropy between scales  defined for a random variable $g$ taking values in $G$, two scales $r_1,r_2 > 0$ and a parameter $a > 0$ as $$H_a(g; r_1| r_2) = (H(gs_{r_1,a}) - H(s_{r_1,a})) - (H(gs_{r_2,a}) - H(s_{r_2,a})),$$ where $H(\cdot)$ is the differential entropy and $s_{r,a}$ is a smoothing function supported on a ball of radius $ar$ and satisfying for $\ell = \dim \mathfrak{g}$ that
\begin{equation}\label{Outlinesraestimate}
    \mathrm{tr}(\log(s_{r,a})) \asymp \ell r^2 \quad \text{ and } \quad H(s_{r,a}) = \frac{\ell}{2}\log 2\pi e r^2 + O_d(e^{-a^2/4}) - O_{d,a}(r).
\end{equation}
The function $s_{r,a}$ is chosen such that $H(s_{r,a})$ is essentially maximal while being compactly supported, which is necessary towards establishing \eqref{Outline:DecompositionFinalGoal}. The parameter $a > 0$ is useful as it gives us a uniform error bound in \eqref{Outlinesraestimate}. By using moreover \eqref{Outline:RdEntropyTraceBound}, we relate in \cite{KittleKoglerEntropy}*{Proposition 1.5} entropy between scales and the trace by 
\begin{equation}\label{Outline:EntropyTraceBound}
    \tr(g; 2ar) \gg a^{-2}(H_a(g;r|2r) - O_d(e^{-a^2/4}) - O_{d,a}(r)).
\end{equation}

For $\kappa > 0$ denote by $$\tau_{\kappa} = \inf\{ n \geq 1\,:\, \rho(\gamma_1\cdots \gamma_n) \leq \kappa \}.$$ It is then shown in Proposition~\ref{EntropyBetweenScalesIncrease} for $r_1 < r_2$ and with $r_1 \leq \kappa^{\frac{S_{\mu}}{|\chi_{\mu}|}}$  that as $\kappa \to 0$ the following entropy gap holds:
\begin{equation}\label{Outline:EntropyIncrease}
    H_a(q_{\tau_{\kappa}}; r_1|r_2) \geq \left( \frac{h_{\mu}}{|\chi_{\mu}|} - d \right)\log \kappa^{-1} +  \ell \cdot \log r_2 + o_{\mu,d,a}(\log \kappa^{-1}).
\end{equation}
We will give a sketch of the proof of \eqref{Outline:EntropyIncrease} in the beginning of section~\ref{EntropyIsomSection} and just note that the main point of \eqref{Outline:EntropyIncrease} is that most of the elements in the support of $q_{\tau_{\kappa}}$ are separated by $\kappa^{\frac{S_{\mu}}{|\chi_{\mu}|}}$, which by standard properties of entropy implies that $H(q_{\tau_{\kappa}}s_{r_1,a}) \approx H(q_{\tau_{\kappa}}) + H(s_{r_1,a})$. As we have to use a stopping time in \eqref{Outline:EntropyIncrease}, we will need to work with $q_{\tau}$ instead of a deterministic time throughout our proof.

By \eqref{Outline:EntropyIncrease} it follows, assuming $h_{\mu}/|\chi_{\mu}|$ is sufficiently large and $\kappa$ is sufficiently small, that 
\begin{equation}\label{Outline:LotsofEntropy}
    H_{a}(q_{\tau_{\kappa}}; \kappa^{\frac{S_{\mu}}{|\chi_{\mu}|}}| \kappa^{\frac{h_{\mu}}{2\ell|\chi_{\mu}|}}) \gg_d \frac{h_{\mu}}{|\chi_{\mu}|} \log \kappa^{-1}.
\end{equation}

Using \eqref{Outline:LotsofEntropy} and \eqref{Outline:EntropyTraceBound}, we show in Proposition~\ref{LotsOfTrace} with setting $S = 2\max\{ S_{\mu}, h_{\mu} \}$ that for a collection of scales $$s_i \in (\kappa^{\frac{S}{|\chi_{\mu}|}},\kappa^{\frac{h_{\mu}}{2\ell|\chi_{\mu}|}}) \quad\quad \text{ with } \quad\quad 1 \leq i \leq \hat{m}$$ and $\hat{m}$ being a fixed constant depending on $S_{\mu}$ and $\chi_{\mu}$ that 
\begin{equation}\label{Outline:LotsofTrace}
    \sum_{i = 1}^{\hat{m}} \mathrm{tr}(q_{\tau_{\kappa}}; s_i) \gg_d \frac{h_{\mu}}{|\chi_{\mu}|}  \max\left\{ 1,\log \frac{S_{\mu}}{h_{\mu}}  \right\}^{-1}.
\end{equation}
As we explain at the beginning of section~\ref{EntropyIsomSection}, the error term $\max\left\{ 1,\log \frac{S_{\mu}}{h_{\mu}}  \right\}^{-1}$ arises from the error $O_d(e^{-a^2/4})$ in \eqref{Outline:EntropyTraceBound}.

\subsubsection*{\normalfont\textbf{Conclusion of Proof}} The trace bound \eqref{Outline:LotsofTrace} is not sufficient to establish \eqref{Outline:DecompositionFinalGoal} as we require a lower bound depending on $\log \log r^{-1}$. To achieve such a bound and to conclude the proof, we concatenate several decompositions arising from \eqref{Outline:LotsofTrace} and therefore develop a suitable theory of such decompositions in section~\ref{section:decomposition}.

It therefore remains to find sufficiently many parameters $\kappa_1, \ldots, \kappa_m$ such that the resulting intervals $$(\kappa_1^{\frac{S}{|\chi_{\mu}|}},\kappa_1^{\frac{h_{\mu}}{2\ell|\chi_{\mu}|}}), \quad (\kappa_2^{\frac{S}{|\chi_{\mu}|}},\kappa_2^{\frac{h_{\mu}}{2\ell|\chi_{\mu}|}}), \quad \ldots \quad (\kappa_m^{\frac{S}{|\chi_{\mu}|}},\kappa_m^{\frac{h_{\mu}}{2\ell|\chi_{\mu}|}})$$ are disjoint. As we require that all of the scales are $\geq r$, we set $\kappa_1 = r^{\frac{|\chi_{\mu}|}{S}}$. On the other hand, we want all scales to be sufficiently small. We, for example, therefore require that $\kappa_m^{\frac{h_{\mu}}{2\ell|\chi_{\mu}|}} < e^{-10}$. Thus setting $\kappa_{i + 1} = \kappa_i^{\frac{h_{\mu}}{3\ell S}}$, thereby ensuring that the resulting intervals are disjoint (provided $h_{\mu}/\chi_{\mu}$ is sufficiently large), a calculation shows that the maximal $m$ we can take is $$ m \asymp \max\left\{ 1,\log \frac{S_{\mu}}{h_{\mu}}  \right\}^{-1}\log \log r^{-1} .$$ Combining all of the above, it follows that when summing over all the scales $$\sum_{i} \mathrm{tr}(q_{\tau_{\kappa_1}}; s_i) \gg_d \frac{h_{\mu}}{|\chi_{\mu}|}  \max\left\{ 1,\log \frac{S_{\mu}}{h_{\mu}}  \right\}^{-2} \log \log r^{-1}.$$ We therefore require in order to satisfy \eqref{Outline:DecompositionFinalGoal} that $$\frac{h_{\mu}}{|\chi_{\mu}|}  \max\left\{ 1,\log \frac{S_{\mu}}{h_{\mu}}  \right\}^{-2}  \geq C^3c_1^{-1},$$ which leads us to the condition from Theorem~\ref{MainResult} and concludes our sketch of the proof.

\subsection{Notation} \label{section:Notation} We use the asymptotic notation $A \ll B$ or $A = O(B)$ to denote that $|A| \leq CB$ for a constant $C > 0$. If the constant $C$ depends on additional parameters we add subscripts. Moreover, $A \asymp B$ denotes $A \ll B$ and $B \ll A$. 

For an integer $n \geq 1$ we abbreviate $[n] = \{ 1,2,\ldots , n \}$. On $\R^d$ the euclidean norm is denoted $|\cdot|$.

Given two positive semi-definite symmetric real $d \times d$ matrices $M_1$ and $M_2$ we write 
\begin{equation}\label{MatrixPartialOrder}
    M_1 \geq M_2 \quad\quad \text{ if and only if } \quad\quad x^TM_1 x \geq x^TM_2 x \quad \text{ for all } x \in \R^d.
\end{equation}

For a random variable $X$ on $\R^d$ we denote by $\mathrm{Var}(X)$ the covariance matrix of $X$ and by $\mathrm{tr}(X) = \tr\, \mathrm{Var}(X)$ the trace of the covariance matrix.

Given a metric space $(M,d)$, $p \in [1,\infty)$ and two probability measures $\lambda_1$ and $\lambda_2$ on $M$, we define the $L^p$-Wasserstein metric as
\begin{equation}\label{WassersteinDef}
    \mathcal{W}_p(\lambda_1, \lambda_2) = \inf_{\gamma \in \Gamma(\lambda_1,\lambda_2)} \left( \int_{M\times M} d(x,y)^p \, d\gamma(x,y)\right)^{\frac{1}{p}},
\end{equation}
where $\Gamma(\lambda_1,\lambda_2)$ is the set of couplings of $\lambda_1$ and $\lambda_2$, i.e. of probability measures $\gamma$ on $M \times M$ whose projections to the first coordinate is $\lambda_1$ and to the second is $\lambda_2$.

Throughout this paper we fix $d \geq 1$ and write $G = \mathrm{Sim}(\R^d)$. The Lie algebra of $G$ will be denoted $\mathfrak{g}$ and $\ell = \dim \mathfrak{g}$. We usually consider a fixed probability measure $\mu$ on $G$ and independent samples $\gamma_1, \gamma_2, \ldots$ of $\mu$.  We write for $\kappa > 0$ $$q_n = \gamma_1\cdots \gamma_n \quad\quad  \text{ and } \quad\quad  \tau_{\kappa} = \inf\{ n \geq 1 \,;\, \rho(\gamma_n) \leq \kappa \}.$$

When $\mu$ is a probability measure on $G = \mathrm{Sim}(\R^d)$ and $\nu$ is a probability measure $\R^d$ we denote by $\mu * \nu$ the probability measure uniquely characterized by $$(\mu*\nu)(f) = \int\int f(gx) \, d\mu(g)d\nu(x)$$ for $f\in C_c(\R^d)$. When $\mu = \sum_{i}p_i \delta_{g_i}$ is finitely supported, then 
\begin{equation}\label{Equation:ConcreteConvolution}
    \mu * \nu = \sum_{i} p_i g_i\nu,
\end{equation}
where $g_i\nu$ is the pushforward of $\nu$ by $g_i$ defined by $(g_i\nu)(B) = \nu(g_i^{-1}B)$ for all Borel sets $B \subset \R^d$.

The various notions of entropy between scales as well as $\tr(g,r)$ are the same as in \cite{KittleKoglerEntropy} and will be recalled in section~\ref{section:vartrdef}. 

We will denote by $\Haarof{G}$ a normalised Haar measure on $\mathrm{Sim}(\R^d)$. Moreover if $H\subset O(d)$ is a closed subgroup, we will denote by $\Haarof{H}$ the Haar probability measure on $H$. For a probability measure $\mu_U$ on $H$, the $L^2$-spectral gap of $\mu_U$ in $H$ is defined as 
\begin{equation}\label{SpectralGapDef}
    \mathrm{gap}_H(\mu_U) = 1 - ||T_{\mu_U}|_{L^2_0(G)}||,
\end{equation}
where $(T_{\mu_U}f)(k) = \int f(hk) \, d\mu_U(h)$ for $f \in L^2(H)$ and $L^2_0(H) = \{ f \in L^2(H) \,:\, \Haarof{H}(f) = 0\}$ for $||\circ||$ the operator norm.

\subsection{Organisation} \label{section:organisation} 

In section~\ref{section:prelim} the Taylor expansion bound \eqref{Outline:TaylorBound} is proved and we establish several probabilistic preliminaries. We discuss order $k$ detail in section~\ref{section:Detail}, establish  \eqref{Outline:GoodDetailBound} as well as show how to convert \eqref{Outline:DecompositionGoal} into suitable detail bounds. In section~\ref{EntropyIsomSection} we prove \eqref{Outline:EntropyIncrease} and \eqref{Outline:LotsofTrace}. Finally, we deduce Theorem~\ref{MainResult} as well as Theorem~\ref{HighdimMainTheorem} in section~\ref{section:decomposition} by developing a decomposition theory for stopped random walks. We study $(c,T)$-well-mixing and $(\alpha_0,\theta,A)$-non-degeneracy in section~\ref{section:mixnondeg} and prove Proposition~\ref{UnifNonDeg} and Proposition~\ref{UnifNonDegContAv}. In section~\ref{section:examples} we establish explicit examples and in particular we prove all the corollaries stated in the introduction.

    \section{Preliminaries}

    \label{section:prelim}

In this section we first study the derivatives of the $G$ action on $\R^d$ in section~\ref{subsection:DerivativeBounds} and then versions of the large deviation principle in section~\ref{subsection:LDP}.

\subsection{Derivative Bounds}

\label{subsection:DerivativeBounds}

\subsubsection{Basic Properties}

Let $G = \mathrm{Sim}(\R^d)$ with Lie algebra $\mathfrak{g} = \mathrm{Lie}(G)$. For $x \in \R^d$ consider the map $$ w_x: \mathfrak{g} \to \R^d, \quad\quad u \mapsto \exp(u)x.$$ Denote by $\psi_x = D_0w_x : \mathfrak{g} \to \R^d$ the differential at zero of $w_x$. 

Note that we can embed $G = \mathrm{Sim}(\R^d)$ into $\mathrm{GL}_{d + 1}(\R)$ via the map $$g \mapsto \begin{pmatrix}
    r(g)U(g) & b(g) \\ 0 & 1
\end{pmatrix}.$$ We can therefore identify $\mathfrak{g}$ as a matrix Lie algebra and so can write $$\mathfrak{g} = \left\{ \begin{pmatrix}
        \alpha & \beta \\ 0 & 0
    \end{pmatrix}
    \,:\, \alpha\in \R I +   \mathfrak{so}_d(\R), \beta \in \R^d \right\} \subset \mathfrak{gl}_{n + 1}(\R)$$  Thus for $u = (\begin{smallmatrix}
        \alpha & \beta \\ 0 & 0
    \end{smallmatrix})$ it follows that $\psi_x(u) = u(\begin{smallmatrix}
    x \\ 1
\end{smallmatrix}) = \alpha x + \beta$. With this viewpoint we also use the following convenient notation
\begin{equation}\label{NotationConvention}
    ux = \psi_x(u) = \alpha x + \beta
\end{equation}

We fix an inner product on $\mathfrak{g}$ and denote by $|\circ|$ the associated norm. Moreover, we choose an ordered orthonormal basis of $\mathfrak{g}$, endowing $\mathfrak{g}$ with a coordinate system. So every element $u \in \mathfrak{g}$ can be written as a sum $u = \sum_{i = 1}^{\ell} u_i$, where $u_i$ is the projection of $u$ to the $i$-th basis vector. On numerous occasions we will consider derivatives with respect to $u_i$.

In the following lemma, some properties about the derivatives of $w_x, \psi_x$ and the map $g$ are collected. For notational convenience, we denote throughout this subsection by $\frac{\partial f}{\partial x}$ the derivative $D_x f$ of a function $f: \R^{d_1} \to \R^{d_2}$ at a vector $x \in \R^{d_1}$. We furthermore write $\ell = \dim \mathfrak{g}$.

\begin{lemma}\label{BasicDerivativeProp}
    The following properties hold:
    \begin{enumerate}[label = (\roman*)]
        \item Let $g = \rho U + b \in G$. Then for all $x \in \R^d$, it holds that $\frac{\partial g}{\partial x} = \rho U$ and all of the second derivatives of $g$ are zero.
        \item Whenever $x \in \R^d$ and $|u| \leq 1$ and $1 \leq i,j \leq \ell$, $$\bigg| \frac{\partial w_x}{\partial u_i} \bigg| \ll_d \max(|x|,1) \quad \text{ and }\quad \bigg| \frac{\partial w_x}{\partial u_i \partial u_j}\bigg| \ll_d \max(|x|,1) .$$
        \item For any $x_1, x_2 \in \R^d$ we have that $$||\psi_{x_1} - \psi_{x_2}|| \ll_d |x_1 - x_2|.$$
        \item Let $u \in \mathfrak{g} \backslash \{ 0 \}$. Then there is a proper subspace $W_u \subset \R^d$ and a vector $u_0 \in \R^d$ such that if $\psi_x(u) = 0$ then $x \in u_0 + W_u$ for $x \in \R^d$.
        \item For all $\theta, A > 0$ there is $\delta > 0$ such that the following is true. Let $v \in \mathfrak{g}$ be a unit vector. Then there is a proper subspace $W_v \subset \R^d$ and a vector $v_0 \in \R^d$ such that if $$x \in \R^d \backslash B_{\theta}(v_0 + W_v) \quad\text{ and } \quad |x| \leq A$$ for $B_{\theta}(v_0 + W_v)$ the $\theta$-ball around $v_0 + W_v$ then $$|\psi_x(v)| \geq \delta.$$
    \end{enumerate}
\end{lemma}

\begin{proof}
    (i) follows by definition and (ii) by compactness of $\{u\in \mathfrak{g} \,:\, |u| \leq 1 \}$ and using that a pure translation by a small vector has norm $O_d(1)$. For (iii) using notation \eqref{NotationConvention} it holds for $u = (\begin{smallmatrix}
        \alpha & \beta \\ 0 & 0
    \end{smallmatrix}) \in \mathfrak{g}$ with $|u| \leq 1$ that 
    \begin{align*}
        |\psi_{x_1}(u) - \psi_{x_2}(u)| = |\alpha x_1 - \alpha x_2| &\leq ||\alpha||\cdot |x_1 - x_2| \\ &\ll_d |\alpha| \cdot |x_1 - x_2| \leq |u|\cdot |x_1 - x_2|
    \end{align*}
    using that the operator norm $||\circ||$ is equivalent to the inner product norm on $\mathfrak{g}$. To show (iv),  we may assume that $\beta \in \mathrm{Im}(\alpha)$ as otherwise there is nothing to show. Then set $W_u = \ker(\alpha)$ and $u_0 \in \R^d$ such that $\alpha u_0 = -\beta$, implying the claim.  (v) follows from (iv) by continuity.
\end{proof}

For $u \in \mathfrak{g} \backslash \{ 0 \}$ we define   $$E_{\theta}(u) = \R^d \backslash B_{\theta}(u_0 + W_u).$$ 

Given a random variable $U$ taking values in $\mathfrak{g}$, we say that $u \in \mathfrak{g}$ is a first principal component if it is an eigenvector of its covariance matrix with maximal eigenvalue. Set $$E_{\theta}(U) = \bigcup_{v \in P} E_{\theta}(v),$$ where $P$ is the set of first principal components of $U$. Similarly if $\mu$ is a probability measure which is the law of a random variable $U$ then we define $E_{\theta}(\mu) = E_{\theta}(U)$. Recall that given a random variable $U$ in $\R^{\ell}$, we denote by $\mathrm{tr}(U)$ the trace of the covariance matrix of $U$.

\begin{proposition}\label{FirstDerivativeBound}
    For all $\theta,A > 0$ there is some $\delta = \delta(d,\theta,A) > 0$ such that the following is true. Suppose that $U$ is a random variable taking values in $\mathfrak{g}$ and that $x \in \R^d$ with $|x| \leq A$. Suppose that $x \in E_\theta(U)$. Then $$\tr(Ux) \geq \delta \cdot \tr(U).$$
\end{proposition}

\begin{proof}
    We applied here the notation \eqref{NotationConvention} that $\psi_x(U) = Ux$. Write $\ell = \dim \mathfrak{g}$ and let $w_1, \ldots , w_{\ell}$ be an orthonormal basis of eigenvectors of the covariance matrix $\mathrm{Var}(U)$. We may assume that $U$ has mean zero. Denote by $U_i = \langle U, w_i \rangle = U^Tw_i$ for $1 \leq i \leq \ell$ and assume without loss of generality that $\mathrm{Var}(U_1) \geq \ldots \geq \mathrm{Var}(U_{\ell})$ so that $w_1$ is a principal component. Then the $(U_i)_{1 \leq i \leq \ell}$ are uncorrelated since for $i \neq j$ 
    \begin{align*}
       \mathrm{cov}(U_i,U_j) = \E[U_i U_j] &= \E[\langle U^T w_i, U^T w_j \rangle] \\ &= \E[ \langle UU^T w_i, w_j \rangle ] =  \langle \mathrm{Var}(U)w_i, w_j \rangle = 0 
    \end{align*}
    and it holds that $U = \sum_{i = 1}^{\ell} U_i w_i$ and that $\mathrm{Var}(U_1) \geq \frac{1}{\ell}\mathrm{tr}(U)$. Also by Proposition~\ref{BasicDerivativeProp} (v) it holds that $|\psi_x(w_1)| \geq \delta$. We then compute 
    \begin{align*}
        \tr(\rho_x(U)) = \E[|\rho_x(U)|^2] = \E\left[ \sum_{i = 1}^{\ell} U_i^2 |\rho_x(w_i)|^2  \right] \geq \E[U_1^2 |\rho_x(w_1)|^2] \geq \frac{\delta}{\ell} \tr(U).
    \end{align*}
\end{proof}

\begin{lemma}\label{ZetaVariance}
    Let $U$ be a random variable on $\mathfrak{g}$ and let $g\in G$ and $x \in \R^d$. Denote $$\zeta = D_u g\exp(u)x|_{u = 0}.$$ Then $$\var(\zeta(U)) = \rho(g)^2 \cdot U(g) \psi_x \circ \mathrm{Var}(U) \circ \psi_x^T U(g)^T.$$
\end{lemma}

\begin{proof}
    Note that by the chain rule $\zeta(U) = \rho(g)U(g)\psi_x(U)$ and therefore $$\var \zeta(U) = \rho(g)^2 U(g)\var(\psi_x(U))U(g)^T$$ Viewing $\psi_x: \mathfrak{g} \to \R^d$ as a matrix with our choice of coordinate system we write $\psi_x(U) = \psi_x \circ U$ and the claim follows.
\end{proof}

\subsubsection{Taylor Expansion Bound} The aim of this subsection is to prove the following proposition, which crucially relies on the $G$ action on $\R^d$ having vanishing second derivatives. 

\begin{proposition}\label{MainTaylorBound}
    For every $A > 0$ there exists $C = C(d,A) > 1$ such that the following holds. Let $n \geq 1$, $r \in (0,1)$ and let $u^{(1)}, \ldots , u^{(n)} \in \mathfrak{g}$. Let $g_1, \ldots , g_n \in G$ with $$\rho(g_i) < 1, \quad \quad |b(g_i)| \leq A \quad \text{ and } \quad |u^{(i)}| \leq \rho(g_1 \cdots g_i)^{-1}r < 1.$$ Let $v \in \R^d$ with $|v| \leq A$ and write $$x = g_1 \exp(u^{(1)})\cdots g_n\exp(u^{(n)})v$$ and $$\zeta_i = D_0(g_1g_2 \cdots g_{i}\exp(u)g_{i + 1} \cdots g_{n-1}g_n v)$$ and let $$S = g_1\cdots g_n v + \sum_{i = 1}^n \zeta_i(u^{(i)}).$$ Then it holds that $$|x-S| \leq C^n\rho(g_1 \cdots g_n)^{-1}r^2.$$
\end{proposition}

To prove Proposition~\ref{MainTaylorBound} we use the following version of Taylor's theorem.

\begin{theorem}\label{Taylor}
    Let $f: \R^n \to \R$ be a $C^2$-function, let $R_1, \ldots , R_n > 0$ and write $B = [-R_1, R_1] \times \ldots \times [-R_n, R_n]$. For integers $i,j \in [1,n]$ let $K_{ij} = \sup_B |\frac{\partial^2 f}{\partial x_i \partial x_j}|$ and let $x\in B$. Then we have that $$\bigg| f(x) - f(0) - \sum_{i = 1}^n x_i \frac{\partial f}{\partial x_i}\bigg|_{x = 0}  \bigg| \leq \frac{1}{2}\sum_{i,j = 1}^n  K_{i,j}\, |x_i|\,|x_j|.$$
\end{theorem}

\begin{lemma}\label{wSecondDerivative}
    Let $$w : \mathfrak{g} \times \mathfrak{g} \longrightarrow \R^d, \quad\quad (x,y) \longmapsto \exp(x)g\exp(y)v $$ for fixed $g,v$. Then if $|x|, |y| \leq 1$ it holds that $$\bigg|\frac{\partial w(x,y)}{\partial x_i \partial y_i}\bigg| \ll_d \rho(g)\max(|v|,1).$$
\end{lemma}

\begin{proof}
    Let $\hat{v} = \exp(y)v$ and by Lemma~\ref{BasicDerivativeProp} (ii), $| \frac{\partial \hat{v}}{\partial y_i} | \ll_d \max(|v|,1).$ Now let $\tilde{v}= g\hat{v}$. Therefore, by Lemma~\ref{BasicDerivativeProp} (i), $||\frac{\partial \tilde{v}}{\partial\hat{v}}|| \leq \rho(g)$ and moreover, since $w = \exp(x)\tilde{v}$ and $|x| \leq 1$, it is readily shown that $||\frac{\partial^2 w}{\partial x_i \partial \tilde{v}}|| \ll_d 1$. We conclude therefore by the chain rule
    $$\bigg|\frac{\partial w}{\partial x_i \partial y_i}\bigg| = \bigg|\bigg|\frac{\partial w}{\partial x_i \partial \tilde{v}}\bigg|\bigg|\cdot \bigg|\bigg|\frac{\partial \tilde{v}}{\partial\hat{v}}\bigg|\bigg| \cdot \bigg| \frac{\partial \hat{v}}{\partial y_i} \bigg| \ll_d \rho(g)\max(|v|,1).$$
\end{proof}

\begin{proposition}\label{DerivativeBound}
      There exists a constants $C = C(d) > 1$ such that the following holds. Suppose that $n \in \Z_{>0}$, $g_1, g_2, \ldots , g_n \in G$ and let $u^{(1)}, \ldots , u^{(n)} \in \mathfrak{g}$ be such that $|u^{(i)}| \leq 1$.

      Let $v\in \R^d$ and $$x = g_1 \exp(u^{(1)})g_2\exp(u^{(2)}) \cdots g_n \exp(u^{(n)})v.$$ Then for any $1 \leq i,j \leq \ell$ and any integers $k,m \in [1,n]$ with $k\leq m$ we have $$\bigg|\frac{\partial^2x}{\partial u_i^{(k)} \partial u_j^{(m)}} \bigg| \leq C^n\rho(g_1\cdots g_{m})\max(|g_{m + 1}\exp(u^{(m + 1)}) \cdots g_{n}\exp(u^{(n)}) v|,1).$$
\end{proposition}

\begin{proof}
    First, we deal with the case $k = m$. Let $$a = g_1 \exp(u^{(1)})g_2\exp(u^{(2)}) \cdots g_{k-1} \exp(u^{(k-1)})g_k$$ and $$b = g_{k + 1} \exp(u^{(k+1)})g_{k + 2}\exp(u^{(k + 2)}) \cdots g_{n} \exp(u^{(n)})v$$ and let $\tilde{b} = \exp(u^{(k)})b$. We have $$\frac{\partial x}{\partial u_i^{(k)}} = \frac{\partial x}{\partial \tilde{b}}  \frac{\partial \tilde{b}}{\partial u_i^{(k)}}.$$ Note that by Lemma~\ref{BasicDerivativeProp} (i) all of the second derivatives of $x$ with respect to $\tilde{b}$ are zero and therefore
    \begin{equation}\label{secondderivative}
        \bigg|\frac{\partial^2 x}{\partial u_i^{(k)}\partial u_j^{(k)}}\bigg|  \leq \bigg| \bigg| \frac{\partial x}{\partial \tilde{b}}\bigg|\bigg| \cdot \bigg| \frac{\partial^2 \tilde{b}}  {\partial u_i^{(k)}\partial u_j^{(k)}}\bigg| .
    \end{equation}
    Thus by Lemma~\ref{BasicDerivativeProp} (i) and (ii) we conclude that $$\bigg|\frac{\partial^2 x}{\partial u_i^{(k)}\partial u_j^{(k)}}\bigg| \ll_d \rho(a)\max(|b|,1) \leq C^n\rho(g_1\cdots g_{\ell})\max(|b|,1)$$ for a suitable constant $C > 1$ using that $\rho(\exp(u^{(i)}))$ is bounded. 

    For the case $k < m$ we consider  
    \begin{align*}
        a_1 &= g_1 \exp(u^{(1)})g_2\exp(u^{(2)}) \cdots g_{k-1} \exp(u^{(k-1)})g_k \\
        a_2 &= g_{k + 1} \exp(u^{(k+1)})g_{k + 2}\exp(u^{(k + 2)}) \cdots g_{m}  \\
        b &= g_{m + 1} \exp(u^{(m+1)})g_{m + 2}\exp(u^{(m + 2)}) \cdots g_{n} \exp(u^{(n)})v.
    \end{align*} Then we consider $\tilde{b} = \exp(u^{(k)})a_2\exp(u^{(m)})b$ and as before we conclude $$\frac{\partial^2 x}{\partial u_i^{(k)}\partial u_j^{(m)}} = \frac{\partial x}{\partial \tilde{b}}\frac{\partial^2 \tilde{b}}{\partial u_i^{(k)}\partial u_j^{(m)}}.$$ We again arrive at \eqref{secondderivative} and deduce the claim as in the case $k = m$ using Lemma~\ref{wSecondDerivative} instead of Lemma~\ref{BasicDerivativeProp} (i).
\end{proof}

\begin{proof}(of Proposition~\ref{MainTaylorBound})
    We first show that there is a constant $C_1 = C_1(A,d)$ depending on $A$ such that for all $1 \leq i \leq n$ we have that 
    \begin{equation}\label{ProdBound}
        |g_i\exp(u^{(i)})\cdots g_n\exp(u^{(n)})v| \leq C_1^{n-i + 1}.
    \end{equation}
    Indeed, we note that for any $u \in \mathfrak{g}$ with $|u| \leq 1$ and $v_0 \in \R^d$ it holds that $|\exp(u)v_0 - v_0| \leq C_2(|v_0| + 1)$ for an absolute constant $C_2 = C_2(d)$. Without loss of generality we assume that $C_2(d) > 1$. Therefore $|\exp(u^{(n)})v| \leq C_2 (2|v| + 1)$. Next note that as $\rho(g_n) < 1$, 
    \begin{align*}
        |g_n\exp(u^{(n)})v| &\leq |g_n\exp(u^{(n)})v - g_n(0)| + |g_n(0)| \\ &\leq \rho(g_n)|\exp(u^{(n)})v| + |b(g_n)| \\   & \leq C_2(2|v| + |b(g_n)| + 1)  \leq 4C_2(A + 1),
    \end{align*} using that $\rho(g_n) < 1$ and that $|v| \leq A$ and $|b(g_n)| \leq A$. Continuing this argument inductively, we may conclude that $$|g_i\exp(u^{(i)})\cdots g_n\exp(u^{(n)})v| \leq 4^{n-i + 1}C_2^{n-i + 1}(A + (n-i) + 1),$$ which implies \eqref{ProdBound}.
    
	Note that by the assumptions $$\rho(g_1\cdots g_{\ell}) |u^{(\ell)}|^2 \leq \rho(g_1\cdots g_{\ell})^{-1}r^2 \leq \rho(g_1\cdots g_n)^{-1}r^2.$$ Therefore, by applying Theorem~\ref{Taylor} together with Proposition~\ref{DerivativeBound} and \eqref{ProdBound} for a sufficiently large constant $C$ depending on $A$ and $d$ in each of the coordinates of $\R^d$, $$|x - S| \leq dn^2C^n \rho(g_1\cdots g_n)^{-1} r^2,$$ which implies the claim upon enlarging the constant $C$.
\end{proof}

\subsection{Large Deviation Principle} \label{subsection:LDP}

In this subsection we review various versions of the large deviation principle. Throughout this section we denote by $\mu$ a measure on $G$ and by $\gamma_1, \gamma_2, \ldots$ independent samples from $\mu$. Applying the classical large deviation principle to $\rho$, we can state the following.

\begin{lemma}\label{rhoLDP} Let $\mu$ be a compactly supported, contracting on average probability measure on $G$. Then for every $\eps > 0$ there is $\delta = \delta(\mu,\eps) > 0$ such that for all sufficiently large $n$,
\begin{equation*}
    \mathbb{P}\Big[ \, | n\chi_{\mu} - \log \rho(\gamma_1) \cdots \rho(\gamma_n) | > \eps n \,\Big] \leq e^{-\delta n}.
\end{equation*}
\end{lemma}

We generalise Lemma~\ref{rhoLDP} to stopping times. 

\begin{lemma}\label{rhotauLDP}
    Let $\mu$ be a compactly supported contracting on average probability measure on $G$ and let $\kappa > 0$ and denote $$\tau_{\kappa} = \inf\{ n \geq 1 \,:\, \rho(\gamma_1 \ldots \gamma_n) \leq \kappa \}.$$ Then for every $\eps > 0$ there is $\delta > 0$ such that for sufficiently small $\kappa$ $$\mathbb{P}\Big[ \Big|\tau_{\kappa} -  \frac{\log \kappa^{-1}}{|\chi_{\mu}|} \Big| > \eps \log \kappa^{-1} \Big] \leq e^{-\delta \log \kappa^{-1}}$$
\end{lemma}

\begin{proof}
    If $\tau_{\kappa} > \frac{\log \kappa^{-1}}{|\chi_{\mu}|} + \eps \log \kappa^{-1}$ then $$\rho(\gamma_1 \cdots \gamma_{\lfloor \frac{\log \kappa^{-1}}{|\chi_{\mu}|} + \eps \log \kappa^{-1} \rfloor}) \geq \kappa,$$ which by Lemma~\ref{rhoLDP} has probability at most $e^{-\delta \log \kappa^{-1}}$ for some $\delta > 0$ and sufficiently small $\kappa$.

    Write $R = \inf\{ \rho(g) \,:\, g \in \mathrm{supp}(\mu) \} \in (0,1)$, which is non-zero since $\mu$ is compactly supported. Therefore when $\tau_{\kappa} < \frac{\log \kappa^{-1}}{|\chi_{\mu}|} - \eps \log \kappa^{-1}$ happens there must be some integer $$k \in \left[\frac{\log \kappa^{-1}}{|\log R|},  \frac{\log \kappa^{-1}}{|\chi_{\mu}|} - \eps \log \kappa^{-1} \right]$$ such that $$\log \rho(\gamma_1 \cdots \gamma_k) \leq \log \kappa.$$ Note that for sufficiently small $\kappa$ we have $k |\chi_{\mu}| \leq \log \kappa^{-1} - \eps |\chi_{\mu}| \, |\log R|$ and therefore 
    \begin{equation}\label{ApplyLDP}
        \log \rho(\gamma_1 \cdots \gamma_k) \leq \log \kappa \leq k(\chi_{\mu}  + \eps |\log R| \chi_{\mu}).
    \end{equation}
    By Lemma~\ref{rhoLDP} the probability that \eqref{ApplyLDP} happens is $\leq e^{-\delta' k} = e^{-\delta' O_{\mu}(\log \kappa^{-1})}$ for some $\delta' > 0$. Since there are at most $O_{\mu}(\log \kappa^{-1})$ many possibilities for $k$, the claim follows by the union bound. 
\end{proof}

From Lemma~\ref{rhoLDP} and \eqref{PolynomialTailDecay} we can deduce the following corollary.

\begin{corollary} \label{coro:decay_estimates}
    Let $\mu$ be a contracting on average probability measure on $G$. Then for every $\eps > 0$ there is $\delta = \delta(\mu,\eps) > 0$ such that for all sufficiently large $N$
    \begin{equation}\label{rhoLPDBorelCantelli}
        \mathbb{P}\Big[ \, \exists n \geq N : \rho(\gamma_1 \cdots \gamma_n) \geq \exp( (\chi_{\mu} + \eps) n) \Big] \leq e^{-\delta N}
    \end{equation}
    and
    \begin{equation*}\label{bdiff}
        \mathbb{P}\Big[ \, \exists n, m \geq N : |b(\gamma_1 \cdots \gamma_n) - b(\gamma_1 \cdots \gamma_m)| \geq \exp( (\chi_{\mu} + \eps) \min(m, n)) \Big] \leq e^{-\delta N}.
    \end{equation*}
\end{corollary}

\begin{proof}
    Equation \eqref{rhoLPDBorelCantelli} follows from Lemma~\ref{rhoLDP} and Borel-Cantelli. For \eqref{bdiff} note that when $m \geq n + 1$, $$|b(\gamma_1 \cdots \gamma_n) - b(\gamma_1 \cdots \gamma_m)| \leq  \rho(\gamma_1\cdots \gamma_n)|b(\gamma_{n + 1}\cdots \gamma_m)|.$$ Therefore by \eqref{rhoLPDBorelCantelli} it suffices to show that for sufficiently large $N$ we have that $$\mathbb{P}[\exists k \geq 1 \,:\, |b(\gamma_1\cdots \gamma_k)| \geq e^{\eps N} ] \leq e^{-\delta N},$$ which readily follows from \eqref{PolynomialTailDecay} and Borel-Cantelli as $b(\gamma_1\cdots \gamma_k)$ converges exponentially fast in distribution to $\nu$.
\end{proof}

The next lemma was proved in \cite{Kittle2023}.

\begin{lemma}(Corollary 7.9 of \cite{Kittle2023})\label{EffectiveCramerR}
    There is a constant $c > 0$ such that the following is true for all $a \in [0,1)$ and $n\geq 1$. Let $X_1, \ldots , X_n$ be random variables taking values in $[0,1]$ and  let $m_1, \ldots , m_n \geq 0$ be such that we have almost surely
    $\E[X_i | X_1 , \ldots , X_{i-1}] \geq m_i$ for $1 \leq i \leq n$. Suppose that $\sum_{i = 1}^n m_i = an$. Then $$\log \mathbb{P}\left[X_1 + \ldots + X_n \leq \frac{1}{2}na\right] \leq - cna.$$
\end{lemma}

We generalise Lemma~\ref{EffectiveCramerR} to higher dimensions.

\begin{lemma}\label{lemma:cramer}
    There is some absolute constant $c>0$ such that the following is true. Suppose that $X_1, \dots, X_n$ are random $d \times d$ symmetric positive semi-definite matrices such that $X_i \leq b I$ for some $b > 0$ and
    $$\mathbb{E}[X_i | X_1, \dots, X_{i-1}] \geq m_i I.$$
    Suppose that $\sum_{i=1}^{n}m_i = an$. Then there is some constant $C = C(a/b, d)$ depending only on $a/b$ and $d$ such that $$\mathbb{P}\left[X_1 + \dots + X_n > \frac{na}{4}I\right] \geq 1 - Ce^{-can}$$
\end{lemma}

Here we are using the partial ordering \eqref{MatrixPartialOrder}.

\begin{proof}
    For convenience write $Y_n = X_1 + \ldots + X_n$ and choose a set $S$ of unit vectors in $\R^d$ such that if $y$ is any unit vector in $\R^d$ then there exists $x \in S$ with $\| x-y \| \leq \frac{a}{8b}$. Note that the size of $S$ depends only on $d$ and $a/b$.

    By Lemma~\ref{EffectiveCramerR} we know that for any $x \in S$, $$\log \mathbb{P}\left[x^{T}Y_n x \leq \frac{na}{2}\right] \leq - can.$$ Let $A$ be the event that there exists some $x \in S$ with $x^{T}Y_n x \leq \frac{na}{2}$. We have that $\log \mathbb{P}[A]$ is at most $-can + \log |S|$. It suffices therefore to show that on $A^C$ we have $Y_n > \frac{na}{4}I$.

    Indeed let $y \in \R^d$ be a unit vector. Choose some $x \in \R^d$ with $\| x-y\| \leq a /8 b$. Suppose that $A^C$ occurs. Note that we must have $Y_n \leq bnI$ and therefore $||Y_n|| \leq bn$. This means
    \begin{align*}
        y^TY_n y & = x^TY_n x + x^TY_n (y-x) + (y-x)^TY_n y\\
        & > \frac{an}{2} - 2bn \cdot \frac{a}{8b}  = \frac{an}{4}.
    \end{align*}
    and result follows.
\end{proof}

\subsection{$\mathcal{I}_k$ is open}

Recall that we defined for every $k \geq 2$  $$\mathcal{I}_{k} = \{ (U_1, \ldots , U_{k}) \in O(d)^{k} \,:\, \text{the group generated by } U_1,\ldots, U_{k} \text{ is irreducible}  \}.$$ We provide a brief proof of the claim mentioned in the introduction that $\mathcal{I}_k$ is an open subset of $O(d)^{k}$.

\begin{lemma}\label{I_kOpen}
    The set $\mathcal{I}_k$ is an open subset of $O(d)^k$.
\end{lemma}

\begin{proof}
    We show that the complement $O(d)^k \backslash \mathcal{I}_k$ is closed. The complement $O(d)^k \setminus \mathcal{I}_k$ is the set of all $k$-tuples $(U_1, \dots, U_k)$ where the generated group $\langle U_1, \dots, U_k \rangle$ is reducible, i.e. there exists a non-trivial, proper subspace $V \subset \mathbb{R}^d$ such that $U_iV \subseteq V$ for all $i = 1, \dots, k$. Let $P$ be the projection matrix onto a subspace $V$. The condition $U_iV \subseteq V$ is equivalent to the matrix condition $U_i P = P U_i$. 
    
    The set of all subspaces, or equivalently projection matrices, with a fixed rank $\ell$ can be identified with the Grassmannian $\mathrm{Gr}(\ell, d)$, which is a compact space. Then consider the set $S_{\ell}$ defined as
$$
S_{\ell} = \left\{(U_1, \dots, U_k, P) \in O(d)^k \times \mathrm{Gr}(\ell,d) \mid U_i P = P U_i \text{ for all } i=1, \dots, k\right\}.$$
To see that $S_{\ell}$ is a closed subset of $O(d)^k \times \mathrm{Gr}(\ell,d)$, consider the continuous map $\Phi: O(d)^k \times \mathrm{Gr}(\ell,d) \to (\mathbb{R}^{d \times d})^k$ given by
$$
f(U_1, \dots, U_k, P) = (U_1 P - P U_1, \dots, U_k P - P U_k)
$$
Then $S_{\ell} = f^{-1}((0, \ldots , 0))$ and therefore $S_{\ell}$ is closed. 

Finally, write $$S = \bigcup_{\ell = 1}^{d-1} S_{\ell},$$  As $S$ is a union of closed sets, it is closed and therefore compact as $O(d)^k \times \bigcup_{\ell = 1}^{d-1} \mathrm{Gr}(\ell,d)$ is compact.

To conclude the argument, observe that the set $O(d)^k \setminus \mathcal{I}_k$ is closed as it is the image of the compact set $S$ under the continuous projection map $\pi: O(d)^k \times \bigcup_{\ell=1}^{d-1} \mathrm{Gr}(\ell,d) \to O(d)^k$ given by $\pi(U_1, \dots, U_k, P) = (U_1, \dots, U_k)$.

\end{proof}

    \section{Order $k$ Detail}

    \label{section:Detail}

The goal of this section is to prove the product bound \eqref{Outline:StrongProduct} and to show how to convert \eqref{Outline:DecompositionGoal} into suitable estimates for detail. We first recall in section~\ref{Detail:Definition} the definition of the detail $s_r(\lambda)$ of a measure $\lambda$ on $\R^d$ at scale $r > 0$ that was first introduced by \cite{Kittle2021}.  We then expand the definition and results of order $k$ detail $s_r^{(k)}(\lambda)$ of a measure from \cite{Kittle2023} to measures on $\R^d$.

As mentioned in the outline of proofs, the advantage of using $k$-order detail over detail is that it leads to stronger product bounds. Indeed, we will show in Lemma~\ref{BasicProductBound} that 
\begin{equation}\label{BasicorderkBound}
    s_r^{(k)}(\lambda_1 * \cdots * \lambda_k) \leq s_r(\lambda_1)\cdots s_r(\lambda_k)
\end{equation}
for measures $\lambda_1, \ldots , \lambda_k$ on $\R^d$ and $r > 0$. Moreover, if $s_{r}^{(k)}(\lambda) \leq \alpha$ for all $r \in [a,b]$ and some $k \geq 1$ then we show in Proposition~\ref{kto1DetailBound} for a constant $Q'(d)$ depending only on $d$ that 
\begin{equation}\label{Outline:GoodDetailBound}
    s_{a\sqrt{k}}(\lambda) \leq Q'(d)^{k-1}(\alpha + k! ka^2b^{-2}).
\end{equation} Combining \eqref{BasicorderkBound} and \eqref{Outline:GoodDetailBound}, we deduce the strong product bound (Corollary~\ref{StrongProduct}) mentioned at \eqref{Outline:StrongProduct} in the outline of proofs. 

In section~\ref{Detail:Wasserstein}, we show that the difference in the detail of two measures is bounded in term of their Wasserstein distance. Finally, in section~\ref{Detail:Convert} we show how to convert the conditions from \eqref{Outline:DecompositionGoal} into good estimates for detail. The latter requires Berry-Essen type results, the Wasserstein distance bounds from section~\ref{Detail:Wasserstein},  \eqref{BasicorderkBound} and a suitable partition of $\sum_i X_i$.

All of these results will be used in section~\ref{section:decomposition}.

\subsection{Definitions}\label{Detail:Definition}

Denote by $\eta_y$ the standard Gaussian density on $\R^d$ with covariance matrix $y \cdot \mathrm{I}_d$, i.e. $$\eta_y(x) = \frac{1}{(2\pi y)^{d/2}} \exp\left( - \frac{||x||^2}{2y} \right).$$ Moreover, we write $$\eta_y^{(1)} = \frac{\partial}{\partial y} \eta_y.$$

Given a probability measure $\lambda$ on $\R^d$ the detail of $\lambda$ at scale $r > 0$ is defined as $$s_r(\lambda) = r^2 \, Q(d) \, ||\lambda * \eta_{r^2}^{(1)}||_1,$$ where $Q(d) = ||\eta_1^{(1)}||^{-1} = \frac{1}{2}\Gamma(\frac{d}{2})(\frac{d}{2e})^{-d/2}$ and note that by Stirling's approximation $d^{-1/2} \leq Q(d) \leq e d^{-1/2}$ for all $d\geq 1$. Moreover, $r^2  Q(d) = ||\eta_{r^2}^{(1)}||^{-1}$ and therefore $s_r(\lambda) \leq 1$ for every probability measure $\lambda$.

\begin{proposition}\cite[section 2]{Kittle2021}\label{DetailBasic}
    Let $\lambda$ and $\mu$ be probability measures on $\R^d$. Then the following properties hold:
    \begin{enumerate}[label = (\roman*)]
        \item Suppose that there is $\beta > 1$ such that $s_r(\lambda) < (\log r^{-1})^{-\beta}$ for sufficiently small $r$. Then $\lambda$ is absolutely continuous. 
        \item $s_r(\lambda * \mu) \leq s_r(\lambda).$
    \end{enumerate}
\end{proposition}

\begin{definition}
    Given a probability measure $\lambda$ on $\R^d$ and some $k \geq 1$ we define the \textbf{order $k$ detail of $\lambda$ at scale $r$} as $$s_r^{(k)}(\lambda) = r^{2k} \, Q(d)^k \, ||  \lambda *   \eta_{kr^2}^{(k)}  ||_1,$$ where $\eta_y^{(k)} = \frac{\partial^k}{\partial y^k}\eta_y$. 
\end{definition}

\subsection{Bounding Detail}

We have the following properties:

\begin{lemma}\label{BasicProductBound}
    Let $k \geq 1$ and let $\lambda_1, \lambda_2, \ldots , \lambda_k$ be probability measures on $\R^d$. Then  
    \begin{equation}\label{kDetailtoDetailBound}
        s_r^{(k)}(\lambda_1 * \lambda_2 * \ldots * \lambda_k) \leq s_r(\lambda_1) s_2(\lambda_2)\cdots s_r(\lambda_k).
    \end{equation}
    In particular, for any probability measure $\lambda$ on $\R^d$ and $k\geq 1$,
    \begin{equation}\label{kDetailBasic}
        s_r^{(k)}(\lambda) \leq 1.
    \end{equation}
\end{lemma}

\begin{proof}
    Recall that by the Heat equation $\frac{\partial}{\partial y}\eta_y(x) = \frac{1}{2} \sum_{i = 1}^d\frac{\partial^2}{\partial x_i^2}\eta_y(x)$ and therefore by standard properties of convolution 
    \begin{align*}
       \eta^{(k)}_{kr^2} &= \frac{1}{2^k}\sum_{i_1, \ldots , i_k = 1}^d\frac{\partial^{2}}{\partial x_{i_1}^{2}} \cdots \frac{\partial^{2}}{\partial x_{i_k}^{2}} \eta_{kr^2} \\
       &= \underbrace{\left( \frac{1}{2} \sum_{i = 1}^d\frac{\partial^2}{\partial x_i^2}\eta_{r^2}  \right) * \left( \frac{1}{2} \sum_{i = 1}^d\frac{\partial^2}{\partial x_i^2}\eta_{r^2}  \right) * \cdots * \left( \frac{1}{2} \sum_{i = 1}^d\frac{\partial^2}{\partial x_i^2}\eta_{r^2}  \right)}_{k \text{ times}} \\
       &= \underbrace{\eta_{r^2}^{(1)} * \eta_{r^2}^{(1)} * \cdots * \eta_{r^2}^{(1)}}_{k \text{ times}}.
    \end{align*} This concludes the proof of \eqref{kDetailtoDetailBound} as
    \begin{align*}
        ||\lambda_1 * \ldots * \lambda_k * \eta_{kr^2}^{(k)}||_1 &= ||\lambda_1 * \eta_{r^2}^{(1)} * \lambda_2 * \eta_{r^2}^{(1)} * \cdots * \lambda_k * \eta_{r^2}^{(1)}||_1 \\
        &\leq ||\lambda_1 * \eta_{r^2}^{(1)}||_1 \cdot ||\lambda_2 * \eta_{r^2}^{(1)}|| _1  \cdots ||\lambda_k * \eta_{r^2}^{(1)}||_1.
    \end{align*}

    To show \eqref{kDetailBasic} we set $\lambda_1 = \lambda$ and $\lambda_2 = \ldots = \lambda_k = \delta_e$ and use that $s_r(\lambda_i) \leq 1$.
\end{proof}

\begin{lemma}\label{ktok-1DetailBound}
    Let $k$ be an integer greater than $1$ and suppose that $\lambda$ is a probability measure on $\R^d$. Suppose that $a, b, c > 0$ and $\alpha \in (0,1)$. Assume that $a < b$ and that for all $r\in [a,b]$ it holds that $$s_r^{(k)}(\lambda) \leq \alpha + cr^{2k}.$$ Then for all $r \in \left[ a\sqrt{\frac{k}{k-1}}, b \sqrt{\frac{k}{k-1}}  \right]$ we have $$s_{r}^{(k-1)}(\lambda) \leq 2eQ(d)^{-1} \left(\alpha  + (b^{-2(k-1)} + ckb^2)r^{2(k-1)}\right).$$
\end{lemma}

\begin{proof}
    By the assumption and the definition of detail for $y \in [ka^2, kb^2]$ and writing $y = kr^2$, $$||\lambda * \eta_{y}^{(k)}||_1 \leq r^{-2k}Q(d)^{-k}(\alpha + cr^{2k}) = \alpha y^{-k}k^k Q(d)^{-k} + cQ(d)^{-k}.$$ Therefore with $y \in [ka^2, kb^2]$, 
    \begin{align*}
        ||\lambda * \eta^{(k-1)}_y||_1 &\leq ||\lambda * \eta^{(k-1)}_{kb^2}||_1 + \int_y^{kb^2} ||\lambda * \eta^{(k)}_u||_1 \, du \\
        &\leq ||\eta^{(k-1)}_{kb^2}||_1 + \int_y^{kb^2} \alpha u^{-k}k^k Q(d)^{-k} + cQ(d)^{-k} \, du \\
        &\leq (\tfrac{kb^2}{k-1})^{-(k - 1)} Q(d)^{-(k-1)} + \alpha k^k Q(d)^{-k} \tfrac{y^{-(k - 1)}}{k-1} +  Q(d)^{-k}c k b^2 ,
    \end{align*} where we bounded in the last inequality $||\eta^{(k-1)}_{kb^2}||_1$ by using that order $(k-1)$-detail is at most one, $\int_y^{kb^2} \alpha u^{-k}k^k Q(d)^{-k} \, du$ by $\int_y^{\infty} \alpha u^{-k}k^k Q(d)^{-k} \, du$ and $\int_y^{kb^2}  cQ(d)^{-k} \, du$ by $\int_0^{kb^2} cQ(d)^{-k} \, du$. Using that $(\tfrac{k}{k-1})^{-(k - 1)} < 1$ we therefore get 
    $$||\lambda * \eta^{(k-1)}_y||_1  \leq \alpha k^k Q(d)^{-k} \frac{y^{-(k - 1)}}{k-1} + (b^{-(2k-2)} + Q(d)^{-1}ckb^2)Q(d)^{-(k-1)}.$$

    Substituting the definition of order $k$ detail gives for $y = (k-1) r^2 \in [ka^2, kb^2]$ or equivalently  $r \in \left[ a\sqrt{\frac{k}{k-1}}, b \sqrt{\frac{k}{k-1}}  \right]$,
    \begin{align*}
    s_r^{(k-1)}(\lambda) &= r^{2(k-1)}Q(d)^{k-1}||\lambda * \eta^{(k-1)}_{(k-1)r^2}||_1 \\
    &\leq \alpha r^{2(k-1)} k^k Q(d)^{-1} \frac{((k-1)r^2)^{-(k - 1)}}{k-1} +r^{2(k-1)}(b^{-2(k-1)} + Q(d)^{-1}ckb^2) \\
    &\leq \alpha Q(d)^{-1}\left( 1 + \frac{1}{k-1}  \right)^k + (b^{-2(k-1)} + Q(d)^{-1}ckb^2)r^{2(k-1)}.
    \end{align*} Finally using that $\left( 1 + \frac{1}{k-1}  \right)^k \leq 2e$ and that $2eQ(d)^{-1} \geq 1$ the proof is concluded.
\end{proof}

\begin{proposition}\label{kto1DetailBound}
    Let $k$ be an integer greater than $1$ and suppose that $\lambda$ is a probability measure on $\R^d$. Suppose that $a, b > 0$ and $\alpha \in (0,1)$. Assume that $a < b$ and that for all $r \in [a,b]$ we have $$s_r^{(k)}(\lambda) \leq \alpha.$$ Then we have that $$s_{a\sqrt{k}}(\lambda) \leq Q'(d)^{k-1}(\alpha + k! \cdot k a^2b^{-2})$$ for $Q'(d) = 4eQ(d)^{-1} \geq 1$.
\end{proposition}

\begin{proof}
    We will show by induction for $j = k,k-1, \ldots , 1$ that for all $r\in \left[ a\sqrt{\frac{k}{j}},  b\sqrt{\frac{k}{j}} \right]$ we have 
    \begin{equation}\label{InductiveStep}
        s_r^{(j)}(\lambda) \leq Q'(d)^{k-j}\left(\alpha + \frac{k!}{j!}b^{-2j}r^{2j}\right),
    \end{equation} which implies the claim by setting $j = 1$ and $r = a\sqrt{k}$.
    The case $j = k$ follows from the conditions of the lemma. For the inductive step assume now that for all $r\in \left[ a\sqrt{\frac{k}{j}},  b\sqrt{\frac{k}{j}} \right]$ we have that \eqref{InductiveStep} holds. Then by Lemma~\ref{ktok-1DetailBound} we have for all $r\in \left[ a\sqrt{\frac{k}{j-1}},  b\sqrt{\frac{k}{j-1}} \right]$ 
    \begin{align*}
        s_r^{(j-1)}(\lambda) &\leq  Q'(d)^{k-j}2eQ(d)^{-1}\left(\alpha + \left(b^{-2(j-1)} + \frac{k!}{j!}b^{-2j}jb^2\right)r^{2(j-1)}\right) \\
        &\leq Q'(d)^{k-j}2eQ(d)^{-1}\left(\alpha + \left(1 + \frac{k!}{(j-1)!}\right)b^{-2(j-1)}r^{2(j-1)}\right) \\
        &\leq Q'(d)^{k-(j-1)}\left(\alpha + \frac{k!}{(j-1)!}b^{-2(j-1)}r^{2(j-1)}\right).
    \end{align*}
\end{proof}

Combining Lemma~\ref{BasicProductBound} and Proposition~\ref{kto1DetailBound}, we arrive at the following corollary.

\begin{corollary}\label{StrongProduct}
    Let $k \geq 1$ and let $\lambda_1, \lambda_2, \ldots , \lambda_k$ be probability measures on $\R^d$. Suppose that $a, b > 0$ and $\alpha \in (0,1)$. Assume that $a < b$ and that for all $r \in [a,b]$ and $i \in [k]$ we have $$s_r(\lambda_i) \leq \alpha.$$ Then it holds that $$s_{a\sqrt{k}}(\lambda) \leq Q'(d)^{k-1}(\alpha^k + k! \cdot k a^2b^{-2}).$$ 
\end{corollary}

\subsection{Wasserstein Distance} \label{Detail:Wasserstein}

Recall as in \eqref{WassersteinDef} that the Wasserstein 1-distance on $\R^d$ between $\lambda_1$ and $\lambda_2$ is defined as $$\mathcal{W}_1(\lambda_1, \lambda_2) = \inf_{\gamma \in \Gamma(\lambda_1, \lambda_2)}  \int_{\R^d \times \R^d} |x-y| \, d\gamma(x,y),$$ where $\Gamma(\lambda_1, \lambda_2)$ is the set of couplings between $\lambda_1$ and $\lambda_2$. We show that detail is bounded up to a constant by the Wasserstein distance.

\begin{lemma}\label{DetailWassersteinBound}
    Let $\lambda_1$ and $\lambda_2$ be probability measures on $\R^d$. Then for $k \geq 1$ and $r > 0$, $$|s_r^{(k)}(\lambda_1) - s_r^{(k)}(\lambda_2)| \leq edr^{-1}\mathcal{W}_1(\lambda_1, \lambda_2),$$ where $e$ is Euler's number. 
\end{lemma}

\begin{proof}
    Let $X$ and $Y$ be random variables with laws $\lambda_1$ and $\lambda_2$ respectively. Then $$(\lambda_1 - \lambda_2)*\eta^{(k)}_{kr}(v) = \E\left[ \eta^{(k)}_{kr}(v-X)  -  \eta^{(k)}_{kr}(v-Y)  \right]$$ and therefore $$\big|(\lambda_1 - \lambda_2)*\eta^{(k)}_{kr}(v)\big| \leq \E\left[\big| \eta^{(k)}_{kr}(v-X)  -  \eta^{(k)}_{kr}(v-Y)  \big|\right].$$ Note that  $$\big| \eta^{(k)}_{kr}(v-X)  -  \eta^{(k)}_{kr}(v-Y)  \big| \leq \int_X^Y \big| \nabla \eta^{(k)}_{kr}(v-u) \big| \, |du|,$$ where $\int_x^y \cdot |du|$ is understood to be the integral along the shortest path between $x$ and $y$ and $\nabla$ is the gradient. Thus 
    \begin{align*}
        ||(\lambda_1 - \lambda_2)*\eta^{(k)}_{kr}||_1 &\leq \int_{\R^d}  \E\left[\int_X^Y \big| \nabla \eta^{(k)}_{kr}(v-u) \big| \, |du| \right]  \, dv \\
        &=   \E\left[\int_X^Y \int_{\R^d} \big| \nabla \eta^{(k)}_{kr}(v-u) \big|  \, dv  \, |du| \right] \\
        &= ||\nabla \eta_{kr}^{(k)}||_1  \E[|X-Y|] \\
        &\leq \left( \sum_{i = 1}^d \bigg|\bigg|\frac{\partial}{\partial x_i}\eta^{(k)}_{kr}\bigg|\bigg|_1 \right) \E[|X-Y|]
    \end{align*} We next bound $||\frac{\partial}{\partial x_i}\eta^{(k)}_{kr}||_1$. As in the proof of Lemma~\ref{BasicProductBound}, it follows that $$\frac{\partial}{\partial x_i}\eta^{(k)}_{kr^2}  = \left(\frac{\partial}{\partial x_i}\eta_{\frac{k}{k+1} r^2}\right) * \underbrace{\eta^{(1)}_{\frac{k}{k+1} r^2} * \ldots  * \eta^{(1)}_{\frac{k}{k+1} r^2}}_{k\text{ times}}.$$ Using standard properties of Gaussian integrals, $$\bigg|\bigg|\frac{\partial}{\partial x_i}\eta_{\frac{k}{k+1} r^2}\bigg|\bigg|_1 = \sqrt{\frac{2(k + 1)}{k\pi }} r^{-1} \leq \sqrt{\frac{k + 1}{k}} r^{-1}$$ and therefore 
    \begin{align*}
      \bigg|\bigg|\frac{\partial}{\partial x_i}\eta^{(k)}_{kr}\bigg|\bigg|_1 &\leq \bigg|\bigg|\frac{\partial}{\partial x_i}\eta_{\frac{k}{k+1} r^2}\bigg|\bigg|_1 \cdot \big|\big|\eta^{(1)}_{\frac{k}{k+1} r^2}\big|\big|_1^k \\
    &\leq \left( \frac{k+ 1}{k}  \right)^{(k+1)/2} Q(d)^{-k} r^{-2k-1}. 
    \end{align*} Using that $\left( \frac{k+ 1}{k}  \right)^{(k+1)/2} \leq e$, we conclude  
    \begin{align*}
        |s_r^{(k)}(\lambda_1) - s_r^{(k)}(\lambda_2)| &\leq r^{2k} Q(d)^k ||(\lambda_1 - \lambda_2)*\eta_{kr}^{(k)}||_1 \\
        &\leq de r^{-1}\E[|X-Y|].
    \end{align*}
    Choosing a coupling for $X$ and $Y$ which minimizes $\E[|X-Y|]$ gives the required result. 
\end{proof}

\subsection{Small Random Variables Bound in $\R^d$} \label{Detail:Convert}

The aim of this subsection is to show that the sum of independent random variables in $\R^d$ has small detail whenever they are support close to $0$ and have a sufficiently large variance. To state our result, we use the partial order \eqref{MatrixPartialOrder} for positive semi-definite symmetric matrices. 

\begin{proposition}\label{kDetailBerryEssen}
    For every positive integer $d\geq 1$ and every $\alpha > 0$ there exists some $C = C(\alpha,d) > 0$ such that the following is true for all $r > 0$ and positive integers $k$. Let $X_1, X_2, \ldots , X_n$ be independent random variables taking values in $\R^d$ such that almost surely $$|X_i| \leq C^{-1} r \quad \text{ and } \quad \sum_{i = 1}^n \mathrm{Var} \, X_i \geq C k r^2 I.$$ Then $$s_r^{(k)}(X_1 + \ldots + X_n) \leq \alpha^k.$$
\end{proposition}

Proposition~\ref{kDetailBerryEssen} relies on a higher dimensional Berry-Essen type result, which implies Proposition~\ref{kDetailBerryEssen} for $k=1$, as deduced in Lemma~\ref{DetailBerryEssen}. To prove the higher dimensional Berry-Essen type result we first need the following.
	
	\begin{theorem} \label{BerryEssenType1D}
		Let $X_1, X_2, \dots, X_n$ be independent random variables taking values in $\R$ with mean $0$ and for each $i \in [n]$ let $\mathbb{E}[X_i^2] = \omega_i^2$ and $\mathbb{E}[|X_i|^3] = \gamma_i^3 < \infty$. Let $\omega^2 = \sum_{i=1}^n \omega_i^2$ and let $S = X_1 + \dots + X_n$. Let $N$ be a normal distribution with mean $0$ and variance $\omega^2$. Then for an absolute implied constant
		
		\begin{equation*}
		\mathcal{W}_1(S, N) \ll \frac{\sum_{i=1}^n \gamma_i^3}{\sum_{i=1}^n \omega_i^2}.
		\end{equation*}
		
	\end{theorem}
	\begin{proof}
		A proof of this result may be found in \cite{Erickson_1973}.
	\end{proof}

	From this we may deduce the following higher dimensional Berry-Essen type result by using projections onto one-dimensional subspaces.
	
	\begin{lemma}\label{BerryEssenType}
		Let $X_1, X_2, \ldots , X_n$ be independent random variables taking values in $\R^d$ with mean $0$ and for each $i \in [n]$ write $$\Sigma_i = \mathrm{Var} \, X_i.$$ Suppose that $\delta > 0$ is such that for each $i \in [n]$ we have $|X_i| \leq \delta$ almost surely. Let $\Sigma = \sum_{i = 1}^n \Sigma_i$ and $S = X_1 + \ldots + X_n$. Let $N$ be a multivariate normal distribution with mean $0$ and covariance matrix $\Sigma$. Then $$\mathcal{W}_1(S, N) \ll_d \delta.$$
	\end{lemma}
	
	\begin{proof}
		First we will deduce this from Theorem~\ref{BerryEssenType1D} when $d=1$. In this case simply note that
		\begin{align*}
		\sum_{i=1}^{n} \gamma_i^3  = \sum_{i=1}^n \mathbb{E}[|X_i|^3]  \leq \sum_{i=1}^n \mathbb{E}[\delta |X_i|^2]  = \delta \sum_{i = 1}^n \omega_i^2,
		\end{align*}
		showing the claim.
		
		Now in the case $d \geq 1$ the lemma follows by using, as shown in \cite{BayraktarGuoi2021}*{Theorem 2.1}, that  $$\mathcal{W}_1(S,N) \ll_d \sup_{p} \mathcal{W}_1(pS,pN),$$ where the supremum is taken over all one dimensional projections $p$. The result is therefore deduced as in the one dimensional case by using that $\mathbb{E}[|pX_i|^3] \leq \delta \mathbb{E}[|pX_i|^2]$.
	\end{proof}

\begin{lemma}\label{DetailBerryEssen}
    For every positive integer $d\geq 1$ and every $\alpha > 0$ there exists some $C = C(\alpha,d) > 0$ such that the following is true. Let $r > 0$ and let $X_1, X_2, \ldots , X_n$ be independent random variables taking values in $\R^d$ such that $$|X_i| \leq C^{-1} r \quad \text{ and } \quad \sum_{i = 1}^n \mathrm{Var} \, X_i \geq C r^2 I.$$ Then $$s_r(X_1 + \ldots + X_n) \leq \alpha.$$
\end{lemma}

\begin{proof}
    Denote for $1\leq i \leq n$ by $X_i' = X_i - \E[X_i]$ and let $S' = \sum_{i = 1}^n X_i'$. Note that $s_r(\sum_{i = 1}^n X_i) = s_r(S')$. Write $\Sigma_i = \mathrm{Var} \, X_i$ and let $\Sigma = \sum_{i = 1}^n \Sigma_i$. Let $N$ be a multivariate normal distribution with mean $0$ and covariance matrix $\Sigma$. Note that $|X_i'| \leq 2 C^{-1} r$ almost surely. Therefore by Lemma~\ref{BerryEssenType}, $$\mathcal{W}_1(S',N) \ll_d C^{-1} r.$$ Also
    \begin{align*}
        s_r(N) & \leq s_r(\eta_{C^2 r^2} ) = \frac{\| \eta_{C^2 r^2 + r^2}^{(1)} \|}{\| \eta_{r^2}^{(1)} \|} = \frac{1}{C^2 + 1}.
    \end{align*}
    Thus by Lemma~\ref{DetailWassersteinBound}, $$s_r(X_1 + \ldots + X_n) = s_r(S')  \ll_d C^{-1} + \frac{1}{1 + C^2} ,$$ implying the claim. 
\end{proof}

The proof of Proposition~\ref{kDetailBerryEssen} in the case $k \geq 2$ is more involved than the proof in the case $k=1$. In order to prove this proposition we also need the following lemma and a corollary of it.

\begin{lemma}\label{EucVecLemma}
    Let $V$ be a Euclidean vector space, let $v_1, \dots, v_n \in V$ and write $S =  v_1 + \dots + v_n$. Let $c_1, c_2 > 0$ be such that for all $i \in [n]$ we have $$|v_i| \leq c_1 \quad \text{ and } \quad v_i \cdot S \geq c_2 |v_i| |S|.$$ Let $k$ be a positive integer. Then we can partition $[n]$ as $J_1 \sqcup J_2 \sqcup \dots \sqcup J_k$ such that for each $j \in [k]$ we have $$|S_j - \tfrac{1}{k}S| < c_2^{-1} \sqrt{\tfrac{2c_1}{k}|S|} + 2 c_2^{-2} c_1$$
    where $S_j = \sum_{i \in J_j} v_i$.
\end{lemma}

\begin{proof}
    Choose the $J_j$ such that 
    \begin{equation}
        \sum_{j=1}^{k} | S_j |^2 \label{eq:partition_sum}
    \end{equation} is minimized. For each $i \in [n]$ let $j(i)$ denote the unique $j \in [k]$ such that $i \in J_j$. For each $i \in [n]$ and $j' \in [k]$ we know that moving $i$ from $J_{j(i)}$ to $J_{j'}$ cannot decrease the sum in \eqref{eq:partition_sum}. Therefore
    $$|S_{j(i)} - v_i|^2 + |S_{j'} + v_i |^2 \geq |S_{j(i)}|^2 + |S_{j'}|^2.$$
    Expanding this out and cancelling gives $$S_{j(i)} \cdot v_i - |v_i|^2 \leq S_{j'} \cdot v_i$$ and summing over all $i \in J_j$, we get $$S_j \cdot S_j \leq S_j \cdot S_{j'} + \sum_{i \in J_j} |v_i|^2.$$ Let $A_j$ denote $ \sum_{i \in J_j} |v_i|^2$. Note that the above equation gives $|S_j - S_{j'}|^2 \leq A_j + A_{j'}$
    and so
    \begin{align}
        |S_j - \tfrac{1}{k}S| \leq \max_{j' \in [k]} |S_j - S_{j'}| 
        \leq \sqrt{2 \max_{j' \in [k]} A_{j'}}. \label{eq:distance_to_s} 
    \end{align}
    Now let $\Lambda^2 = \max_{j' \in [k]} A_{j'}$. We compute
    \begin{align*}
        \sum_{i \in J_j} |v_i|^2 & \leq c_2^{-2} |S|^{-2} \sum_{i \in J_j} (v_i \cdot S)^2 \\
        & \leq c_2^{-2} |S|^{-2} \sum_{i \in J_j} (v_i \cdot S) c_1 |S| \\
        & = c_2^{-2} c_1 |S|^{-1} S \cdot S_j
         \leq c_2^{-2} c_1 |S_j|
         \leq c_2^{-2} c_1 (\tfrac{1}{k}|S|  + \sqrt{2} \Lambda).
    \end{align*}
    Therefore $\Lambda^2 \leq c_2^{-2} c_1 (|S| / k + \sqrt{2} \Lambda)$, which gives $$\left(\Lambda - c_2^{-2} c_1 / \sqrt{2} \right)^2 \leq c_2^{-2} c_1 |S| / k + c_2^{-4} c_1^2 / 2$$ and so
    \begin{align*}
        \Lambda & \leq \sqrt{\frac{c_2^{-2} c_1 |S|}{k} + \frac{c_2^{-4} c_1^2} {2}} + \frac{c_2^{-2} c_1}{\sqrt{2}}\\
        & \leq c_2^{-1} \sqrt{\tfrac{c_1}{k} |S|} + c_2^{-2} c_1 \sqrt{2},
    \end{align*} showing the required result by \eqref{eq:distance_to_s}.
\end{proof}

\begin{corollary} \label{PartitionMatrixSums}
    Let $A_1, \dots, A_n$ be symmetric positive semi-definite $d \times d$ matrices. Suppose that $\sum_{i=1}^n A_i \geq CkI$ and that for each $i \in [n]$ we have $\| A_i \| \leq c$. Then we can partition $[n]$ as $J_1 \sqcup J_2 \sqcup \dots \sqcup J_k$ such that for each $j \in [k]$ we have $$\sum_{i \in J_j} A_i \geq \left(C - d \sqrt{2 c C} - 2 d^{3/2} c\right) I.$$
\end{corollary}

\begin{proof}
    Let $M = \sum_{i=1}^n A_i$. We know that $M$ is symmetric positive semi-definite and so it may be diagonalised as $M = P^{-1} D P$ for some orthogonal matrix $P$ and a diagonal matrix $D$ with non-zero real entries. Since $M \geq CkI$ all of the diagonal entries of $D$ are at least $Ck$. Let $D' = \sqrt{CkD^{-1}}$ be a diagonal matrix and for each $i \in [n]$ let $A_i' = Q A_i Q$ where $Q = P^{-1} D' P$. Note that $A_i'$ is symmetric positive semi-definite, $\| A_i' \| \leq c$ as $\| Q \| \leq 1$ and that $\sum_{i=1}^n A_i' = CkI$ since $$QMQ = (P^{-1}D'P)(P^{-1}DP)(P^{-1}D'P) = P^{-1}D'DD'P = CkI.$$

    We now apply Lemma~\ref{EucVecLemma} with $V$ being the space of symmetric $d \times d$ matrices with inner product given by $A \cdot B = \sum_{x = 1}^n \sum_{y = 1}^n A_{xy} B_{xy} = \tr \, AB$ and with $v_1, \dots, v_n$ being $A_1', \dots, A_n'$. We will denote the norm induced by this inner product by $|\cdot|$. Note that given a symmetric matrix $A$ we have that $|A|^2$ is equal to the sum of the squares of the eigenvalues of $A$ and so in particular $\|\cdot\| \leq |\cdot| \leq \sqrt{d} \| \cdot \|$. This means that we can take $c_1 = \sqrt{d} c$ so that $|A_1'| \leq c_1$.
    
    All that we need to do is find some lower bound on $A_i' \cdot CkI$ in terms of $| A_i' |\cdot |CkI|$. Note that $\tr \, A_i'$ is equal to the sum of the eigenvalues of $A_i'$ and that $|A_i'|^2$ is equal to the sum of the squares of these eigenvalues. In particular since the eigenvalues are non-negative $\tr \, A_i' \geq | A_i'|$ and so
    \begin{align*}
        A_i' \cdot CkI & = Ck \, \tr \, A_i' \geq Ck |A_i'|
        = |A_i'| \cdot |CkI| / \sqrt{d}.
    \end{align*}
    This means that we can take $c_2 = 1 / \sqrt{d}$.

    We now apply Lemma~\ref{EucVecLemma} with $S = \sum_{i=1}^n A_i' = CkI$ to construct our partition $[n] = J_1 \sqcup J_2 \sqcup \dots \sqcup J_k$ such that for all $j \in [k]$, $$\left\| \sum_{i \in J_j} A_i' - CI \right\| \leq \left| \sum_{i \in J_j} A_i' - CI \right| \leq d \sqrt{2 c C} + 2 d^{3/2} c.$$ Therefore
		$$\left\| \sum_{i \in J_j} A_i - CQ^{-2} \right\| \leq (d \sqrt{2 c C} + 2 d^{3/2} c)||Q^{-2}||$$
		and hence,
		\begin{align*}
		\sum_{i \in J_j} A_i  & \geq C  Q^{-2} - \left( d \sqrt{2 c C} + 2 d^{3/2} c\right)  ||Q^{-2}|| I\\
		&\geq CI - \left(d \sqrt{2 c C} - 2 d^{3/2} c\right) ||Q^{-2}|| I \\
		& \geq \left(C - d \sqrt{2 c C} - 2 d^{3/2} c\right) I
		\end{align*} using in the penultimate line that $Q^{-2} = P^{-1}(D')^{-2}P$ is symmetric and all eigenvalues are $\geq 1$ and in the last line that $||Q^{-1}|| \geq 1$.
\end{proof}

Finally we can prove Proposition~\ref{kDetailBerryEssen}.

\begin{proof}[Proof of Proposition~\ref{kDetailBerryEssen}]
    Note that since $|X_i| \leq C^{-1} r$ almost surely we have $\| \var X_i \| \leq C^{-2} r^2$. By Corollary~\ref{PartitionMatrixSums} we can partition $[n]$ as $J_1 \sqcup J_2 \sqcup \dots \sqcup J_k$ such that for each $j \in [k]$ we have $$\sum_{i \in J_j} \var X_i \geq \left( C - d \sqrt{2C^{-1}} - 2 d^{3/2} C^{-2} \right) r^{2} I.$$

    This means that by Lemma~\ref{DetailBerryEssen}, provided that $C$ is sufficiently large in terms of $d$, we know that $$s_r\left( \sum_{i \in J_j} X_i \right) \leq \alpha.$$ The result now follows from Proposition~\ref{BasicProductBound}.
\end{proof}

    \section{Entropy Gap and Variance Growth on $\mathrm{Sim}(\R^d)$}
    
    \label{EntropyIsomSection}

    Throughout this section we use various results from \cite{KittleKoglerEntropy} and we refer the reader to the latter paper for an extensive discussion of the used results. We first give an outline of how  the main results in this section are established. 

For $\mu$ a probability measure on $G = \mathrm{Sim}(\R^d)$ we denote by $\gamma_1, \gamma_2, \ldots$ independent $\mu$-distributed samples of $\mu$ and write $$q_n = \gamma_1\cdots \gamma_n.$$  For $\kappa > 0$ be denote by $\tau_{\kappa}$ the stopping time $$\tau_{\kappa} = \inf\{ n\,:\, \rho(q_n) \leq \kappa \}.$$

The goal of this section is to give bounds for $\sum_{i = 1}^N \mathrm{tr}(q_{\tau_{\kappa}}, s_i)$ for suitable scales $s_i$, for $\mathrm{tr}(q_{\tau_{\kappa}}, s_i)$ as defined in section~\ref{section:vartrdef}. Towards the proof of our main theorem as discussed in section~\ref{section:Outline}, it would be ideal to give a bound roughly of the form
\begin{equation}\label{IdealTraceSumBound}
    \sum_{i = 1}^{N} \mathrm{tr}(q_{\tau_{\kappa}}, 2^ia r) \gg \frac{h_{\mu}}{|\chi_{\mu}|} \log \kappa^{-1} \quad\quad \text{ with } \quad r  \approx \kappa^{\frac{S_{\mu}}{|\chi_{\mu}|}} \quad \text{ and } \quad 2^Nr \approx \kappa^{\frac{h_{\mu}}{2\ell|\chi_{\mu}|}}
\end{equation} for sufficiently small $\kappa$. As we explain below, we can't quite achieve \eqref{IdealTraceSumBound} and the bound we arrive at will also depend on the separation rate $S_{\mu}$. To estimate the left hand side of \eqref{IdealTraceSumBound} we apply \cite{KittleKoglerEntropy}*{Theorem 1.5} to each of the terms $\mathrm{tr}(q_{\tau_{\kappa}}, 2^ia r)$ which gives
\begin{equation}\label{SketchTraceSum}
    \sum_{i = 1}^{N} \mathrm{tr}(q_{\tau_{\kappa}}, 2^ia r) \gg a^{-2}(H_a(q_{\tau_{\kappa}}; r|2^Nr) + O_{d}(Ne^{-a^2/4}) + O_{d}(r) )
\end{equation}
having used that by a telescoping sum $$H_a(q_{\tau_{\kappa}}; r|2^Nr) = \sum_{i = 1}^N H_a(q_{\tau_{\kappa}}; 2^{i - 1}r|2^i r).$$ The main contribution from \eqref{IdealTraceSumBound} comes from suitable estimates for $H_a(q_{\tau}; r|2^Nr)$. Indeed, we will show in Proposition~\ref{EntropyBetweenScalesIncrease} that, up to negligible error terms, 
\begin{equation}\label{ApproxEntropyBetweenScales}
    H_a(q_{\tau_{\kappa}}; r|2^Nr)  \gg \frac{h_{\mu}}{|\chi_{\mu}|} \log \kappa^{-1}.
\end{equation}
To show this, we recall that $$H_a(q_{\tau_{\kappa}}; r|2^Nr) = H_a(q_{\tau_{\kappa}}; r) - H_a(q_{\tau_{\kappa}}; 2^N r)$$ and therefore we need to estimate the terms $H_a(q_{\tau_{\kappa}}; r)$ and $H_a(q_{\tau_{\kappa}}; 2^N r)$. To bound the first term (see \cite{KittleKoglerEntropy}), we use that with high probability $\tau_{\kappa} \approx  \log(\kappa^{-1})/|\chi_{\mu}|$ and so the points in the support of $q_{\tau_{\kappa}}$ are separated by distance $r \approx \kappa^{\frac{S_{\mu}}{|\chi_{\mu}|}} \approx \exp(-S_{\mu}\tau_{\kappa})$. For the second term we exploit decay properties of smoothenings of our self-similar measure. 

Combining \eqref{SketchTraceSum} with \eqref{ApproxEntropyBetweenScales} would lead to \eqref{IdealTraceSumBound} would it not be for the error term $O_{d}(Ne^{-a^2/\ell})$. Indeed, to not cancel out the lower bound from \eqref{ApproxEntropyBetweenScales} we require that $$Ne^{-a^2/\ell} \leq c \frac{h_{\mu}}{|\chi_{\mu}|} \log \kappa^{-1}$$ for a sufficiently small constant $c$. By our choice of $N$ it holds that $N \approx \frac{S_{\mu}}{|\chi_{\mu}|} \log \kappa^{-1}$ and therefore $$e^{-a^2/\ell} \leq c \frac{h_{\mu}}{S_{\mu}}.$$ So we have to set $$a^2 = c\max\left\{ 1,\log \frac{S_{\mu}}{h_{\mu}}  \right\}.$$ Applying then \eqref{SketchTraceSum}, since the error term $O_d(r)$ is negligible, we conclude that 
\begin{equation}\label{IntroLotsofTrace}
    \sum_{i = 1}^{N} \mathrm{tr}(q_{\tau_{\kappa}}, 2^i a r) \gg \frac{h_{\mu}}{|\chi_{\mu}|} \log \kappa^{-1} \max\left\{ 1,\log \frac{S_{\mu}}{h_{\mu}}  \right\}^{-1}.
\end{equation}
We will deduce a more precise result from \cite{KittleKoglerEntropy}*{Proposition 1.5} in Proposition~\ref{LotsOfTrace}.

\subsection{Definitions of entropy and trace}\label{section:vartrdef}

For the convenience of the reader, we recall the definitions from \cite{KittleKoglerEntropy}. The following smoothing functions are used: For given $r > 0$ and $a \geq 1$, denote by $\beta_{a,r}$ a random variable on $\mathfrak{g}$ with density function $f_{a,r}: \mathfrak{g} \to \R$ defined as $$f_{a,r}(x) = \begin{cases}
    C_{a,r}e^{-\frac{|x|^2}{2r^2}} & \text{if } |x| \leq ar, \\
    0 &\text{otherwise},
    \end{cases}$$ where $C_{a,r}$ is a normalizing constant to ensure that $f_{a,r}$ integrates to $1$. We then set 
    \begin{equation}\label{sardef}
        s_{a,r} = \exp(\beta_{a,r})
    \end{equation}
    and then define the entropy of a $G$-valued random variable $g$ at scale $r > 0$ with respect to the parameter $a \geq 1$ as 
    \begin{equation}\label{Hadef}
        H_a(g;r) = H(g; s_{a,r}) = H(gs_{a,r}) - H(s_{a,r}),
    \end{equation}
    where $H$ is the differential entropy with respect to the Haar measure $\Haarof{G}$ on $G$. The entropy between scales $r_1, r_2 > 0$ is defined as 
\begin{align*}\label{DefEntScales}
    H_a(g;r_1|r_2) &= H(g; s_{r_1,a}| s_{r_2,a}) = H_a(g;r_1) - H_a(g;r_2) \\ &= (H(gs_{r_1,a}) - H(s_{r_1,a})) - ( H(gs_{r_2,a}) - H(s_{r_2,a})).
\end{align*}

We define the trace of $g$ (a $G$-valued random variable) at scale $r$, denoted $\tr(g; r)$,  to be the supremum of $t\geq 0$ such there exists some $\sigma$-algebra $\mathscr{A}$ and a $\mathscr{A}$-measurable random variable $h$ taking values in $G$ such that $$|\log(h^{-1}g)| \leq r \quad \text{ and } \quad \E[\tr_h(g|\mathscr{A})] \geq tr^2.$$

\subsection{Entropy Gap of Stopped Random Walk}

In this subsection we show that the entropy between scales is large for a suitable stopped random walk on $G = \mathrm{Sim}(\R^d)$. Indeed, we establish the following more precise version of \eqref{ApproxEntropyBetweenScales}.

\begin{proposition}\label{EntropyBetweenScalesIncrease}
    Let $\mu$ be a finitely supported, contracting on average probability measure on $G$.  Suppose that $S_{\mu} < \infty$ and  that $h_{\mu}/|\chi_{\mu}|$ is sufficiently large. Let $S > S_{\mu}$, $\kappa > 0$ and $a \geq 1$ and suppose that $0 < r_1 < r_2 < a^{-1}$ with $r_1 < \exp( -S \log(\kappa^{-1})/|\chi_{\mu}|)$. Then as $\kappa \to 0$,
    \begin{align*}
          H_a(q_{\tau_{\kappa}}; r_1|r_2) \geq \left( \frac{\RWEntropy}{|\chi_{\mu}|} - d  \right) \log \kappa^{-1} + H(s_{r_2,a}) + o_{\mu,d,S,a}(\log \kappa^{-1}).
    \end{align*}
\end{proposition}

Proposition~\ref{EntropyBetweenScalesIncrease} directly follows from Lemma~\ref{EntScaleGrowth1} and Lemma~\ref{EntScaleGrowth2}. 
 
\begin{lemma}\label{EntScaleGrowth1}
    Under the assumptions of Proposition~\ref{EntropyBetweenScalesIncrease}, as $\kappa \to 0$, $$H_a(q_{\tau_{\kappa}}; r_1) \geq \frac{h_{\mu}}{|\chi_{\mu}|} \log \kappa^{-1} + o_{\mu,d,S,a}(\log \kappa^{-1}).$$
\end{lemma}

\begin{proof}
    This follows from \cite{KittleKoglerEntropy}*{Corollary 1.3} since $q_{\tau}$ satisfies a large deviation principle by Lemma~\ref{rhotauLDP} and we refer to the latter paper for a discussion.
\end{proof}

\begin{lemma}\label{EntScaleGrowth2}
    Under the assumptions of Proposition~\ref{EntropyBetweenScalesIncrease}, as $\kappa \to 0$,  $$H(q_{\tau_{\kappa}} s_{r_2,a}) \leq d \log \kappa^{-1} + o_{\mu,d,a}(\log \kappa^{-1}).$$ 
\end{lemma}

\begin{proof}
    For convenience write $\tau = \tau_{\kappa}$ and $K = |\mathrm{supp}(\mu)|$. We use the product structure on $G$ combined with \cite{KittleKoglerEntropy}*{Lemma 2.5}. Indeed, note that a choice of Haar measure on $G$ is given as $$\int f \, d\Haarof{G} = \int f(\rho U + b) \, \rho^{-(d+1)}  d\rho dU db,$$ for $dr, db$ the Lebesgue measure and $dU$ the Haar probability measure on $O(d)$. Therefore by \cite{KittleKoglerEntropy}*{Lemma 2.5}, $H(q_{\tau} s_{r_2,a}) \leq $ $$D_{\mathrm{KL}}(\rho(q_{\tau} s_{r_2,a})\, || \, \rho^{-(d + 1)}d\rho) + D_{\mathrm{KL}}(U(q_{\tau} s_{r_2,a})\, || \, dU) + D_{\mathrm{KL}}(b(q_{\tau} s_{r_2,a})\, || \, db).$$ We give suitable bounds for each these terms. As $dU$ is a probability measure $D_{\mathrm{KL}}(U(q_{\tau} s_{r_2,a})\, || \, dU) \leq 0$ by \cite{KittleKoglerEntropy}*{Lemma 2.4}. 

    We next deal with $D_{\mathrm{KL}}(b(q_{\tau} s_{r_2,a})\, || \, db)$. Denote by $\nu_{\tau}$ the distribution of $b(q_{\tau} s_{r_2,a})$. We claim that there is $\alpha = \alpha(\mu,d,a)$ such that 
    \begin{equation}\label{nutaudecay}
        \nu_{\tau}(B_R^C) \leq R^{-\alpha}
    \end{equation}
    for all sufficiently small $\kappa$ and sufficiently large $R$. Note that $$|b(q_{\tau} s_{r_2,a})| = |\rho(q_{\tau})U(q_{\tau})b(s_{r_2,a}) + b(q_{\tau})| \leq \kappa |b(s_{r_2,a})| + |b(q_{\tau})|$$ and therefore it suffices to show \eqref{nutaudecay} for the distribution of $b(q_{\tau})$, which we denote by $\nu'_{\tau}$. For $x \in \R^d$, $$|b(q_{\tau}) - q_{\tau}(x)| \leq |q_{\tau}(0) - q_{\tau}(x)| \leq \rho(q_{\tau})|x| \leq \kappa|x|$$ and so $|b(q_{\tau})| \leq |q_{\tau}(x)| + \kappa|x|$. Therefore if $R\leq |b(q_{\tau})|$ then either $R/2 \leq |q_{\tau}(x)|$ or $R/2 \leq \kappa|x|$. Also note that if $x$ is sampled from $\nu$ independently from $\gamma_1, \gamma_2,\ldots$ then $q_{\tau}(x)$ has law $\nu$. By \eqref{PolynomialTailDecay} this implies that $$\nu_{\tau}'(B_R^c) \leq \nu(B_{R/2}^c) + \nu(B_{R/2\kappa}^c) \leq R^{-\alpha_2}2^{\alpha_2}(1 + \kappa^{-1})^{-\alpha_2},$$ showing \eqref{nutaudecay}.

    To conclude we deduce from \eqref{nutaudecay} that $D_{\mathrm{KL}}( \nu_{\tau}  \, || \, db)$ is bounded by a constant depending on $\mu, d$ and $a $ and therefore is $\leq o_{\mu,d,a}(\log \kappa^{-1})$. Indeed denote by $f_{\tau}$ the density of $\nu_{\tau}$ such that $$D_{\mathrm{KL}}( \nu_{\tau}  \, || \, db) = \int -f_{\tau} \log f_{\tau} \, d\Haarof{\R^d}.$$ Also let $L > 1$ be a constant and for $i = 0,1,2,\ldots $  write $p_i = \nu_{\tau}(B_{L^{i+1}}\backslash B_{L^{i}})$ such that $p_i \leq \nu_{\tau}(B_{L^i}^c) \leq L^{-i\alpha}$. Thus it holds by Jensen's inequality for $h(x) = -x\log x$, 
    \begin{align*}
        D_{\mathrm{KL}}( \nu_{\tau}  \, || \, db) &= \sum_{i \geq 0} \int_{B_{L^{i+1}}\backslash B_{L^{i}}} -f_{\tau} \log f_{\tau} \, d\Haarof{\R^d} \\
        &= \sum_{i \geq 0} \int_{B_{L^{i+1}}\backslash B_{L^{i}}} -f_{\tau} \log \left( \frac{f_{\tau}p_i}{p_i} \right) \, d\Haarof{\R^d} \\
        &= \sum_{i \geq 0} \left(\int h(f_{\tau}p_i) \frac{1_{B_{L^{i+1}}\backslash B_{L^{i}}}}{p_i} \, d\Haarof{\R^d} + p_i\log(p_i) \right) \\
        &\leq \sum_{i \geq 0} h(p_i) \leq \sum_{0 \leq i \leq I} h(p_i) + \sum_{i \geq  I} h(L^{-i\alpha})  < \infty,
    \end{align*} having used in the last line that $\log(p_i) \leq 0$ and that $h(x)$ is monotonically decreasing for small $x$ and therefore $h(p_i) \leq h(L^{-i\alpha})$ for $i \geq I$ with $I$ sufficiently large.
    
    Finally, we estimate $D_{\mathrm{KL}}(\rho(q_{\tau_{\kappa}} s_{r_2,a})\, || \, \rho^{-(d + 1)}d\rho)$.  Fix $\eps > 0$ and let $A$ be the event that $\rho(q_{\tau}) \geq \kappa^{(1 + \eps)}.$ By Lemma~\ref{rhotauLDP} there is $\delta > 0$ only depending on $\mu$ and $\eps$ such that $\mathbb{P}[A^c] \leq \kappa^{\delta}$. By \cite{KittleKoglerEntropy}*{Lemma 2.4}, 
    \begin{align*}
        D_{\mathrm{KL}}(\mathcal{L}(\rho(q_{\tau_{\kappa}} s_{r_2,a}))|_A\, || \, \rho^{-(d + 1)}d\rho) &\leq \log\left( \int_{\kappa^{1 + \eps}}^{\infty} \rho^{-(d + 1)} \, d\rho \right) \\ &= \log\left( d^{-1}\kappa^{-d(1 + \eps)} \right) \leq  d(1 + \eps)\log \kappa^{-1}.
    \end{align*}
    To bound $H(\mathcal{L}(q_{\tau}s_{r_2,a})|_{A^c})$, we note that as in \cite{KittleKoglerEntropy}*{Lemma 2.3} it suffices to bound the Shannon entropy of $H(\mathcal{L}(q_{\tau})|_{A^c})$. If $\tau \leq 2\frac{\log \kappa^{-1}}{|\chi_{\mu}|}$, the contribution can be bounded by $\kappa^{\delta}\frac{2\log \kappa^{-1}}{|\chi_{\mu}|}\log K$. By the large deviation principle, when $n \geq 2\frac{\log \kappa^{-1}}{|\chi_{\mu}|}$ it holds that $\mathbb{P}[\tau = n] \leq \alpha^n$ for some $\alpha \in (0,1)$. Therefore the contribution in this case is $\leq \alpha^n n \log K$ where $\alpha \in (0,1)$ is some constant depending on $\mu$. Summing over all $n \geq 2\frac{\log \kappa^{-1}}{|\chi_{\mu}|}$ and using \cite{KittleKoglerEntropy}*{Lemma 2.2}, we conclude that $H(\mathcal{L}(q_{\tau}s_{r_2,a})|_{A^c})$ is bounded and therefore $o_{\mu,\eps}(\log \kappa^{-1})$. As $\eps > 0$ was arbitrary the claim follows. 
\end{proof}

\subsection{Trace Bounds for Stopped Random Walk}

In this subsection we give a precise proof of \eqref{IntroLotsofTrace} using results from \cite{KittleKoglerDimension}. We show further that $s_{i + 1} \geq \kappa^{-3}s_i$ in order to concatenate proper decompositions as defined and discussed in section~\ref{section:decomposition}.

\begin{proposition}\label{LotsOfTrace}
    Let $\mu$ be a finitely supported, contracting on average probability measure on $G = \mathrm{Sim}(\R^d)$ and write $\ell = \dim G = \frac{d(d + 1)}{2} + 1$. Suppose that $S_{\mu} < \infty$ and that $\RWEntropy/|\chi_{\mu}|$ is sufficiently large. Let $S > S_{\mu}$ be chosen such that $S \in [\RWEntropy, 2\RWEntropy]$. Suppose that $\kappa$ is sufficiently small (depending on $\mu$ and $S$) and let $\widehat{m} = \lfloor \frac{S}{100 |\chi_{\mu}|} \rfloor$.

    Then there exist $s_1, s_2, \ldots , s_{\widehat{m}} > 0$ such that for each $i \in [\widehat{m}]$, $$s_i \in (\kappa^{\frac{S}{|\chi_{\mu}|}}, \kappa^{\frac{\RWEntropy}{2\ell|\chi_{\mu}|}})$$ and for each $i \in [\widehat{m}-1]$ $s_{i + 1} \geq \kappa^{-3}s_i$ and $$\sum_{i = 1}^{\widehat{m}}  \tr(q_{\tau_{\kappa}}; s_i) \gg_d  \left(  \frac{\RWEntropy}{|\chi_{\mu}|} \right) \max\left\{ 1, \log \frac{S}{\RWEntropy} \right\}^{-1}.$$ 
\end{proposition}

\begin{proof}
    This follows from \cite{KittleKoglerEntropy}*{Proposition 1.5}. Indeed, we set $r_1 = a^{-1}\kappa^{\frac{S}{|\chi_{\mu}|}}$ and $r_2 = \frac{1}{4}a^{-1}\kappa^{\frac{\RWEntropy}{2\ell|\chi_{\mu}|}}$. Then by \cite{KittleKoglerEntropy}*{Lemma 4.6} provided that $ar_2$ is sufficiently small it holds that $H_a(s_{a,r_2}) = -\frac{h_{\mu}}{2|\chi_{\mu}|}\log \kappa^{-1} + o_{d,a}(\log \kappa^{-1}).$ Therefore it follows from Proposition~\ref{EntropyBetweenScalesIncrease} that $$H_a(q_{\tau_{\kappa}}; r_1'| r_2') \gg_d \frac{h_{\mu}}{|\chi_{\mu}|}\log \kappa^{-1} + o_{\mu,d,S,a}(\log \kappa^{-1})$$ for $r_1' \in [r_1,2r_1]$ and $r_2' \in [r_2/2,2r_2]$.

    Set $A = \kappa^{\frac{\RWEntropy}{4\widehat{m} \ell |\chi_{\mu}|}-\frac{S}{2\widehat{m}|\chi_{\mu}|}}.$ Note that, provided $\RWEntropy/|\chi_{\mu}|$ is sufficiently large, we have $\kappa^{-3} \leq A \leq \kappa^{-50}$. We now apply \cite{KittleKoglerEntropy}*{Proposition 1.5} with $A$ to deduce that there exist $s_1, s_2, \ldots , s_{\widehat{m}} > 0$ with $s_i \in (\kappa^{\frac{S}{|\chi_{\mu}|}}, \kappa^{\frac{\RWEntropy}{2\ell|\chi_{\mu}|}})$ such that for constants $c, C$ only depending on $d$ we have that $$\sum_{i = 1}^{\widehat{m}}  \tr(q_{\tau_{\kappa}}; s_i) \geq \frac{c\frac{h_{\mu}}{|\chi_{\mu}|}\log \kappa^{-1} - CN e^{-\frac{a^2}{4}} + o_{\mu,d,S,a}(\log \kappa^{-1}) }{a^2 \log \kappa^{-1}}$$ for $$ N = \left\lceil \left( \frac{S}{|\chi_{\mu}|} - \frac{h_{\mu}}{2\ell |\chi_{\mu}|} \right) \frac{\log \kappa^{-1}}{ \log 2}  \right\rceil - 1. $$  We take our value of $a$ to be $$a = 2\sqrt{  \log\left(  \frac{4 C}{c \log 2} \frac{S}{\RWEntropy}  \right) }.$$ Then $$CNe^{-\frac{a^2}{4}} \leq c\frac{\RWEntropy}{4 |\chi_{\mu}|}  \log \kappa^{-1}$$ and the claim follows readily.
\end{proof}

    \section{Decomposition of Stopped Random Walk}

    \label{section:decomposition}

In this section Theorem~\ref{MainResult} is proved. We construct samples from $\nu$ in a suitable way in order to bound the order $k$ detail of $\nu$. Given a probability measure $\mu$ on $G = \mathrm{Sim}(\R^d)$ we denote by $\gamma_1, \gamma_2, \ldots$ independent $\mu$-distributed random variables and write $q_n = \gamma_1 \cdots \gamma_n$. Recall that if $x$ is distributed like $\nu$ and $\tau$ is a stopping time, then by Lemma 2.24 from \cite{Kittle2023} the random variable $q_{\tau}x$ is distributed like $\nu$.

As discussed in the outline of proofs, one uses Proposition~\ref{LotsOfTrace} to make a decomposition 
\begin{equation}\label{Decomposition}
    q_{\tau_{\kappa}}x = g_1\exp(U_1)g_2\exp(U_2)\cdots g_n \exp(U_n)x
\end{equation}
with a suitable $\kappa>0$ and integer $n \geq 1$ that satisfies for $1 \leq i \leq n$, 
\begin{equation}\label{TraceInq}
    |U_i| \leq \rho(g_1\cdots g_i)^{-1}r \quad \text{ and } \quad \sum_{i = 1}^n \tr(\rho(g_1\cdots g_i)U_i) \geq Cr^{2}
\end{equation}
for a sufficiently large constant $C$ and a given scale $r > 0$. The definition of $\mathrm{tr}(q_{\tau_{\kappa}}, s_i)$ requires us to work with a $\sigma$-algebra $\mathscr{A}$ and with the conditional trace in \eqref{TraceInq}. As stated in \eqref{Outline:DecompositionFinalGoal}, we need to have \eqref{TraceInq} at $O(\log \log r^{-1})$ many suitable times $\kappa_i$. 

Indeed, in order to deduce \eqref{TraceInq} from Proposition~\ref{LotsOfTrace} we need to combine all the information at the scales $s_1, \ldots , s_{\hat{m}}$. One also needs to ensure that the assumptions from the Taylor-approximation result Proposition~\ref{MainTaylorBound} are satisfied for each scale $s_i$ and that we can apply our $(c,T)$-well-mixing and $(\alpha_0,\theta,A)$-non-degeneracy conditions to deduce that $$\Var(\zeta_i(U_i)) \geq c_1\mathrm{tr}(\rho(g_1\cdots g_i)U_i)I$$ for $c_1$ a constant depending on $d,c,T,\alpha_0,\theta$ and $A$. We will achieve the latter by ensuring that each $g_i$ is a product of sufficiently many $\gamma_i$ so that $g_ix$ is in distribution sufficiently close to $\nu$.

To combine the trace bounds at the various scales while ensuring that the above conditions are satisfied, a theory of decompositions of the form \eqref{Decomposition} will be developed. We call decompositions \eqref{Decomposition} satisfying suitable properties \textit{proper decompositions}. It is important for our purposes to track the amount of variance we can gain from a given proper decomposition, which is a quantity we will call the variance sum and  denote by $V(\mu,n,K,\kappa,A;r)$ (see definition~\ref{VDef} for the various parameters).

In section~\ref{section:ExistenceProper} we will show that there exist proper decompositions that allow us to compare the variance sum $V$ and $\mathrm{tr}$. Proper decompositions can be concatenated in such a way that the variance sum is additive, as is shown in section~\ref{section:Concatenation}. We establish how to convert an estimate on the variance sum $V$ into an estimate for detail in section~\ref{section:FromVariancetoDetail}. The proof of Theorem~\ref{MainResult} culminates in section~\ref{section:ProofMain} combining the previous results. Finally, we establish Theorem~\ref{HighdimMainTheorem} in section~\ref{subsection:HighDimMain}.

\subsection{Proper Decompositions}

\begin{definition}\label{ProperDecomp}
    Let $\mu$ be a probability measure on $G$, let $n,K \in \Z_{\geq 0}$ and let $A, r > 0$ and $r \in (0,1)$. Then a \textbf{proper decomposition} of $(\mu,n,K,A)$ at scale $r$ consists of the following data
    \begin{enumerate}[label = (\roman*)]
        \item $f = (f_i)_{i = 1}^n$ and $h = (h_i)_{i = 1}^n$ random variables taking values in $G$,
        \item $U = (U_i)_{i = 1}^n$ random variables taking values in $\mathfrak{g}$,
        \item $\mathscr{A}_0 \subset \mathscr{A}_1 \subset \ldots \subset \mathscr{A}_n$ a nested sequence of $\sigma$-algebras,
        \item $\gamma = (\gamma_i)_{i = 1}^{\infty}$ be i.i.d. samples from $\mu$ and let $\mathscr{F} = (\mathscr{F}_i)_{i = 1}^{\infty}$ be a filtration for $\gamma$ with $\gamma_{i + 1}$ being independent from $\mathscr{F}_i$ for $i \geq 1$,
        \item stopping times $S = (S_i)_{i = 1}^n$ and $T = (T_i)_{i = 1}^n$ for the filtration $\mathscr{F}$,
        \item $m = (m_i)_{i = 1}^n$ non-negative real numbers,
    \end{enumerate}

    satisfying the following properties:
        
    \begin{enumerate}[label=\textbf{A\arabic*}]
        \item The stopping times satisfy $$S_1 \leq T_1 \leq  S_2 \leq T_2 \leq \ldots \leq S_n \leq T_n,$$ $S_1 \geq K$ as well as $S_i \geq T_{i-1} +K$ and $T_i \geq S_i + K$ for $i \in [n]$, \label{item:stop}
        \item We have $f_1 \exp(U_1) = \gamma_1 \dots \gamma_{S_1}$ and for $2 \leq i \leq n$ we have $f_i \exp(U_i) = \gamma_{T_{i-1} + 1} \cdots \gamma_{S_i}$. Furthermore for each $i$ we have that $f_i$ is $\mathscr{A}_i$-measurable, \label{item:f_idef}
        \item $h_i = \gamma_{S_i + 1} \cdots \gamma_{T_i}$ and $h_i$ is $\mathscr{A}_i$-measurable, \label{item:h_idef}
        \item $\rho(f_i) < 1$ for all $1 \leq i \leq n$, \label{item:rhof_i}
        \item Whenever $|b(h_i)| > A$, we have $U_i = 0$, \label{item:bh_i}
        \item For each $1 \leq i \leq n$ we have $$|U_i| \leq \rho(f_1 h_1 f_2 h_2 \cdots h_{i-1}f_{i})^{-1} r,$$ \label{item:U_ibound}
        \item For each $1 \leq i \leq n$,  we have that $U_i$ is conditionally independent of $\mathscr{A}_n$ given $\mathscr{A}_i$, \label{item:U_1ind1}
        \item The $U_i$ are conditionally independent given $\mathscr{A}_n$, \label{item:U_iind2}
        \item For each $1 \leq i \leq n$, it holds  $$\E\left[   \frac{\var(\rho(f_i) U(f_i) U_i b(h_i)|\mathscr{A}_i)}{\rho(f_1h_1 f_2h_2 \cdots f_{i-1}h_{i-1})^{-2} r^2} \,|\, \mathscr{A}_{i-1}  \right] \geq m_i I.$$ \label{item:m_i}
    \end{enumerate}
\end{definition}

Note that in \ref{item:m_i} by $\var$ we mean the covariance matrix and we are using the ordering given by positive semi-definiteness \eqref{MatrixPartialOrder} and we denote, as in section~\ref{subsection:DerivativeBounds}, by $U_ib(h_i) = \psi_{b(h_i)}(U_i)$.

A proper decomposition as above gives us 
\begin{equation}
    \gamma_1\cdots \gamma_{T_n} = f_1\exp(U_1)h_1f_2 \exp(U_2)h_2 \cdots h_{n-1}f_n\exp(U_n)h_n
\end{equation}

We briefly comment on the various properties of proper decompositions. We use the parameter $K$ and \ref{item:stop} to ensure that each of the $f_ix$ and $h_ix$ for $x \in \R^d$ are close in distribution to $\nu$. Properties \ref{item:rhof_i}, \ref{item:bh_i} and \ref{item:U_ibound} are needed in order to apply Proposition~\ref{MainTaylorBound}. We require \ref{item:U_1ind1} so that we have $\Var(U_i|\mathscr{A}_n) = \mathrm{Var}(U_i|\mathscr{A}_i)$ and in particular the latter is a $\mathscr{A}_i$-measurable random variable. \ref{item:U_iind2} is needed so that $[U_1|\mathscr{A}_n], \ldots ,[U_n|\mathscr{A}_n]$ are independent random variables and therefore we can apply Proposition~\ref{kDetailBerryEssen}.

One works with two sequences of random variables $f$ and $h$ instead of one in order to be able to concatenate proper decompositions as in Proposition~\ref{VarSumAdds}. Indeed, if we had proper decompositions of the form $$\gamma_1\cdots \gamma_{T_n} = g_1\exp(U_1)g_2 \exp(U_2)g_3 \cdots g_{n}\exp(U_n)g_{n + 1}$$ we could show a variant of \eqref{BasicConcat} and all other results on proper decompositions. However we could not prove anything like Proposition~\ref{VarSumAdds}, whose flexible choice of the parameter $M$ is necessary to apply Proposition~\ref{LotsOfTrace}. 

Next, we define the $V$ function mentioned above. The additional parameter $\kappa > 0$ is introduced in order to be able to concatenate the decompositions in a suitable way (Proposition~\ref{VarSumAdds}).

\begin{definition}\label{VDef}
    Given $(\mu,n,K,A)$ and $\kappa,r > 0$ we denote by $$V(\mu,n,K, \kappa ,A; r)$$ the \textbf{variance sum} defined as the supremum for $k = 0,1,2, \ldots , n$ of all possible values of $$\sum_{i = 1}^k m_i$$ for a proper decomposition of $(\mu,k,K,A)$ at scale $r$ with $\rho(f_1h_1 \cdots f_k h_k) \geq \kappa$ almost surely. 
\end{definition}

It is clear that for any $\kappa' > 0$ with $\kappa' \leq \kappa$ we have 
\begin{equation}\label{TrivialVarianceBound}
    V(\mu,n,K, \kappa' ,A; r) \geq V(\mu,n,K, \kappa ,A; r).
\end{equation}

\subsection{Existence of Proper Decompositions}\label{section:ExistenceProper}

We show that for a suitable dependence of the involved parameters, we can construct proper decompositions comparing the variance sum and the trace.

\begin{proposition} \label{InitialDecompositionV}
    Let $d \in \Z_{\geq 1}$ and $c, T, \alpha_0, \theta, A, R > 0$ with $c, \alpha_0 \in (0,1)$ and $T \geq 1$. Then there exists $c_1 = c_1(d,R,c, T,\alpha_0,\theta, A) > 0$ such that the following is true. Let $\mu$ be a contracting on average, $(c, T)$-well-mixing and $(\alpha_0, \theta, A)$-non-degenerate probability measure on $G$ such that $\rho(g) \in [R^{-1}, R]$ for all $g \in \mathrm{supp}(\mu)$. 
    
    Let $\kappa, s > 0$ with $\kappa$ and $s$ sufficiently small (in terms of $\mu$ and $R$). Let $K$ be sufficiently large in terms of $\mu$, $R$, and $T$. Then $$V(\mu,1,K,R^{-3K} \kappa, A ; R^{-K} \kappa s) \geq c_1 \mathrm{tr}(q_{\tau_{\kappa}}; s). $$
\end{proposition}

\begin{proof}
    We construct a proper decomposition with $n = 1$. Let $F$ be uniform on $[0, T] \cap \Z$ and independent of $\gamma$. Let $\underline{S}$ be defined as $$\underline{S} = \inf\{ n \,:\, \rho(q_n) \leq R^{-K-1}  \} + F$$ and let $$S_1 := \inf\{ n \geq \underline{S} \,:\, \rho(\gamma_{\underline{S} + 1}\cdots \gamma_n) \leq \kappa \}.$$ Denote $$\underline{f} = \gamma_1 \cdots \gamma_{\underline{S}} \quad\quad \text{ and } \quad\quad g = \gamma_{\underline{S} + 1} \gamma_{\underline{S} + 2} \cdots \gamma_{S_1}.$$ 

    By the definition of $\mathrm{tr}(q_{\tau_{\kappa}},s)$ there is some $\sigma$-algebra $\mathscr{A}$, some random variable $V$ taking values in $\mathfrak{g}$, some $\mathscr{A}$-measurable random variable $\overline{f}$ taking values in $G$ such that $g = \overline{f}\exp(V)$ with $|V| \leq s$ and 
    \begin{equation}\label{TraceDef}
        \E[\tr(V|\mathscr{A})] \geq \frac{1}{2} s^2 \mathrm{tr}(q_{\tau_{\kappa}},s).
    \end{equation}
    We define $T_1 = S_1 + K$ and set $$h_1 = \gamma_{S_1 + 1} \gamma_{S_1 + 2} \cdots \gamma_{T_1}.$$
    
    Denote
    \begin{equation*}
        U_1 = \begin{cases}
        V & \text{if } |b(h_1)| \leq A, \\ 
        0 & \text{otherwise}
        \end{cases} \quad \text{ and } \quad 
        f_1 = \begin{cases}
        \underline{f}\overline{f} & \text{if } |b(h_1)| \leq A, \\ 
        \underline{f}g & \text{otherwise.}
        \end{cases}
    \end{equation*}
    Furthermore we set $\mathscr{A}_1 = \sigma(\underline{f},f_1, h_1, \mathscr{A})$.

    We have
    \begin{align*}
        R^{-K-3} R^{-T} \kappa \leq \rho(\underline{f}g) \leq R^{-K-1} R^T \kappa.
    \end{align*} In particular, we note that $|U_1| \leq s $ and so providing $\kappa$ and $s$ are sufficiently small in terms of $R$, we have $R^{-K-4}R^{-T}  \kappa \leq \rho( f_1) \leq R^{T-K}\kappa < 1$ for $K$ sufficiently large in terms of $T$. This means that $|U_1| \leq s \leq \rho(f_1)^{-1} R^{T-K}\kappa s$.
    
    Now let $x \in \R^d$ be a unit vector. We wish to show that

    $$\mathbb{E}\left[\var(x \cdot \rho(f_1) U(f_1) U_1 b(h_1)| \mathscr{A}_1) \right] \geq c_1 \mathrm{tr}(q_{\tau_{\kappa}}; s) R^{-2K} \kappa^2 s^{2}.$$
    
    Let $f' = \underline{f}^{-1} f_1$ and let $P_1, \dots , P_d$ be orthogonal eigenvectors of the covariance matrix of $(U_1 b(h_1)|\mathscr{A})$ with eigenvalues $\lambda_1 \geq \dots \geq \lambda_d$. We have 
    \begin{align}
        \MoveEqLeft \var(x \cdot \rho(f_1) U(f_1) U_1 b(h_1)| \mathscr{A}_1) \nonumber \\  \geq &  R^{-2K-8} R^{-2T} \kappa^2  \var(x \cdot U(\underline{f})  U(f')  U_1 b(h_1)| \mathscr{A}_1 ) \nonumber \\
         = &R^{-2K-8} R^{-2T} \kappa^2 \sum_{i=1}^d \left| x \cdot U(\underline{f}) U(f') P_i \right|^2 \lambda_i \nonumber\\
         \geq &R^{-2K-8} R^{-2T}\kappa^2  \left| x \cdot U(\underline{f}) U(f') P_1 \right|^2 \tr(U_1 b(h_1) | \mathscr{A}_1) / d. \label{CompExp}
    \end{align}

    By Proposition~\ref{FirstDerivativeBound} we know that when $b(h_1) \in E_{\theta}(V)$ and $|b(h_1)| \leq A$ we have $$\tr(U_1 b(h_1) | \mathscr{A}_1) \geq \delta \cdot \tr (U_1| \mathscr{A}_1).$$ By our $(\alpha_0, \theta, A)$-non-degeneracy condition and since $\mu^{*n}*\delta_{x}$ converges to $\nu$ exponentially fast (see for example \cite{KittleKoglerTail}*{Lemma 2.2}) we know that providing $K$ is sufficiently large this happens, conditional on $\mathscr{A}$, with probability at least $\frac{1}{2}(1 - \alpha)$. Therefore by \eqref{TraceDef} $$\mathbb{E}[\tr(U_1 b(h_1) | \mathscr{A}_1)] \geq \frac{1}{4}(1 - \alpha) \delta \mathrm{tr}(q_{\tau_{\kappa}}; s) s^2 .$$

    By our $(c, T)$-well-mixing condition we have that providing $K$ is sufficiently large in terms of $\mu$, $$\mathbb{E}\left[ \left| x \cdot U(\underline{f}) U(f') P_1 \right|^2 | \sigma(h_1, \mathscr{A}) \right] \geq c.$$ Clearly $\var(U_1 b(h_1) | \mathscr{A}_1)$ is $\sigma(h_1, \mathscr{A})$-measurable. Therefore, by \eqref{CompExp} and conditioning by $\sigma(h_1, \mathscr{A})$,
    \begin{align*}
         \MoveEqLeft \mathbb{E}\left[\var(x \cdot \rho(f_1) U(f_1) U_1 b(h_1)| \mathscr{A}_1)\right] \\ 
         \geq &  R^{-2K-8} R^{-2T}\kappa^2 d^{-1}  \mathbb{E}\left[\left| x \cdot U(\underline{f}) U(f') P_1 \right|^2 \tr (U_1 b(h_1) | \mathscr{A}_1)  \right] \\ 
         \geq &  R^{-2K-8} R^{-2T}\kappa^2 d^{-1} \mathbb{E}\left[ \mathbb{E}\left[ \left| x \cdot U(\underline{f}) U(f') P_1 \right|^2 \tr (U_1 b(h_1) | \mathscr{A}_1)  \, \Big| \,\sigma(h_1,\mathscr{A}) \right]\right] \\
         \\ 
         \geq & R^{-2K-8} R^{-2T}\kappa^2 d^{-1} \mathbb{E}\left[ \mathbb{E}\left[ \left| x \cdot U(\underline{f}) U(f') P_1 \right|^2  \, \Big| \,\sigma(h_1,\mathscr{A}) \right] \tr (U_1 b(h_1) | \mathscr{A}_1) \right] \\
         \geq & R^{-2K-8} R^{-2T}\kappa^2 d^{-1}c \cdot\mathbb{E}\left[ \tr (U_1 b(h_1) | \mathscr{A}_1) \right] \\
        \geq &  R^{-2K-8} R^{-2T} d^{-1}c \cdot \frac{1}{4}(1 - \alpha) \delta \mathrm{tr}(q_{\tau_{\kappa}}; s)  \kappa^2 s^2 \\
        = &  c_1 \mathrm{tr}(q_{\tau_{\kappa}}; s) R^{-2K} \kappa^2 s^2
    \end{align*}
    where $c_1 = R^{-8} R^{-2T} d^{-1} (1 - \alpha) \delta c / 4$. Since this is true for any unit vector $x \in \R^d$ we have $$\mathbb{E}\left[\frac{\var(\rho(f_1) U(f_1) U_1 b(h_1)| \mathscr{A}_1)}{R^{-2K} \kappa^2 s^2}\right] \geq c_1 \mathrm{tr}(q_{\tau_{\kappa}}; s) I$$ as required. Finally note that
    \begin{align*}
        \rho(f_1 h_1) \geq R^{-1} \rho(\underline{f} g h_1) \geq R^{-1} R^{-K-3}R^{-T} \kappa \cdot R^{-K} = \kappa  R^{-2K-4-T} \geq R^{-3K}\kappa 
    \end{align*}
    providing $K$ is sufficiently large in terms of $T$ and $R$.
\end{proof}

\subsection{Concatenating Decompositions}\label{section:Concatenation}

We note that it is straightforward to show that for any measure $\mu$ and any admissible choice of coefficients, the variance sum is additive
\begin{align}
    \MoveEqLeft V(\mu, n_1 + n_2, K, \kappa_1\kappa_2,A; r) \nonumber \\&\geq V(\mu, n_1, K, \kappa_1,A; r) + V(\mu, n_2, K, \kappa_2,A; \kappa_1^{-1}r). \label{BasicConcat}
\end{align} 

However, in order to use Proposition~\ref{LotsOfTrace} it is necessary to work with different scales $r_1$ and $r_2$ and therefore we show the following proposition.

\begin{proposition} \label{VarSumAdds}
    Let $\mu$ be a probability measure on $G$. Let $R > 1$ be such that $\rho(g) \in [R^{-1}, R]$ for every $g \in \mathrm{supp}(\mu)$. Let $n_1, n_2, K \in \mathbb{Z}_{\geq 0}$ with $n_2, K > 0$ and let $\kappa_1, \kappa_2, r \in (0, 1)$. Let $A > 0$ and let $M \geq R$. Then
    \begin{align*}
        \MoveEqLeft V(\mu,  n_1 + n_2, K, R^{-1} M^{-1} \kappa_1 \kappa_2, A; r) \\&\geq V(\mu, n_1, K, \kappa_1, A; r) + V(\mu,  n_2, K, \kappa_2, A; M \kappa_1^{-1} r).
    \end{align*} 
\end{proposition}

\begin{proof}
    For $j \in \{ 1,2 \}$ let $\gamma_1^{(j)}, \gamma_2^{(j)}, \dots$ be a sequence of i.i.d.\ samples from $\mu$ defined on the probability space $\left(\Omega_{(j)}, \mathscr{F}_{(j)}, \mathbb{P}_{(j)}\right)$. Let $\hat{\gamma}_1, \hat{\gamma}_2, \dots$ be a sequence of i.i.d.\ samples from $\mu$ defined on the probability space $\left(\hat{\Omega}, \hat{\mathscr{F}}, \hat{\mathbb{P}}\right)$. Consider the product probability space $$\left(\Omega, \mathscr{F}, \mathbb{P}\right) = \left( \Omega_1 \times \hat{\Omega} \times \Omega_2, \mathscr{F}_1 \times \hat{\mathscr{F}} \times \mathscr{F}_2, \mathbb{P}_1 \times \hat{\mathbb{P}} \times \mathbb{P}_2 \right).$$ 

    Let $\left(\gamma_i^{(1)}, S_i^{(1)}, T_i^{(1)}, f_i^{(1)}, U_i^{(1)}, h_i^{(1)}, \mathscr{A}_i^{(1)}, m_i^{(1)}\right)$ be a proper decomposition for $(\mu, k_1, K, \kappa_1, A)$ at scale $r$ defined on the probability space $\left(\Omega^{(1)}, \mathscr{F}^{(1)}, \mathbb{P}^{(1)}\right)$ such that $\sum_{i = 1}^{k_1} m_i^{(1)}$ approaches $V(\mu,n_1,K,\kappa_1,A;r)$ and  $$ \rho(f_1^{(1)}h_1^{(1)}\cdots f_{k_1}^{(1)}h_{k_1}^{(1)}) \geq \kappa_1.$$ 

    Given $\omega_1 \in \Omega_1$ and $\hat{\omega} \in \hat{\Omega}$, let $\tau = \tau(\omega_1, \hat{\omega})$ be given by $$\tau = \min \{k \in \mathbb{Z}_{\geq 0}: \rho( f^{(1)}_1 h^{(1)}_1 f^{(1)}_2 h^{(1)}_2 \dots f^{(1)}_{k_1} h^{(1)}_{k_1} \hat{\gamma}_1 \hat{\gamma}_2 \dots \hat{\gamma}_k) < M^{-1} \kappa_1  \} $$ and let $\hat{\rho} = \rho( f^{(1)}_1 h^{(1)}_1 f^{(1)}_2 h^{(1)}_2 \dots f^{(1)}_{k_1} h^{(1)}_{k_1} \hat{\gamma}_1 \hat{\gamma}_2 \dots \hat{\gamma}_\tau)$ such that $$\hat{\rho} \in [M^{-1} R^{-1} \kappa_1 , M^{-1} \kappa_1 ]. $$

    Now given $\omega_1 \in \Omega_1$ and $\hat{\omega} \in \hat{\Omega}$, let $\left(\gamma_i^{(2)}, S_i^{(2)}, T_i^{(2)}, f_i^{(2)}, U_i^{(2)}, h_i^{(2)}, \mathscr{A}_i^{(2)}, m_i^{(2)}\right)$ be a proper decomposition for $(\mu, k_2, K, \kappa_2, A)$ at scale $M\kappa_1^{-1}r$ defined on the probability space $\left(\Omega^{(2)}, \mathscr{F}^{(2)}, \mathbb{P}^{(2)}\right)$ such that $\sum_{i = 1}^{k_2} m_i^{(2)}$ approaches $V(\mu,n_2,K,\kappa_2,A;M\kappa_1^{-1}r)$ and  $$\rho(f_1^{(1)}h_1^{(1)}\cdots f_{k_2}^{(1)}h_{k_2}^{(1)}) \geq \kappa_2.$$ 

    We now concatenate the two decompositions as follows. Let $\gamma_1, \gamma_2, \dots$ be the sequence of random variables on the probability space $\left(\Omega, \mathscr{F}, \mathbb{P}\right)$ defined by $$\gamma_i = \begin{cases}
        \gamma_i^{(1)} & \text{if } i \leq T^{(1)}_{k_1} \\
        \hat{\gamma}_{i - T^{(1)}_{k_1}} & \text{if } i > T^{(1)}_{k_1} \text{ and } i \leq T^{(1)}_{k_1} + \tau \\
        \gamma_{i - T^{(1)}_{k_1} - \tau} & \text{if } i > T^{(1)}_{k_1} + \tau.
    \end{cases}$$
    Clearly these are i.i.d.\ samples from $\mu$. For $i = 1, 2, \dots, k_1 + k_2$ we define $S_i$ by
    \begin{equation*}
        S_i =
        \begin{cases}
            S_i^{(1)} & \text{if } i \leq k_1\\
            S_{i-k_1}^{(2)} + T^{(1)}_{k_1} + \tau & \text{if } i > k_1
        \end{cases}
    \end{equation*}
    and we define $T_i$ analogously. We define $f_i$ by
    \begin{equation*}
        f_i =
        \begin{cases}
            f_i^{(1)} & \text{if } i \leq k_1\\
            \hat{\gamma}_1 \dots \hat{\gamma}_{\tau} f_{1}^{(2)} & \text{if } i = k_1 + 1 \\
            f_{i-k_1}^{(2)} & \text{if } i > k_1 + 1.
        \end{cases}
    \end{equation*}
    We define $U_i$ by
    \begin{equation*}
        U_i =
        \begin{cases}
            U_i^{(1)} & \text{if } i \leq k_1\\
            U_{i-k_1}^{(2)} & \text{if } i > k_1.
        \end{cases}
    \end{equation*}
    and define $h_i$ and $m_i$ analogously. Finally we define $\mathscr{A}_i$ by
    \begin{equation*}
        \mathscr{A}_i =
        \begin{cases}
            \mathscr{A}_i^{(1)} \times \hat{\Omega} \times \Omega^{(2)} & \text{if } i \leq k_1\\
            \mathscr{A}_{k_1}^{(1)} \times \hat{\mathscr{F}} \times \mathscr{A}_{i-k_1}^{(2)} & \text{if } i > k_1.
        \end{cases}
    \end{equation*}
    It is easy to check that $\left(\gamma_i, S_i, T_i, f_i, U_i, h_i, \mathscr{A}_i, m_i\right)$ is a proper decomposition for $(\mu, R, k_1 + k_2, K, R^{-1} M^{-1} \kappa_1 \kappa_2, A)$ at scale $r$ and it holds that $$\sum_{i=1}^{k_1+k_2} m_i = \sum_{i=1}^{k_1} m_i^{(1)} + \sum_{i=1}^{k_2} m_i^{(2)}.$$ Indeed, we note that for $i > k_2$ we have that since $M\kappa_1^{-1} \leq \hat{\rho}^{-1}$,
    \begin{align*}
        |U_i| = |U_{i-k_1}^{(2)}| &\leq \rho(f_1^{(2)}h_1^{(2)}f_2^{(2)}h_2^{(2)}\cdots h_{i-k_1 -1}^{(2)} f_{i-k_1}^{(2)})^{-1} M\kappa_1^{-1}r \\
        &\leq \hat{\rho}^{-1}\rho(f_1^{(2)}h_1^{(2)}f_2^{(2)}h_2^{(2)}\cdots h_{i-k_1 -1}^{(2)} f_{i-k_1}^{(2)})^{-1} r \\
        &= \rho(f_1h_1f_2h_2\cdots h_{i-1} f_{i})^{-1}r.
    \end{align*}
    Similarly, for $i > k_2 + 1$ and using that $\hat{\rho}^2M^2\kappa_1^{-2} \leq 1$,
    \begin{align*}
        \MoveEqLeft \E\left[   \frac{\var(\rho(f_i) U(f_i) U_i b(h_i)|\mathscr{A}_i)}{\rho(f_1h_1 f_2h_2 \cdots f_{i-1}h_{i-1})^{-2} r^2} \,|\, \mathscr{A}_{i-1}  \right] \\
        &= \E\left[   \frac{\var(\rho(f^{(2)}_{i -k_1}) U(f^{(2)}_{i -k_1}) U^{(2)}_{i-k_1} b(h^{(2)}_{i - k_1})|\mathscr{A}_i)}{\hat{\rho}^{-2}\rho(f_1^{(2)}h_1^{(2)}f_2^{(2)}h_2^{(2)}\cdots h_{i-k_1}^{(2)})^{-2} r^2} \,|\, \mathscr{A}_{i-1}  \right] \\
        &\geq  \E\left[   \frac{\var(\rho(f^{(2)}_{i -k_1}) U(f^{(2)}_{i -k_1}) U^{(2)}_{i-k_1} b(h^{(2)}_{i - k_1})|\mathscr{A}_i)}{\hat{\rho}^{-2}\rho(f_1^{(2)}h_1^{(2)}f_2^{(2)}h_2^{(2)}\cdots h_{i-k_1}^{(2)})^{-2} \hat{\rho}^2 M^2 \kappa_1^{-2} r^2} \,|\, \mathscr{A}_{i-1}  \right] \\
        &\geq m_{i - k_1}^{(2)}I.
    \end{align*} The remainder of the properties are straightforward to check. 
\end{proof}

\begin{corollary} \label{ReScaleDecomposition}
    Let $\mu$ be a probability measure on $G$. Let $R > 1$ be such that $\rho(g) \in [R^{-1}, R]$ for every $g \in \mathrm{supp}(\mu)$. Let $n, K \in \mathbb{Z}_{> 0}$ and let $\kappa, r \in (0, 1)$. Let $C, A > 0$ and let $M \geq R$. Then
    \begin{align*}
        V(\mu, n, K, R^{-1} M^{-1} \kappa, A; M^{-1} r) \geq V(\mu, n, K, \kappa, A ; r)
    \end{align*}
\end{corollary}

\begin{proof}
    By Proposition~\ref{VarSumAdds} we have \begin{align*}
        \MoveEqLeft V(\mu,  n, K, R^{-1} M^{-1} \kappa , A; M^{-1}r) \\&\geq V(\mu, 0, K, 1, A; M^{-1}r) + V(\mu,  n , K, \kappa, A;  r).
    \end{align*} 
    and simply note that $V(\mu, 0, K, 1, A; M^{-1}r) = 0$.
\end{proof}

\subsection{From Variance Sum to Bounding Detail}\label{section:FromVariancetoDetail}

\begin{proposition} \label{DecompositionToDetail}
    For every $d\geq 1$ and $A, \alpha > 0$ there is a constants $C = C(d,A,\alpha) > 0$ such that the following is true. Suppose that $\mu$ is a contracting on average probability measure on $G$. Then there is some $c = c(\mu) > 0$ such that whenever $\kappa \leq 1$ and $k,K,n \in \Z_{>0}$ with $K$ and $n$ sufficiently large (in terms of $A, \alpha$ and $\mu$) and $r > 0$ is sufficiently small (in terms of $A, \alpha$ and $\mu$) and $$V(\mu, n,K, \kappa, A; r) > Ck$$ we have $$s_{r}^{(k)}(\nu) < \alpha^k + n\exp(-cK) + C^n\kappa^{-1}r.$$
\end{proposition}

\begin{proof}
    Suppose that $(f, h, U, \mathscr{A}, \gamma, \mathscr{F}, S, T, m)$ is a proper decomposition of $(\mu, n, K, A)$ at scale $r$ such that $\sum_{i=1}^n m_i \geq Ck / 2$ and let $v$ be an independent sample from $\nu$. Let $$I = \{ i \in [1, n] \cap \Z : |b(h_i)| \leq A \}$$ and let $m = |I|$. Enumerate $I$ as $i_1 < i_2 < \cdots < i_m$ and define $g_1, \dots, g_m$ by $g_1 = f_1 h_1 \dots f_{i_1}$ and $g_j = h_{i_{j-1}} f_{i_{j-1}+1} \dots f_{i_j}$ for $2 \leq j \leq m$. Define $\overline{v}$ by $\overline{v} = h_{i_m} f_{i_m +1} \dots h_n v$ and let $V_j = U_{i_j}$. Let $x$ be defined by $$x = g_1 \exp(V_1) \dots g_m \exp(V_m) \overline{v}.$$ Note that $x$ is a sample from $\nu$. Let $\hat{\mathscr{A}}$ be the $\sigma$-algebra generated by $\mathscr{A}_n$ and $v$. Note that the $g_j$ and $\overline{v}$ are $\hat{\mathscr{A}}$-measurable.

    We will bound the order $k$ detail of $x$ by showing that with high probability we can apply Proposition~\ref{MainTaylorBound} to $g_1, \dots, g_m$, $V_1, \dots, V_m$, and $\overline{v}$ and then bound the order $k$ detail of this using Proposition~\ref{kDetailBerryEssen}.

    Let $E$ be the event that $|\overline{v}| \leq 2A$ and that for each $j = 1, \dots, m$ we have $|b(g_j)| \leq 2A$, $\rho(g_j) < 1$ and note that by assumption A6 we have $|V_j| \leq \rho(g_1 \dots g_j)^{-1} r$. We claim that $\mathbb{P}[E^C] \leq \exp(-c_1 K)$ for some $c_1 = c_1(\mu, A) > 0$. Indeed, by Corollary \ref{coro:decay_estimates}, with probability $1 - \exp(-c_1 K)$ we have for all $j = 1, \dots, m$ that $|b(g_j) - b(h_{j-1})| \leq e^{-\delta K}$ for some $\delta > 0$. Therefore, the claim follows by using that $|b(h_{j-1})| \leq A$.
    
    For $j=1, \dots, m$ define $\zeta_j$ by $$\zeta_j = D_u(g_1 \cdots g_j \exp(u) g_{j + 1} \cdots g_m \overline{v})|_{u=0}.$$ By Proposition~\ref{MainTaylorBound} on $E$ we have $$\left|x-g_1 \dots g_m \overline{v} - \sum_{j=1}^m \zeta_j(V_j)\right| \leq C_1^m \rho(g_1 \dots g_m)^{-1} r^2$$ for some $C_1 = C_1(A) > 0$. Clearly the right hand side is at most $C_1^n \kappa^{-1} r^2$. By Lemma \ref{DetailWassersteinBound} this means that on $E$ we have $$s_r^{(k)}(x|\hat{\mathscr{A}}) \leq s_r^{(k)}\left(\sum_{j=1}^m \zeta_j(V_j) | \hat{\mathscr{A}}\right) + C_1^n ed \kappa^{-1} r$$ where $e$ is Euler's number.

    Let $C_3 = C_3(\alpha, d)$ be the constant $C$ from Proposition~\ref{kDetailBerryEssen} with the same values of $\alpha$ and $d$ and let $F$ be the event that $$\sum_{j=1}^m \var \zeta_j(V_j|\hat{\mathscr{A}}) \geq k C_3 I.$$ By Proposition~\ref{kDetailBerryEssen}, using that by \ref{item:U_iind2} the $[V_1|\hat{\mathscr{A}}], \ldots ,[V_m|\hat{\mathscr{A}}]$ are independent almost surely, we have that on $F$ $$s_r^{(k)}\left(\sum_{j=1}^m \zeta_j(V_j) | \hat{\mathscr{A}}\right) \leq \alpha^k.$$ Therefore
    \begin{equation*}
        s_r^{(k)}(x|\hat{\mathscr{A}}) \leq \alpha^k + C_1^n ed \kappa^{-1} r + \mathbb{I}_{E^C \cup F^C}
    \end{equation*}
    and so by the convexity of order $k$ detail we have
    \begin{equation*}
        s_r^{(k)}(x) \leq \alpha^k + C_1^n ed \kappa^{-1} r^2 + \mathbb{P}[E^C] + \mathbb{P}[F^C]. \label{eq:decomposition_order_k_detail_estimate}
    \end{equation*}
    We already have that $\mathbb{P}[E^C] \leq \exp(-c_1 K)$ so it only remains to bound $\mathbb{P}[F^C]$.

    For $i=1, \dots, n$ define $$\hat{\zeta}_i = D_u(f_1 h_1 \cdots h_{i-1}f_i \exp(u) b(h_i))|_{u=0}$$ and let $\underline{F}$ be the event that
    \begin{equation*}
        \left\| \sum_{i=1}^{n} \var \hat{\zeta}_i(U_i|\hat{\mathscr{A}}) - \sum_{j=1}^m \var \zeta_j(V_j|\hat{\mathscr{A}}) \right\| < 1.
    \end{equation*}

    Recall that $C_3 = C_3(\alpha, d)$ is the constant $C$ from Proposition~\ref{kDetailBerryEssen} with the same values of $\alpha$ and $d$ and let $\overline{F}$ be the $\hat{\mathscr{A}}$-measurable event that $\sum_{i=1}^{n} \var(\hat{\zeta}_i(U_i)|\hat{\mathscr{A}}) \geq (C_3 + 1) k I r^2$. Clearly $\underline{F} \cup \overline{F} \subset F$ so it suffices to bound $\mathbb{P}[\underline{F}^C]$ and $\mathbb{P}[\overline{F}^C]$.

    Since $g_1, \ldots , g_m$ and $\overline{v}$ are $\hat{\mathscr{A}}$ measurable, by Lemma~\ref{ZetaVariance} we have for $j=1, \dots, m$ that $\var (\zeta_j(V_j)|\hat{\mathscr{A}})$ is equal to $$ \rho(g_1 \dots g_j)^2\cdot U(g_1 \dots g_j)\psi_{g_{j + 1} \dots g_m \overline{v}} \circ \var (V_j |\hat{\mathscr{A}}) \circ \psi_{g_{j + 1} \dots g_m}^TU(g_1 \dots g_j)^T$$ and that $$\var (\hat{\zeta}_{i_j}(U_{i_j})|\hat{\mathscr{A}}) = \rho(g_1\cdots g_j)^2 \cdot U(g_1 \dots g_j)\psi_{b(h_{i_j})} \circ \var (V_j|\hat{\mathscr{A}}) \circ \psi_{b(h_{i_j})}^TU(g_1 \dots g_j)^T.$$ We also have that $| V_j | \leq \rho(g_1 \cdots g_j)^{-1} r$ almost surely and so consequently $\| \var V_j \| \leq \rho(g_1\cdots g_j)^{-2} r^2$. Therefore by Lemma~\ref{BasicDerivativeProp} (iii), $$\| \var \zeta_j(V_i| \hat{\mathscr{A}}) - \var \hat{\zeta}_{i_j}(U_{i_j}| \hat{\mathscr{A}}) \| \ll_d |b(h_j) - g_{j + 1} \dots g_m \overline{v}|^2 r^2.$$ Furthermore we have that whenever $i \notin I$ that $\var (\hat{\zeta}_{i}(U_{i})|\hat{\mathscr{A}}) = 0$. We may assume without loss of generality that $n \exp(-K \chi_{\mu} / 10) < 1$. This means that, providing $K$ is sufficiently large (in terms of $d$), in order for $\underline{F}$ to occur it is sufficient that for each $j = 1, \dots, m$ we have
    $$|b(h_j) - g_{j + 1} \dots g_m \overline{v}| < \exp(-K \chi_{\mu} / 10) < 1/n.$$
    By Corollary \ref{coro:decay_estimates} this occurs with probability at least $1 - m\exp(-c_2 K)$ for some $c_2 = c_2(\mu) > 0$ and therefore $\mathbb{P}[\underline{F}^C] \leq m\exp(-c_2 K) \leq n\exp(-c_2 K)$.

    Finally we wish to bound $\mathbb{P}[\overline{F}^C]$. Let 
    \begin{align*}
        \Sigma_i &= r^{-2} \var(\hat{\zeta_i}(U_i)|\hat{\mathscr{A}}) =  r^{-2}\var(\hat{\zeta_i}(U_i)|\mathscr{A}_i) \\
        &= r^{-2}\Var(\rho(f_1h_1\cdots h_{i-1}f_i)U(f_1h_1\cdots h_{i-1}f_i)U_ib(h_i)|\mathscr{A}_i))
    \end{align*}
    By construction we know that $$\mathbb{E}[\Sigma_i | \Sigma_1, \dots, \Sigma_{i-1}] \geq m_i I.$$
    We also know that $\| \Sigma_i \| \leq A^2$ since $||\psi_{b(h_i)}|| \leq |b(h_i)| \leq A$. This means that we can apply Lemma \ref{lemma:cramer}. By Lemma \ref{lemma:cramer} we know that providing $C$ is sufficiently large we have
    $$\mathbb{P}\left[ \sum_{i=1}^n \Sigma_i \geq (C_3+1) k I \right] \geq 1 - \exp\left(-c_3 k \sum_{i=1}^n m_i\right)$$
    for some absolute $c_3 > 0$. Providing we choose $C$ to be sufficiently large, we therefore have $\mathbb{P}[\overline{F}^C] \leq \exp(-c_3kC) \leq \alpha^k$  this is less than $\alpha^k$.

    Putting everything together we have $$s_r^{(k)}(x) \leq 2 \alpha^k + n\exp(-c_3K) + ed C_1^n \kappa^{-1} r.$$
    Replacing $\alpha$ be a slightly smaller value gives the required result.
\end{proof}

\subsection{Conclusion of Proof of Theorem~\ref{MainResult}}\label{section:ProofMain}

We finally show a decay in detail under the assumption of Theorem~\ref{MainResult}. What follows is a rather intricate calculation and we refer the reader to the outline of proofs in section~\ref{section:Outline} for intuition and a sketch of the argument.

\begin{proposition}\label{kDetailEstimate}
    Let $d \in \Z_{\geq 1}$ and $c, T, \alpha_0, \theta, A, R > 0$ with $c, \alpha_0 \in (0,1)$ and $T \geq 1$. Then there exists $C = C(d,R,c, T,\alpha_0,\theta, A) > 0$ such that the following is true. Let $\mu$ be a contracting on average, $(c,T)$-well-mixing and $(\alpha_0,\theta, A)$-non-degenerate probability measure on $G$ with $\rho(g) \in [R^{-1}, R]$ for all $g \in \mathrm{supp}(\mu)$ and assume that $$\frac{h_{\mu}}{|\chi_{\mu}|} > C \max\left\{ 1, \left(\log \frac{S_{\mu}}{h_{\mu}} \right)^2  \right\}.$$ Then for all sufficiently small $r > 0$ and all integers $k \in [\log \log r^{-1}, 2\log \log r^{-1}]$ we have that $$s_r^{(k)}(\nu) < (\log r^{-1})^{-10d}.$$
\end{proposition}

\begin{proof}
    We prove this by repeatedly applying Proposition~\ref{InitialDecompositionV} and Proposition~\ref{VarSumAdds} and then applying Proposition~\ref{DecompositionToDetail}. First let $C$ be as in Proposition~\ref{DecompositionToDetail} with $\alpha = \exp(-20 d)$.

    Now let $r > 0$ be sufficiently small and let $K = \exp(\sqrt{\log \log r^{-1}})$. This value of $K$ is chosen so that $K$ grows more slowly than $(\log r^{-1})^{\varepsilon}$ but faster than any polynomial in $\log \log r^{-1}$ as $r \to 0$. Let $S = 2\max \{h_{\mu}, S_{\mu} \}$.
    
    Note that $\frac{h_{\mu}}{2\ell S} < 1$ and for $i = 1, 2, \dots$ let $$\kappa_i = \exp \left( - \frac{|\chi_{\mu}| \log r^{-1}}{2 S} \left(\frac{h_{\mu}}{3 \ell S} \right)^{i-1} \right) = r^{\frac{|\chi_{\mu}|}{2 S}(\frac{h_{\mu}}{3 \ell S})^{i-1}}$$ with $\ell = \dim G$. Then $$\kappa_1 = r^{\frac{|\chi_{\mu}|}{2 S}} \quad \text{ and } \quad \kappa_{i+1} = \kappa_i^{\frac{h_{\mu}}{3 \ell S}}$$ and let $m$ be chosen as large as possible such that $$\kappa_m < \min \{R^{-10K}, 2^{-10K} \}.$$ We require $\kappa_m < R^{-10K}$ later in the proof and assume $\kappa_m < 2^{-10K}$ so that $\kappa_m$ is surely sufficiently small when $r$ is small enough so that we can apply Proposition~\ref{LotsOfTrace}. Note that this gives 
    $$\log \log R +  \sqrt{\log \log r^{-1}} \ll \log \log r^{-1} + m\log \frac{h_{\mu}}{2\ell S}  + \log \frac{\chi_{\mu}}{2S} $$ which is equivalent to $$m \log \left( 4\ell \max\left\{1, \frac{S_{\mu}}{h_{\mu}}\right\} \right)= m\log \frac{2\ell S}{h_{\mu}} \ll_d \log \log r^{-1}$$ and therefore it follows that $$m \asymp \left( \max \left\{ 1, \log \frac{S_{\mu}}{h_{\mu}} \right\} \right)^{-1} \log \log r^{-1} . $$ 

    Now as in Proposition~\ref{LotsOfTrace} let $\hat{m} = \lfloor \frac{S}{100 |\chi_{\mu}|} \rfloor$. For each $i = 1, 2, \dots, m$ let $s_1^{(i)}, s_2^{(i)}, \dots, s_{\hat{m}}^{(i)} > 0$ be the $s_i$ from Proposition~\ref{LotsOfTrace} with $\kappa_i$ in the role of $\kappa$. So $s_j^{(i)} \in (\kappa_i^{\frac{S}{|\chi_{\mu}|}}, \kappa_i^{\frac{h_{\mu}}{2\ell|\chi_{\mu}|}} )$. By Proposition~\ref{InitialDecompositionV} we have for each $j \in [\hat{m}]$,
    \begin{equation*}
        V(\mu,1,K,R^{-3K} \kappa_i, A; R^{-K} \kappa_i s_j^{(i)}) \geq c_1 \mathrm{tr}(q_{\tau_{\kappa_i}}; s_j^{(i)}) 
    \end{equation*}
    for some constant $c_1 = c_1(c, T, \alpha_0, \theta, A, R, d) > 0$. Therefore by Proposition~\ref{VarSumAdds} with $M= R^{-1_{\{ \geq 2 \}}(j)} R^{-3K} \kappa_i s_{j+1}^{(i)} / {s_{j}^{(i)}}$, where we denote $1_{\{ \geq 2 \}}(j) = 1$ whenever $j \geq 2$, we can prove inductively for $j=2, 3, \dots, \hat{m}$ that
    \begin{equation*}
        V(\mu,j,K, R^{-1} R^{-3K} \kappa_i s_1^{(i)} /  s_j^{(i)} , A; R^{-K} \kappa_i s_1^{(i)}) \geq c_1 \sum_{a = 1}^{j} \mathrm{tr}(q_{\tau_{\kappa_i}}; s_j^{(i)}).
    \end{equation*} We have used here that $s_{j + 1}^{(i)}/s_{j}^{(i)} \geq \kappa_i^{-3}$ and so $M \geq R^{-6K}\kappa_i^{-2} \geq R^{10K} \geq R$ since $\kappa_i < R^{-10K}$. By Proposition~\ref{LotsOfTrace} and \eqref{TrivialVarianceBound} we conclude that
    \begin{equation*}
        V(\mu,\hat{m},K, R^{-4K} \kappa_i s_1^{(i)}/ s_{\hat{m}}^{(i)} , A; R^{-K}\kappa_i s_{1}^{(i)}) \geq c_2 \frac{h_{\mu}}{|\chi_{\mu}|} \max \left\{1, \log \frac{S_{\mu}}{h_{\mu}} \right\}^{-1}
    \end{equation*}
    for some constant $c_2> 0$ depending on all of the parameters. 
    
    Note that for $i = 1, 2, \dots, m-1$ when $h_{\mu} / |\chi_{\mu}|$ is sufficiently large we have
    \begin{align*}
        R^{-4K} \kappa_{i + 1} s_1^{(i+1)} / s_{\hat{m}}^{(i)} & \geq R^{-4K} \kappa_{i+1}^{\frac{S}{|\chi_{\mu}|} + 1} \kappa_i^{ - \frac{h_{\mu}}{2 \ell |\chi_{\mu}|}} \\ &\geq R^{-4K} \kappa_i^{\frac{h_{\mu}}{3 \ell |\chi_{\mu}|} - \frac{h_{\mu}}{2 \ell |\chi_{\mu}|} + \frac{h_{\mu}}{3\ell S}}  \\ &\geq R^{-4K}\kappa_i^{-1} \geq R^{6K} \geq R.
    \end{align*}
    as $\kappa_{i + 1} = \kappa_i^{\frac{h_{\mu}}{3 \ell S}}$ and $\kappa_i < R^{-10K}$ and so we may repeatedly apply Proposition~\ref{VarSumAdds} with $$M=R^{-1_{\{ \geq 2 \}}(i)}R^{-4K} \kappa_{i+1} s_1^{(i+1)}/s_{\hat{m}}^{(i)},$$ where we denote $1_{\{ \geq 2 \}}(i) = 1$ whenever $i \geq 2$, to inductively show for $i = 2, 3, \dots, m$ that
    \begin{align*}
       \MoveEqLeft V(\mu,i\hat{m},K,R^{-1}R^{-4K} \kappa_1 s_1^{(1)} / s_{\hat{m}}^{(i)} , A; R^{-K} \kappa_1 s_{1}^{(1)}) \\ &\geq c_2 i \frac{h_{\mu}}{|\chi_{\mu}|} \max \left\{1, \log \frac{S_{\mu}}{h_{\mu}} \right\}^{-1}.
    \end{align*}
    This means using \eqref{TrivialVarianceBound}
    \begin{align*}
        \MoveEqLeft V(\mu,m\hat{m},K,R^{-5K} \kappa_1 s_1^{(1)} / s_{\hat{m}}^{(m)} , A; R^{-K} \kappa_1 s_{1}^{(1)}) \\\geq & c_3  \frac{h_{\mu}}{|\chi_{\mu}|} \max \left\{1, \log \frac{S_{\mu}}{h_{\mu}} \right\}^{-2} \log \log r^{-1}
    \end{align*}
    for some constant $c_3 > 0$ depending on all of the parameters. Since $$R^{-K} \kappa_1 s_{1}^{(1)} \geq  R^{-K} \kappa_1^{\frac{S}{|\chi_{\mu}|} + 1} = R^{-K}  r^{\frac{1}{2} + \frac{|\chi_{\mu}|}{2S}} \geq R^{-K} r^{\frac{1}{2} + \frac{1}{4d}} \geq r$$ for $r$ sufficiently small by Corollary~\ref{ReScaleDecomposition} with $M = R^{-K} \kappa_1 s_{1}^{(1)} r^{-1} \geq R$
    \begin{equation*}
        V(\mu,m\hat{m},K,R^{-5K} r / s_{\hat{m}}^{(m)} , A; r) \geq c_3  \frac{h_{\mu}}{|\chi_{\mu}|} \max \left\{1, \log \frac{S_{\mu}}{h_{\mu}} \right\}^{-2} \log \log r^{-1}.
    \end{equation*}
    
     Note that $1  / s_{\hat{m}}^{(m)} \geq \kappa_m^{-\frac{h_{\mu}}{2 \ell |\chi_{\mu}|}}$ and so in particular providing $h_{\mu} / |\chi_{\mu}|$ is sufficiently large we have $R^{-5K} r / s_{\hat{m}}^{(m)} \geq R^K r$. By Proposition~\ref{DecompositionToDetail} provided $$\frac{h_{\mu}}{|\chi_{\mu}|} \max \left\{1, \log \frac{S_{\mu}}{h_{\mu}} \right\}^{-2} \geq 2 c_3^{-1} C$$ we deduce
     \begin{equation*}
         s_r^{(k)}(\nu) \leq \exp(-20 d k) + m \hat{m} \exp(-c_4 K) + R^{-K} C^{m \hat{m}}
     \end{equation*}
     for some constant $c_4=c_4(\mu)>0$ and $k \in [\log \log r^{-1}, 2\log \log r^{-1}]$. Since $m \hat{m} \ll_{\mu} \log \log r^{-1}$ it is easy to see that $$m \hat{m} \exp(-c_4 K) + R^{-K} C^{m \hat{m}} < \left( \log r^{-1} \right)^{-20d}$$ whenever $r>0$ is sufficiently small (in terms of $\mu$). Since $k \geq \log \log r^{-1}$ we have that $\exp(-20dk) \leq \left( \log r^{-1} \right)^{-20 d}$. Overall this means that provided $r>0$ is sufficiently small (in terms of $\mu$) we have
     \begin{equation*}
         s_r^{(k)}(\nu) < \left( \log r^{-1} \right)^{-10 d}. \qedhere
     \end{equation*}
    
\end{proof}

We deduce the main theorem from Proposition~\ref{kDetailEstimate}.

\begin{proof}(of Theorem~\ref{MainResult})
    We combine Proposition~\ref{kDetailEstimate} with Lemma~\ref{kto1DetailBound}. Given $r > 0$ sufficiently small, let $k = \frac{3}{2} \log \log r^{-1}$, $a = r/\sqrt{k}$ and $b = rk^k$.

    Suppose that $s\in [a,b]$ and note that then $k \in [\log \log s^{-1}, 2\log\log s^{-1}]$ and $\frac{1}{2} \log r^{-1} < \log s^{-1}$ for $r$ sufficiently small and therefore by Proposition~\ref{kDetailEstimate} $$s_{s}^{(k)}(\nu) < (\log s^{-1})^{-10d} < 2^{10d}  (\log r^{-1})^{-10d}.$$ By Lemma~\ref{kto1DetailBound} it follows that 
    \begin{align*}
        s_r(\nu) \leq Q'(d)^{k-1}(2^{10d} (\log r^{-1})^{-10d} + k^{-k}),
    \end{align*}
    which is easily shown to be $\leq (\log r^{-1})^{-2}$ for $r$ sufficiently small. Indeed, recall that $Q'(d) \leq ed^{-1/2} \leq e$ for all $d \geq 1$ and therefore $Q'(d)^{k} \leq (\log(r^{-1}))^{e}$. 

    This concludes the proof of the main theorem of this paper. 
\end{proof}

\subsection{Proof of Theorem~\ref{HighdimMainTheorem}} \label{subsection:HighDimMain}

In this section we show how to work with the entropy and separation rate on $O(d)$ instead of the one on $G$. Recall that for a measure $\mu$ on $G$ the measure $U(\mu)$ on $O(d)$ is the pushforward of $\mu$ under the map $g \mapsto U(g)$. We then denote for a finitely supported $\mu$ by $h_{U(\mu)}$ and $S_{U(\mu)}$ the analogously defined Shannon entropy and separation rate of $U(\mu)$. 

The proof of Theorem~\ref{HighdimMainTheorem} is analogous to the proof of Theorem~\ref{MainResult}. The only point where a slightly different argument is needed is the following version of Proposition~\ref{EntropyBetweenScalesIncrease}. The remainder of the proof is verbatim to the proof of Theorem~\ref{MainResult} with only changing the notation of $h_{\mu}$ to $h_{U(\mu)}$ and $S_{\mu}$ to $S_{U(\mu)}$.

\begin{proposition}\label{UEntropyBetweenScalesIncrease}
    Let $\mu$ be a finitely supported, contracting on average probability measure on $G$.  Suppose that $S_{U(\mu)} < \infty$ and  that $h_{U(\mu)}/|\chi_{\mu}|$ is sufficiently large. Let $S > S_{U(\mu)}$, $\kappa > 0$ and $a \geq 1$ and suppose that $0 < r_1 < r_2 < a^{-1}$ with $r_1 < \exp( -S \log(\kappa^{-1})/|\chi_{\mu}|)$. Then as $\kappa \to 0$,
    \begin{align*}
          H_a(q_{\tau_{\kappa}}; r_1|r_2) \geq \left( \frac{h_{U(\mu)}}{|\chi_{\mu}|} - d -1 \right) \log \kappa^{-1} + H(s_{r_2,a}) + o_{\mu,d,S,a}(\log \kappa^{-1}).
    \end{align*}
\end{proposition}

\begin{proof}
    The proof is similar to the one of Proposition~\ref{EntropyBetweenScalesIncrease} thus we only provide a sketch. Lemma~\ref{EntScaleGrowth2} still holds and therefore we only need to show that 
    \begin{equation}\label{UEntScaleGrowth1}
        H_a(q_{\tau_{\kappa}}; r_1) \geq \left( \frac{h_{U(\mu)}}{|\chi_{\mu}|} -1 \right)\log \kappa^{-1} + o_{\mu,d,S,a}(\log \kappa^{-1}),
    \end{equation}
    where $H_a(q_{\tau_{\kappa}}; r_1) = H(q_{\tau_{\kappa}}s_{r_1,a}) - H(s_{r_1,a})$. To show \eqref{UEntScaleGrowth1} we apply \cite{KittleKoglerEntropy}*{Lemma 2.6} with $X = \R_{>0} \times O(d) \times \R^d$ and $\Phi : G \to X, g \mapsto (\rho(g), U(g), b(g))$ and $\Haarof{X}$ the product measure on $X$ as used in Lemma~\ref{EntScaleGrowth2}. Indeed, as $\Phi_{*}\Haarof{G} = \Haarof{X}$, it follows by \cite{KittleKoglerEntropy}*{Lemma 2.6} that $H(q_{\tau_{\kappa}}s_{r_1,a}) = D_{\mathrm{KL}}(q_{\tau_{\kappa}}s_{r_1,a}\,||\, \Haarof{G}) = D_{\mathrm{KL}}(\Phi_{*}q_{\tau_{\kappa}}s_{r_1,a}\,||\, \Haarof{X})$ and therefore, since $\Haarof{X}$ is a product measure,
    \begin{align*}
    H(q_{\tau_{\kappa}}s_{r_1,a}) &= D_{\mathrm{KL}}(U(q_{\tau_{\kappa}}s_{r_1,a})\,||\, dU) + D_{\mathrm{KL}}(\rho(q_{\tau_{\kappa}}s_{r_1,a}) \,||\, \rho^{-(d + 1)}d\rho) \\ &+ D_{\mathrm{KL}}(b(q_{\tau_{\kappa}}s_{r_1,a})\,||\, db).
    \end{align*}
    As in Proposition~\ref{EntropyBetweenScalesIncrease} one shows that $$D_{\mathrm{KL}}(U(q_{\tau_{\kappa}}s_{r_1,a})\,||\, dU) \geq \frac{h_{U(\mu)}}{|\chi_{\mu}|}\log \kappa^{-1} + D_{\mathrm{KL}}(U(s_{r_1,a})\,||\, dU) + o_{\mu,d,S,a}(\log \kappa^{-1}).$$ On the other hand, $$D_{\mathrm{KL}}(\rho(q_{\tau_{\kappa}}s_{r_1,a}) \,||\, \rho^{-(d + 1)}d\rho) \gg D_{\mathrm{KL}}(\rho(s_{r_1,a}) \,||\, \rho^{-(d + 1)}d\rho)$$ and $$D_{\mathrm{KL}}(b(q_{\tau_{\kappa}}s_{r_1,a})\,||\, db) \gg D_{\mathrm{KL}}(b(s_{r_1,a})\,||\, db)$$ and note that by \cite{KittleKoglerEntropy}*{Lemma 2.5}, $$D_{\mathrm{KL}}(U(s_{r_1,a})\,||\, dU) + D_{\mathrm{KL}}(\rho(s_{r_1,a}) \,||\, \rho^{-(d + 1)}d\rho) + D_{\mathrm{KL}}(b(s_{r_1,a})\,||\, db) \geq H(s_{r_1,a}).$$ All these estimates combined imply the claim. 
\end{proof}

    \section{Well-Mixing and Non-Degeneracy}

    \label{section:mixnondeg}

In this section we study $(c,T)$-well mixing as well as $(\alpha_0, \theta, A)$-non-degeneracy. The goal of this section is prove Proposition~\ref{UnifNonDeg} and Proposition~\ref{UnifNonDegContAv}. We treat $(c,T)$-well-mixing in section~\ref{MixingSection} and show that we have uniform results as long as $U(\mu)$ is fixed. In section~\ref{section:non-deg} we conclude the proofs of Proposition~\ref{UnifNonDeg} and Proposition~\ref{UnifNonDegContAv} by proving strong results on non-degeneracy. 

\subsection{$(c, T)$-well-mixing}\label{MixingSection}

In this subsection, we establish that we have uniform $(c,T)$-well-mixing whenever $U(\mu)$ is fixed and irreducible and note that well-mixing is continuous in pertubations of $\mu_U$. Moreover, we prove that $(c,T)$ can be taken to be uniform when we know a lower bound on the spectral gap of $U(\mu)$. We start with a preliminary lemma that will also be used in section~\ref{section:non-deg}. Throughout this section and next we denote by $\Haarof{H}$ the Haar probability measure on a closed subgroup $H\subset O(d)$ and by $I \in O(d)$ the identity matrix. 

\begin{lemma}\label{IrreducibleTrick}(Schur-type Lemma)
    Suppose that $d\geq 1$ and that $H$ is an closed, irreducible subgroup of $O(d)$ and let $V$ be a uniform random variable on $H$. Let $B$ be a random variable independent from $V$ taking values in $\R^d$. Then $VB$ has mean zero and covariance matrix of the form $\lambda I$ for some $\lambda \geq 0$.
\end{lemma}

\begin{proof}
    For $h \in H$ the random variables $h VB$ and $VB$ have the same law. This means that the mean of $VB$ is invariant under $H$ and so since $H$ is irreducible it must be zero. Moreover the covariance matrix $M$ of $VB$ is invariant under conjugation by elements of $H$. Since $M$ is symmetric positive definite, it has an eigenvector $v$ and therefore $Mv = \lambda v$ and $Mhv = hMv = \lambda h v$ for some $\lambda \geq 0$ and all $h\in H$. Since $H$ is irreducible it therefore follows that $M  = \lambda I$ as claimed. 
\end{proof}

\begin{lemma}\label{Ufixeduniformmixing}
    Let $\mu_U$ be a finitely supported probability measure on $O(d)$ such that $\mathrm{supp}(\mu_U)$ acts irreducibly on $\R^d$. Then there exists $T = T(\mu_U)$ only depending on $\mu_U$ such that every finitely supported probability measure $\mu$ on $G$ with $U(\mu)$ is $(\frac{1}{2d},T)$-well-mixing.
\end{lemma}

\begin{proof}
    Let $H \subset O(d)$ be the closure of the group generated by $\mathrm{supp}(\mu_U)$. Then $H$ is compact and let $\Haarof{H}$ be the Haar probablility measure on $G$ and denote by $V$ a uniform random variable on $H$. We first claim that for all unit vectors $x$ and $y$ in $\R^d$ we have 
    \begin{equation}\label{HExpectation}
        \E[|x\cdot Vy|^2] = \frac{1}{d}.
    \end{equation}
    Indeed, we can view $y$ as a random variable independent from $V$ and therefore, by Lemma~\ref{IrreducibleTrick}, the random variable $Vy$ has mean zero and covariance matrix $\lambda I$. Moreover, since $\E[|Vy|^2] = d\lambda = 1$ it follows that $\lambda = \frac{1}{d}$ and therefore \eqref{HExpectation} holds. 
    
    Let $F$ be a uniform random variable on $[0,T]$. Then $U(q_F)$ is distributed as 
    \begin{equation}\label{Fdist}
        \frac{1}{T + 1} \sum_{i = 0}^T \mu_U^{*i}.
    \end{equation}
    We claim that $\eqref{Fdist}$ converges as $T\to \infty$ to $\Haarof{H}$ in the weak$^{*}$-topology. Indeed, we note that any weak$^{*}$-limit $m$ of \eqref{Fdist} is $\mu_U$-stationary and, upon performing an ergodic decomposition, we may assume without loss of generality that $m$ is in addition ergodic. As this is equivalent to the measure being extremal, we conclude that $m$ is invariant under the group generated by $\mathrm{supp}(\mu_U)$ and therefore also by $H$, implying that $m = m_H$.

    Finally, we just choose $c = \frac{1}{2d}$ and $T$ sufficiently large depending on $\mu_U$ such that $\eqref{Fdist}$ is sufficiently close in distribution to $\Haarof{H}$ and therefore $\E[|x\cdot U(q_F)y|^2] \geq \frac{1}{2d}$ for all unit vectors $x,y \in \R^d$, implying the claim.
\end{proof}

It is straightforward to show that $(c,T)$-well-mixing is continuous in measure $\mu_U$. We state the following version with the $L^3$-Wasserstein distance in order to prove our non-degeneracy results in the next subsection.

\begin{lemma}\label{MixingContinuity}
    Let $\mu_U$ be a probability measure on $O(d)$ and suppose that $\mu_U$ is $(c, T)$-well-mixing for some $c, T > 0$. Then for every $\varepsilon \in (0, c)$ there is some $\delta > 0$ such that if $\mathcal{W}_3(U(\mu), \mu_U) < \delta$, for $\mu$ a measure on $G$, then $\mu$ is $(c - \varepsilon, T)$-well-mixing.
\end{lemma}

\begin{proof}
    This follows as small pertubations in the matrices, lead to small pertubations in $|x\cdot U(q_F)y|$ for all $x,y \in \mathbb{R}^d$ with unit norm.
\end{proof}

For a closed subgroup $H \subset O(d)$ and a probability measure $\mu_U$ supported on $H$ we denote, as defined in \eqref{SpectralGapDef}, by $\mathrm{gap}_H(\mu_U)$ the $L^2$-spectral gap of $\mu_U$ on $L^2(H)$. We aim to show uniform well-mixing as long as $\mathrm{gap}_H(\mu_U) \geq \eps$ independent of the subgroup $H$. To do so, we first show that we have uniform convergence in the Wasserstein distance with a rate only depending on $\eps$ and $d$.

\begin{lemma}\label{GapWassersteinConvergence}
    Let $d\geq 1, \eps \in (0,1)$ and let $\mu_U$ be a probability measure on $O(d)$. Assume that $\mathrm{gap}_H(\mu_U) \geq \eps$ for $H$ the subgroup generated by the support of $\mu_U$. Then for $n\geq 1$ $$\mathcal{W}_1(\mu_U^{*n}, \Haarof{H}) \ll_d (1-\eps)^{\alpha n}$$ for $\alpha = (1 + \frac{1}{2}\dim O(d))^{-1}$.
\end{lemma}

\begin{proof}
    We consider the metric $d(g_1, g_2) = ||g_1 - g_2||$ on $O(d)$ for $||\circ||$ the operator norm and note that it is bi-invariant and restricts to $H$. Denote by $B^H_{\delta}(h)$ for $h\in H$ and $\delta > 0$ the $\delta$-ball around $h\in H$ and denote $$P_{\delta} = \frac{1_{B^H_{\delta}(e)}}{\Haarof{H}(B^H_{\delta}(e))}.$$ For $\delta \in (0,1)$ we note that $\Haarof{H}(B^H_{\delta}(e)) \gg_d \delta^{\dim O(d)}$ for an implied constant depending only on $d$ and therefore $||P_{\delta}||_2  \ll_d \delta^{-(\dim O(d))/2}$. Also we note that for $h \in H$ we have $(\mu^{*n}*P_{\delta})(h) = \frac{\mu^{*n}(B_{\delta}^H(h))}{\Haarof{H}(B^H_{\delta}(e))}.$ By the triangle inequality, $$\mathcal{W}_1(\mu^{*n}, \Haarof{H}) \leq \mathcal{W}_1(\mu^{*n}, \mu^{*n}*P_{\delta}) + \mathcal{W}_1(\mu^{*n}*P_{\delta}, \Haarof{H}).$$ Note $\mathcal{W}_1(\mu^{*n}, \mu^{*n}*P_{\delta}) \ll_d \delta$ and since $H$ is compact, 
    \begin{align*}
        \mathcal{W}_1(\mu^{*n}*P_{\delta}, \Haarof{H}) &\ll_d ||\mu^{*n}*P_{\delta} - 1||_1 \\ &\leq ||\mu^{*n}*P_{\delta} - 1||_2 \\ & \leq (1-\eps)^n ||P_{\delta}||_2 \ll_d (1-\eps)^n \delta^{-(\dim O(d))/2}.
    \end{align*} To conclude, if follows $$\mathcal{W}_1(\mu^{*n}, \Haarof{H}) \ll_d \delta + (1-\eps)^n \delta^{-(\dim O(d))/2}.$$ Therefore setting $\delta = (1-\eps)^{\alpha n}$ for $\alpha = (1 + \frac{1}{2}\dim O(d))^{-1}$ implies the claim.
\end{proof}

\begin{lemma}\label{UnifWellMixSpecGap}
    Let $d\geq 1, \eps \in (0,1)$ and let $\mu_U$ be a probability measure on $O(d)$. Assume that $\mathrm{gap}_H(\mu_U) \geq \eps$ for $H$ the subgroup generated by the support of $\mu_U$. Then there exists $T = T(d,\eps)$ only depending on $d$ and $\eps$ such every probability measure $\mu$ on $G$ with $U(\mu) = \mu_U$ is $(\frac{1}{2d},T)$-well-mixing.
\end{lemma}

\begin{proof}
    The proof is similar to the one of Lemma~\ref{Ufixeduniformmixing} and recall the notation used in it. Consider a list of tuples of unit vectors $(x_1,y_1), \ldots ,(x_m, y_m)$ such that for every two unit vectors $x$ and $y$ in $\R^d$ there is some $i \in [m]$ such that $$\sup_{U \in O(d)}\Big| \, |x\cdot U y|^2 - |x_i \cdot U y_i|^2 \Big| < \frac{1}{4d}.$$ Such a list of tuples exists as the action of $O(d)$ on $\mathbb{S}^{d-1}\subset \R^d$ is uniformly continuous. We claim that for $T$ large enough depending only on $\eps$ we have for all $i \in [m]$ that $$\E[|x_i \cdot U(q_F)y_i|^2] \geq \frac{3}{4d}.$$ Indeed, we note that for $h_1, h_2 \in H$ we have $$\big|\,|x_i \cdot h_1y_i|^2 - |x_i \cdot h_2y_i|^2\,\big| \leq \big|\,|x_i \cdot h_1y_i| + |x_i \cdot h_2y_i|\,\big| \cdot \big|\,|x_i \cdot h_1y_i| - |x_i \cdot h_2y_i|\,\big| \leq 2||h_1 - h_2||.$$ Thus it follows that $$\E[|x_i \cdot Vy_i|^2 - |x_i \cdot U(q_n)y_i|^2] \leq 2\mathcal{W}_1(\mu^{*n},\Haarof{H})$$ and the claim follows by Lemma~\ref{GapWassersteinConvergence}. This concludes the proof as for all $x$ and $y$ we have $$\E[|x \cdot U(q_F)y|^2] \geq \sup_{i \in [m]} \E[|x_i \cdot U(q_F)y_i|^2] - \frac{1}{4d} \geq \frac{1}{2d}.$$
\end{proof}

\subsection{$(\alpha_0, \theta, A)$-non-degeneracy}\label{section:non-deg}

In order to state our results on $(\alpha_0, \theta, A)$-non-degeneracy it is useful to understand that we can translate and rescale our generating measures, without changing any of the fundamental properties. It is also beneficial to replace $\mu$ by $\frac{1}{2}\delta_e + \frac{1}{2}\mu$ and we show in the following lemma that these changes do not change our self-similar measure or any of the relevant constants in a fundamental way. 

\begin{lemma}\label{ConjugationLemma}
    Let $\mu = \sum_{i} p_i \delta_{g_i}$ be a contracting on average probability measure on $G$ with self-measure $\nu$. Let $h \in G$ and consider the measures $$\mu_h = \sum_i p_i \delta_{hg_ih^{-1}} \quad \text{ and } \quad \mu_h' = \frac{1}{2}\delta_e + \frac{1}{2}\mu_h.$$ Then the following properties hold: 
    \begin{enumerate}[label = (\roman*)]
        \item $h_{\mu} = h_{\mu_h} = 2h_{\mu_h'}$,
        \item $\chi_{\mu} = \chi_{\mu_h} = 2\chi_{\mu_h'}$,
        \item $S_{\mu} = S_{\mu_h} = S_{\mu_h'}$,
        \item $\mathrm{gap}_H(\mu) = \mathrm{gap}_{U(h)HU(h)^{-1}}(\mu_h) = 2\mathrm{gap}_{U(h)HU(h)^{-1}}(\mu_h')$,
        \item $\mu_h$ and $\mu_{h'}$ have $h\nu$ as self-similar measure.
    \end{enumerate}
\end{lemma}

\begin{proof}
    As conjugation is a bijection on $G$ and by using \cite[Lemma 6.8]{HochmanSolomyak2017}, (i) follows. Moreover, (ii) follows since $\rho(hg_ih^{-1}) = \rho(g_i)$ and (iv) follows similarly. To show (iii) note $S_{\mu_h} = S_{\mu_h'}$ since by the triangle inequality $d(g,h) \leq d(g,e) + d(e,h)$ for all $g,h \in G$. To show that $S_{\mu} = S_{\mu_h}$, set $$A = \min_{g_1, g_2 \in \mathrm{supp}(\mu), g_1 \neq g_2} d(g_1, g_2)$$ and note that there is a constant $C_h$ depending on $h$ such that $d(hg_1h^{-1}, hg_2h^{-1}) \leq C_h d(g_1, g_2)$ for $d(g_1, g_2) \leq A $. Thus it holds that 
    \begin{align*}
        S_{\mu_h} &= \limsup_{g_1,g_2 \in S_n, g_1 \neq g_2} -\frac{1}{n}\log d(hg_1h^{-1}, hg_2h^{-1}) \\ &\leq \limsup_{g_1,g_2 \in S_n, g_1 \neq g_2} -\frac{1}{n}\log C_h d(g_1, g_2) = S_{\mu}
    \end{align*}
    Applying the same argument to conjugation by $h^{-1}$ implies the claim. Finally, we note that $\mu_h$ and $\mu_h'$ have the same self-similar measure and it holds that $$h\nu = h\sum_i p_i g_i \nu = \sum_i p_i hg_i h^{-1} h\nu$$ and therefore $h\nu$ is the self-similar measure of $\mu_h$ and $\mu_h'$.
\end{proof}
 
In particular, it follows that the self-similar measure of $\mu$ is absolutely continuous if and only if the one of $\mu_h$ or $\mu_{h'}$ is and all of the relevant quantities are the same up to a factor of $2$. 

To give an idea of the proof of the main results in this subsection, we first discuss how to show that real Bernoulli convolutions $\nu_{\lambda}$ are uniformly non-degenerate. Indeed, we distinguish between $\lambda \geq \lambda_0$ and $\lambda \leq \lambda_0$ for some $\lambda_0$ sufficiently close to $1$. Note that $\nu_{\lambda}$ is supported on $[-(1-\lambda)^{-1}, (1-\lambda)^{-1}]$ and thus when $\lambda \leq \lambda_0$ one easily checks uniform non-degeneracy depending only on $\lambda_0$ using for example that Bernoulli convolutions are symmetric around $0$. In the case $\lambda \geq \lambda_0$ the idea is to use an $L^p$-version of the Berry-Essen Theorem as stated in Theorem~\ref{theo:l_p_wasserstien_p}. Indeed, by rescaling $\nu_{\lambda}$ by $\frac{1}{\sqrt{1-\lambda^2}}$ to $\nu_{\lambda}'$ so that it has unit variance it follows from Theorem~\ref{theo:l_p_wasserstien_p} that $\mathcal{W}_3(\nu_{\lambda}, \mathcal{N}(0,1))) \ll \frac{(1-\lambda^2)^{1/2}}{(1-\lambda^3)^{1/3}} \to 0$ as $\lambda \to 1$. The latter then implies the claim by Lemma~\ref{CloseToNormalNonDegen} using that $\mathcal{W}_1 \leq \mathcal{W}_3$. It is necessary to work with $\mathcal{W}_3$ as a brief calculation show that the above calculation would not work with $\mathcal{W}_1$ or $\mathcal{W}_2$.

Our results will be deduced from suitable results in the case when $\mu$ has a uniform contraction ratio and then in the general case from comparing our given measure with a self-similar measure with uniform contraction ratio. We now state the main proposition of this section.

\begin{proposition} \label{prop:non_degen}
    Let $d\geq 1$, $\eps > 0$ and let $\mu_U$ be an irreducible probability measure on $O(d)$. Then there is  $\delta > 0$, $\tilde{\rho} \in (0,1)$ and some $(\alpha_0,\theta, A)$ depending on $d, \eps$ and $\mu_U$ such that the following is true. Let $\mu = \sum_{i = 1}^k p_i \delta_{g_i}$ be a contracting on average probability measure on $G$ satisfying  $\mathcal{W}_3(U(\mu), \mu_U) < \delta$ and $p_i \geq \eps$ for $1 \leq i \leq k$. Suppose further that there is some $\hat{\rho} \in (\tilde{\rho},1)$ such that $$\frac{\mathbb{E}_{\gamma \sim \mu} [|\hat{\rho} - \rho(\gamma) |]}{1 - \mathbb{E}_{\gamma \sim \mu}[\rho(\gamma)]} < 1 - \varepsilon.$$
    Then there is some $h \in G$ with $U(h) = I$ such that the conjugate measure $\mu_h' = \frac{1}{2}\delta_e + \frac{1}{2}\sum_i p_i \delta_{hg_ih^{-1}}$  is $(\alpha_0, \theta, A)$-non-degenerate.

    Moreover, if in addition $\mathrm{gap}_{H}(U_\mu) \geq \eps$, for $H$ the closure of the subgroup generated by $\supp \mu$, then $\tilde{\rho}$ and $(\alpha_0,\theta,A)$ can be made uniform in $d$ and $\eps$.
\end{proposition}

We first show how to deduce from Proposition~\ref{prop:non_degen} the two propositions \ref{UnifNonDeg} and \ref{UnifNonDegContAv} from section~\ref{subsection:MainResult}. To do so we first state the following lemma.

\begin{lemma} \label{lemm:distance_estimate}
	Suppose $x_1 < x_2$ and let $X$ be a real-valued random variable such that $X \leq x_2$ almost surely and $\mathbb{P}[X \leq x_1] \geq 1/2 +p$ for some $p > 0$. Then $$\mathbb{E}[|X-x_1|] \leq \mathbb{E}[| X - x_2 |] - 2 p (x_2 - x_1).$$ 
\end{lemma}

\begin{proof}
	Let $X_1$ and $X_2$ have the same law as $X$ and be coupled such that at least one of them is at most $x_1$ almost surely. Let $A$ be the event that both $X_1$ and $X_2$ are at most $x_1$. Noting that $A$ has probability at least $2p$ we compute
	\begin{align*}
	\mathbb{E} [ |X_1 - x_1| + |X_2 - x_1|] &= \mathbb{E} [ (|X_1 - x_1| + |X_2 - x_1|) \mathbb{I}_{A^C}] \\ &+ \mathbb{E} [ (|X_1 - x_1| + |X_2 - x_1|) \mathbb{I}_{A}]\\
	&\leq \mathbb{E} [ (|X_1 - x_2| + |X_2 - x_2|) \mathbb{I}_{A^C}] \\ &+ \mathbb{E} [ (|X_1 - x_2| + |X_2 - x_2| - 2 (x_2 - x_1)) \mathbb{I}_{A}]\\
	&\leq \mathbb{E} [ |X_1 - x_2| + |X_2 - x_2|] - 4p (x_2 - x_1).
	\end{align*}
	The result follows.
\end{proof}

We now prove Proposition~\ref{UnifNonDeg} and Proposition~\ref{UnifNonDegContAv}.

\begin{proof}[Proof of Proposition~\ref{UnifNonDeg}.]
	Let $\gamma_1, \gamma_2, \dots$ be i.i.d.\ samples from $\mu$. Let $p_{\text{min}}$ be the smallest of the $p_1, \dots, p_k$ and let $\rho_{\text{min}}$ be the smallest of the $\rho(g_1), \dots, \rho(g_k)$. Clearly
	\begin{equation*}
		\mathbb{P}[\rho(\gamma_1 \dots \gamma_n) \leq \rho_{\text{min}}] \geq 1 - (1 - p_{\text{min}})^n.
	\end{equation*}
	In particular there is some $n$ depending only on $\eps$ such that this is at least $3/4$. Note that by Lemma \ref{lemm:distance_estimate} with $x_1 = \rho_{\min}$ and $x_2 = 1$ and $p = \frac{1}{4}$ we have
	\begin{align*}
		\frac{\mathbb{E}[|\rho(\gamma_1 \cdots \gamma_n) - \rho_{\text{min}} |]}{1 - \mathbb{E}[\rho(\gamma_1 \dots \gamma_n)]} &\leq \frac{1 - \mathbb{E}[\rho(\gamma_1 \dots \gamma_n)] - (1 - \rho_{\text{min}}) / 2}{1 - \mathbb{E}[\rho(\gamma_1 \dots \gamma_n)]}\\
		&= 1 - \frac{1 - \rho_{\text{min}}}{2 (1 - \mathbb{E}[\rho(\gamma_1 \dots \gamma_n)])}\\
		& \leq 1 - \frac{1 - \rho_{\text{min}}}{2  (1 - \rho_{\text{min}}^n)} \\
		& \leq 1 - \frac{1}{2n}.
	\end{align*} The result now follows by applying Proposition \ref{prop:non_degen}, Lemma~\ref{Ufixeduniformmixing} and Lemma~\ref{UnifWellMixSpecGap} to $\mu^{*n}$ along with noting that compact subsets of $\mathcal{I}_k$ are also compact with respect to $\mathcal{W}_3$. 
\end{proof}

\begin{proof}[Proof of Proposition~\ref{UnifNonDegContAv}.]
	This follows directly by Proposition \ref{prop:non_degen}, Lemma~\ref{Ufixeduniformmixing} and Lemma~\ref{UnifWellMixSpecGap}.
\end{proof}

Now we prove Proposition \ref{prop:non_degen}. We use the following definition.

\begin{definition}
	Given two measures $\lambda_1, \lambda_2$ on $\R^d$ we define
	\begin{equation*}
		\mathcal{PW}_1(\lambda_1, \lambda_2) :=  \inf_{\gamma \in \Gamma(\lambda_1, \lambda_2)} \sup_{p \in P(d)} \int |px-py| \, d \gamma(x, y)
	\end{equation*}
	where $P(d)$ is the set of orthogonal projections onto one dimensional subspaces of $\R^d$ and $\Gamma(\lambda_1, \lambda_2)$ is the set of couplings between $\lambda_1$ and $\lambda_2$.
\end{definition}

We use this to show that if a measure is sufficiently close to a spherical normal distribution then it is $(\alpha_0, \theta, A)$-non-degenerate.

\begin{lemma} \label{CloseToNormalNonDegen}
    Let $I$ be the $d\times d$ identity matrix. Then given any $p \in P(d)$ we have $$\mathbb{E}_{x \sim N(0,I)}[|p x|] = \sqrt{\frac{2}{\pi}}.$$ Moreover, for any $\varepsilon > 0$ there exists $\alpha_0 \in (0, 1)$ and $\theta, A > 0$ such that if $\nu$ is a measure on $\R^d$ and $$\mathcal{PW}_1(\nu, N(0, I)) < \sqrt{\frac{2}{\pi}} - \varepsilon$$ then $\nu$ is $(\alpha_0, \theta, A)$-non-degenerate.
\end{lemma}

\begin{proof}
    The first part follows since if $X \sim \mathcal{N}(0,I)$ and $u \in \R^d$ is a unit vector, then $\langle X,u \rangle$ is distributed as $\mathcal{N}(0,1)$. The second part follows from the first part, the fact that the $y \in \R$ such that $\mathbb{E}_{x \sim N(0,1)} |x-y|$ is minimal is $y=0$ and Markov's inequality.
    
    More precisely, we aim to estimate for all $y_0 \in \R^d$ and all proper subspaces $W \subset \R^d$ $$\nu(\{ x \in \R^d \,:\, |x-(y_0 + W)| < \theta \text{ or } |x| \geq A  \}),$$ which is bounded by $\nu(\{ x \in \R^d \,:\, |x-(y_0 + W)| < \theta  \}) + \nu(\{ x \in \R^d \,:\, |x| \geq A  \})$. To deal with the second term we note that by Markov's inequality for a coupling $\gamma$ between $\nu$ and $\mathcal{N}(0,1)$ we have 
    \begin{align*}
        \nu(\{ x \in \R^d \,:\, |x| \geq A  \}) &\leq A^{-1} \int |x| \, d\nu(x) \\
        &\leq A^{-1}\left( \int |y| \, d\mathcal{N}(0,I)(y) + \int |x-y| \, d\gamma(x,y) \right).
    \end{align*} In order to apply our bound for $\mathcal{PW}_1(\nu, N(0, I))$ we consider the projections $p_1, \ldots , p_d$ to the coordinate axes. Then $|x-y| \leq \sum_{i = 1}^d |p_ix - p_iy|$ and therefore by choosing a suitable coupling, it follows that for $A$ sufficiently large only depending on $d$ and $\eps$ we have that $\nu(\{ x \in \R^d \,:\, |x| \geq A  \}) \leq \eps/16$.
    
    To deal with the first term $\nu(\{ x \in \R^d \,:\, |x-(y_0 + W)| < \theta  \})$, we assume without loss of generality that $W$ has dimension $d-1$ and we let $p$ be the orthogonal projection to the orthogonal complement of $W$. Then it holds that $|x - (y_0 + W)| = |px - py_0|$ and therefore $$\nu(\{ x \in \R^d \,:\, |x-(y_0 + W)| < \theta  \}) = \nu(\{ x \in \R^d \,:\, |px-py_0| < \theta  \}).$$ In the following we identify $p\R^{d}$ as the real line. Let $\gamma$ be any coupling between $\nu$ and $\mathcal{N}(0,I)$. Then it holds that \begin{align*}
        \int |px-py| \, d\gamma(x,y) &\geq \int |px-py| 1_{|px - py_0| < \theta}(x,y) \, d\gamma(x,y) \\
        &\geq \nu(\{ x \in \R^d \,:\, (|px-py_0| < \theta  \})\int |py-py_0| - \theta) \, d\mathcal{N}(0,I)(y) \\ &\geq \nu(\{ x \in \R^d \,:\, (|px-py_0| < \theta  \}) \left(\sqrt{\frac{2}{\pi}} - \theta\right),
    \end{align*} having used in the last line that $y \in \R$ such that $\mathbb{E}_{x \sim N(0,1)} |x-y|$ is minimal is $y=0$. By choosing a suitable coupling and setting $\theta = \eps/4$ it therefore follows for $\eps$ sufficiently small that  $$\nu(\{ x \in \R^d \,:\, (|px-py_0| < \theta  \}) \leq \frac{\sqrt{\frac{2}{\pi}} - \eps/2}{\sqrt{\frac{2}{\pi}} - \eps/4} \leq 1 - \eps/8.$$ The claim follows by combining the above two estimates.
\end{proof}

To make this useful we need to show that our self-similar measures are close to spherical normal distributions. We prove this in the case where all of the $\rho_i$ are equal with the following proposition.

\begin{proposition} \label{prop:homo_close_to_normal}
	Given any $\varepsilon > 0$ and any irreducible probability measure $\mu_U = \sum_{i=1}^k p_i \delta_{U_i}$ on $O(d)$ there is some $\delta > 0$ and some $\tilde{\rho} \in (0, 1)$ depending on $\eps$ and $\mu_U$ such that the following is true. Let $\mu = \sum_{i = 1}^k p_i\delta_{g_i}$ be a probability measure on $G$ without a common fixed point and with $\mathcal{W}_3(U(\mu), \mu_U) < \delta$ 
    as well as $p_i \geq \eps$ for all $1\leq i \leq k$. Assume there is $\rho \in (\tilde{\rho}, 1)$ such that $\rho(g_i) = \rho$ for all $1\leq i \leq k$. Then there exists some $h \in G$ with $U(h) = I$ such that the self similar measure $\nu_h'$ generated by the conjugate measure $\mu_h' = \frac{1}{2}\delta_e + \frac{1}{2} \sum_i p_i \delta_{hg_ih^{-1}}$ satisfies
	\begin{equation*}
		\mathcal{W}_3(\nu_h', N(0, I)) < \varepsilon.
	\end{equation*}
    If moreover $\mathrm{gap}_H(\mu_U) \geq \eps$ then $\tilde{\rho}$ is uniform in $d$ and $\eps$.
\end{proposition}


We then extend to the general case using the following lemma.

\begin{lemma} \label{lemm:inhom_close_to_normal}
	Let $\gamma$ and $\tilde{\gamma}$ be contracting on average random variables taking values in $G$ such that $U(\gamma) = U(\tilde{\gamma})$ and $z(\gamma) = z(\tilde{\gamma})$ almost surely. Let $\nu$ and $\tilde{\nu}$ be the self similar measures generated by the laws of $\gamma$ and $\tilde{\gamma}$ respectively. Then
	\begin{equation*}
		\mathcal{PW}_1(\nu, \tilde{\nu}) \leq \frac{\mathbb{E} [|\rho(\gamma) - \rho(\tilde{\gamma}) |]}{1 - \mathbb{E}[\rho(\gamma)]} \sup_{p \in P(d)} \mathbb{E}_{x \sim \tilde{\nu}} |px|.
	\end{equation*}
\end{lemma}

We now have all the ingredients needed to prove Proposition \ref{prop:non_degen}.

\begin{proof}[Proof of Proposition \ref{prop:non_degen}] Without        loss of generality we replace $\mu$ by $\frac{1}{2}\delta_e + \frac{1}{2}\mu$. Let $\tilde{g_i}:x \mapsto \hat{\rho} U_i x + b_i$ and let $\tilde{\mu} = \sum_{i=1}^{n} p_i \delta_{\tilde{g}_i}$ with self-similar measure $\tilde{\nu}$. Then by Proposition \ref{prop:homo_close_to_normal} there is some $h \in G$ with $U(h) = I$ such that
    \begin{equation*}
        \mathcal{W}_3(\tilde{\nu}_h, N(0, I)) < \varepsilon / 10.
    \end{equation*}
    Clearly this implies $\mathcal{W}_1(\tilde{\nu}_h, N(0, I)) < \varepsilon / 10$ and therefore $\mathcal{PW}_1(\tilde{\nu}_h, N(0, I)) < \varepsilon / 10$ and so by Lemma \ref{lemm:inhom_close_to_normal} if we define $\mu_h = \sum_{i=1}^{k} p_i \delta_{h g_i h^{-1}}$ and let $\nu_h$ be the self similar measure generated by $\mu_h$ we have $\mathcal{PW}_1(\nu_h, N(0, I)) < \sqrt{\frac{\pi}{2}} - \varepsilon / 2$. The result follows by Lemma \ref{CloseToNormalNonDegen}.
\end{proof}

Now we just need to prove Lemma \ref{lemm:inhom_close_to_normal} and Proposition \ref{prop:homo_close_to_normal}. We start with Lemma \ref{lemm:inhom_close_to_normal}.

\begin{proof}[Proof of Lemma \ref{lemm:inhom_close_to_normal}]
    Let $x$ be a sample from $\nu$ and $\tilde{x}$ be a sample from $\tilde{\nu}$ such that $(x, \tilde{x})$ is independent from $(\gamma, \tilde{\gamma})$. Note that this means that $\gamma x$ is a sample from $\nu$ and $\tilde{\gamma} \tilde{x}$ is a sample from $\tilde{\nu}$. Let $p \in P(d)$. We have
	\begin{align*}
		\mathbb{E}[| p \gamma x - p \tilde{\gamma} \tilde{x} |] &\leq \mathbb{E}[| p \gamma(x - \tilde{x})|] + \mathbb{E}[| p (\gamma - \tilde{\gamma}) \tilde{x} |] \\
		& = \mathbb{E}[\rho(\gamma)] \mathbb{E}[| p U(\gamma) (x - \tilde{x})|] + \mathbb{E}[|\rho(\gamma) - \rho(\tilde{\gamma})|] \mathbb{E}[| p U(\gamma) (\tilde{x})|].
	\end{align*}
	Therefore by taking a series of couplings such that $\sup_{p \in P(d)} \mathbb{E}[|px - p\tilde{x}|] \to \mathcal{PW}_1(\nu, \tilde{\nu})$ we get
	\begin{equation*}
		\mathcal{PW}_1(\nu, \tilde{\nu}) \leq \mathbb{E}[\rho(\gamma)] \mathcal{PW}_1(\nu, \tilde{\nu}) + \mathbb{E}[|\rho(\gamma) - \rho(\tilde{\gamma})|] \mathbb{E}_{x \sim \tilde{\nu}}[|p(x)|].
	\end{equation*}
\end{proof}

Now we wish to prove Proposition \ref{prop:homo_close_to_normal}. First we need the following result.

\begin{lemma} \label{lemma:dense_tends_to_haar}
    Let $\mu_U$ be a probability measure on $O(d)$ and let $H$ be the closure of the group generated by the support of $\mu_U$ and let $V$ be a uniform random variable on $H$. Let $\gamma_1, \gamma_2, \dots$ be independent samples from $\frac{1}{2}\delta_e + \frac{1}{2}\mu_U$. Then for every $\varepsilon > 0$ there exists $N \in \Z_{>0}$ depending on $d,H$ and $\eps$ such that whenever $n \geq N$ we have
	\begin{equation*}
		\mathcal{W}_3(\gamma_1 \dots \gamma_n, V) < \varepsilon.
	\end{equation*}
		Furthermore, if in addition $\mathrm{gap}_{H}(\mu_U) \geq \eps$, then $N$ can be made uniform $d$ and $\varepsilon$.
\end{lemma}


\begin{proof}
    This follows similarly to the arguments given in section~\ref{MixingSection} since the measure $\mu_U' = \frac{1}{2}\delta_e + \frac{1}{2}\mu_U$ satisfies that $(\mu_U')^{*n}\to \Haarof{H}$ as $n\to \infty$. In the presence of a spectral gap we apply Lemma~\ref{GapWassersteinConvergence} and use that by compactness of $H$ the $L^3$-Wasserstein distance is comparable with the $L^1$-Wasserstein distance.  
\end{proof}

It is convenient to work with measures which are appropriately translated.

\begin{definition}
	We say that a probability measure $\mu$ on $G$ is \emph{centred at zero} if $\mathbb{E}_{\gamma \sim \mu} [\gamma(0)] = 0$.
\end{definition}

\begin{lemma} \label{lemma:centred_variance}
	Suppose that $\mu$ is a probability measure on $G$ which is centred at zero and has uniform contraction ratio $\rho \in (0, 1)$. Then if $\gamma_1, \gamma_2, \dots$ are i.i.d.\ samples from $\mu$ and $n \in \mathbb{Z}_{>0}$ we have
	\begin{equation*}
		\mathbb{E}[\gamma_1 \dots \gamma_n(0)] = 0
	\end{equation*}
	and
	\begin{equation*}
		\mathbb{E} [|\gamma_1 \dots \gamma_n(0)|^2] = \frac{1 - \rho^{2n}}{1 - \rho^2} \mathbb{E} [|\gamma_1(0)|^2].
	\end{equation*}
\end{lemma}

\begin{proof}
	Both of these follow by an induction argument left to the reader.
\end{proof}

In order to prove Proposition~\ref{prop:homo_close_to_normal}, we need the following theorem of Sakhanenko from \cite{sakhanenko1985estimates}.

\begin{theorem} \label{theo:l_p_wasserstien_p}
    For every $p, d \geq 1$ there is some constant $C = C(p, d) >0$ such that the following holds. Suppose that $X_1, \dots, X_n$ are independent random variables taking values in $\R^d$ with mean $0$. Let $\Sigma_i = \var X_i$, suppose that $\sum_{i=1}^n \Sigma_i = I$ and let $L_p = (\sum_{i=1}^n \mathbb{E}[|X_i|^p])^{1/p}$. Then
    \begin{equation*}
        \mathcal{W}_p\left(\sum_{i=1}^n X_i, N(0, I)\right) \leq C L_p.
    \end{equation*}
\end{theorem}

This is enough to deduce the following estimate. We note that we work with $\mathcal{W}_3$ norm in order to establish the decaying $(n')^{-1/6}$ term in \eqref{nprimebound}.

\begin{lemma} \label{lemm:w3_estimate}
	Let $(p_1, \dots, p_k)$ be a probability vector, $U_1, \dots, U_k \in O(d)$ generate an irreducible subgroup, $b_1, \dots, b_k \in \R^d$ and let $\rho \in (0, 1)$. Let $\mu$ be the probability measure on $G$ given by $\mu = \sum_{i=1}^{k} p_i \delta_{g_i}$ where $g_i: x \mapsto \rho U_i x + b_i$. Suppose that $\mu$ is centred at zero and that all of the $b_i$ have modulus at most $1$. Let $\gamma_1, \gamma_2, \dots$ be i.i.d.\ samples from $\mu$. Let $\varepsilon \in (0, 1)$.
	
	Given $\ell \in \Z_{>0}$ we define $S_{\ell} := \mathbb{E}[|\gamma_1 \dots \gamma_{\ell}(0)|^2]$ and
	\begin{equation*}
		W_{\ell} := \mathcal{W}_3 \left( d^{1/2} S_{\ell}^{-1/2} \gamma_1 \dots \gamma_{\ell}(0), N(0, I) \right).
	\end{equation*}	
	
	Suppose that there exist $m, n \in \Z_{>0}$ such that for $V$ a uniform random variable on the closure of the subgroup generated by the $U_1, \ldots , U_k$ we have
	\begin{equation*}
        \mathcal{W}_3(U(\gamma_1 \dots \gamma_m), V) < \varepsilon \quad\quad \text{ and }\quad \quad  \frac{m}{S_n^{1/2}} < \varepsilon.
    \end{equation*}
    Then for $n' \in \Z_{>0}$,
    \begin{equation}
        W_{(m+n) n'} \ll_d (T^{-1/6} + T^{1/2} \varepsilon) (W_n + 1)
    \end{equation}
    where $T := \sum_{i=0}^{n'-1} \rho^{(m+n) i}$. In particular if $\rho^{(m+n) n'} > 1/2$ then $n'/2 \leq T \leq n'$ and therefore
    \begin{equation}\label{nprimebound}
        W_{(m+n) n'} \ll_d ((n')^{-1/6} + (n')^{1/2} \varepsilon) (W_n + 1)
    \end{equation}
\end{lemma}

\begin{proof}
	For $i = 1, \dots, n'$ let $$X_i := \gamma_{(i-1)(n+m) + 1} \dots \gamma_{(i-1)(n+m)+m} $$ and $$Y_i := \gamma_{(i-1)(n+m) + m + 1} \dots \gamma_{i(n+m)}$$ such that $$Z_i = X_iY_i = \gamma_{(i-1)(n+m) + 1} \dots \gamma_{i(n+m)}.$$ Furthermore consider $V_1, \dots, V_k$ independent random variables which are uniform on $H$ (the closure of the subgroup generated by the $U_i$), independent of the $Y_i$ and are such that
	\begin{equation*}
		\mathbb{E} [\| U(X_i) - V_i \|^3] < \varepsilon^3.
	\end{equation*}
	
	Note that
	\begin{align*}
		Z_1 \dots Z_{n'} (0) &= Z_1(0) + \rho^{(m+n)} U(Z_1) Z_2 (0) + \\ & \dots + \rho^{(m+n)(n'-1)} U(Z_1 \dots Z_{n'-1}) Z_{n'} (0).
	\end{align*}
	
	Also note that
	\begin{align*}
	\MoveEqLeft \mathcal{W}_3\left(\rho^{(m+n)(i-1)} U(Z_1 \dots Z_{i-1}) Z_i (0), \rho^{(m+n)(i-1)+m} V_i Y_i(0) \right) \\ &= \rho^{(m+n)(i-1)}\mathcal{W}_3\left(U(Z_1 \dots Z_{i-1}) ( \rho^m U(X_i)Y_i(0) + X_i(0)), \rho^{m} V_i Y_i(0) \right) \\
    &\leq \rho^{(m+n)(i-1)}(m + \varepsilon \rho^{m} (\mathbb{E} \left[ | Y_i(0) |^3 \right])^{1/3})\\
	& \ll_d \varepsilon \rho^{(m+n)(i-1)} S_n^{1/2} (W_n + 1),
	\end{align*} having used the triangle inequality in the second line and that $|X_i(0)| \leq m$ as $\sup_i|b_i| \leq 1$ as well as that 
    \begin{align*}
        \MoveEqLeft \mathcal{W}_3\left(U(Z_1 \dots Z_{i-1})U(X_i)Y_i(0) , V_i Y_i(0) \right) \\ &= \mathcal{W}_3\left(U(Z_1 \dots Z_{i-1})U(X_i)Y_i(0), U(Z_1 \dots Z_{i-1})V_i Y_i(0) \right)
    \end{align*} 
    as $V_i$ is distributed like the Haar measure on $H$.
	
	Note that by Lemma~\ref{IrreducibleTrick} the covariance matrix of $V_iY_i(0)$ is $d^{-1}S_n I$. Therefore by Theorem \ref{theo:l_p_wasserstien_p} letting $A= d^{-1/2} \left( \frac{1 - \rho^{2 n' (m+n) } }{1 - \rho^{2(m+n)}} \right)^{1/2} S_n^{1/2}$ we have that
	\begin{align*}
	\MoveEqLeft \mathcal{W}_3 \left( A^{-1}  \left( \sum_{i=1}^{n'}\rho^{(m+n)(i-1)} V_i Y_i(0) \right), N(0, I) \right) \\ & \ll \left( \sum_{i=1}^{n'} \mathbb{E} \left[ |A^{-1} \rho^{ (m+n)(i-1)} Y_i(0)|^3 \right] \right)^{1/3} \\
	&\ll_d A^{-1} \left( \frac{1 - \rho^{3 (m+n) n'}}{1 - \rho^{3(m+n)}} \right)^{1/3} (W_n + 1) \\
    & \ll_d T^{-1/6} (W_n+1),
	\end{align*} where we exploited that $$\frac{1 - \rho^{2 n' (m+n) } }{1 - \rho^{2(m+n)}} = \frac{1 - \rho^{n' (m+n) } }{1 - \rho^{(m+n)}}\frac{1 + \rho^{ n' (m+n) } }{1 + \rho^{(m+n)}} \in [T/2,T]$$ and a similar estimate for $\left( \frac{1 - \rho^{3 (m+n) n'}}{1 - \rho^{3(m+n)}} \right)^{1/3}$.
	
	Therefore we may deduce that
	\begin{equation*}
	\mathcal{W}_3 \left( A^{-1} \gamma_1 \dots \gamma_{(m+n) n'} (0), N(0, I) \right) \ll_d T^{-1/6} (W_n + 1) + \varepsilon T^{1/2} (W_n + 1)
	\end{equation*}
	By Lemma \ref{lemma:centred_variance} we have that
	\begin{equation*}
	\frac{d^{-1/2} S_{n'}^{-1/2}}{A} = 1 + O(\frac{m}{n}) = 1 + O(\varepsilon).
	\end{equation*}
	We conclude
	\begin{align*}
		W_{(m+n) n'} &\ll_d T^{-1/6} (W_n + 1) + \varepsilon T^{1/2} (W_n + 1) + \varepsilon \\
		&\ll_d T^{-1/6} (W_n + 1) + \varepsilon T^{1/2} (W_n + 1)
	\end{align*}
	as required.
\end{proof}

From this we can deduce the following.
\begin{corollary} \label{coro:many_bounded}
    For every $\varepsilon > 0$ and every irreducible probability measure $\mu_U$ on $O(d)$ there is $\delta > 0$,  $C>1$ and $\tilde{\rho} \in (0, 1)$ such that the following is true. Let $\mu = \sum_{i = 1}^k p_i\delta_{g_i}$ be a probability measure on $G$ such that $\mathcal{W}_3(U(\mu), \mu_U) < \delta$ and $p_i \geq \eps$ for all $1 \leq i \leq k$. Assume further that $\max_{1 \leq i \leq k} |b(g_i)| = 1$ and for some $\rho \in (\tilde{\rho},1)$ we have $\rho(g_i) = \rho$ for all $1\leq i \leq k$. Suppose that $\mu$ is centred at zero and let $\gamma_1, \gamma_2, \dots$ be i.i.d.\ samples from $\mu$. Then for every $\ell \in \Z_{>0}$ with $\ell \geq 24$ such that $C^{\ell+1} < \frac{\rho^{C}}{1 - \rho^{C}}$ there is some $n \in \Z_{>0}$ such that
    \begin{equation*}
        \frac{1}{1 - \rho^n} \in [C^{k}, C^{k+1}]
    \end{equation*}
    and
    \begin{equation*}
        \mathcal{W}_3(d^{1/2} S_n^{-1/2} \gamma_1 \dots \gamma_n(0), N(0, I)) < \varepsilon.
    \end{equation*}

    Moreover, if $U(\mu)$ is supported on $H$ and $\mathrm{gap}_H(\mu) \geq \eps$, then $C$ and $\tilde{\rho}$ can be made uniform $d$ and $\eps$.
\end{corollary}


\begin{proof}
    Let $\varepsilon' > 0$ be sufficiently small in terms of $\varepsilon$. Choose $m = m(\mu_U, \varepsilon')$ and $\delta = \delta(\mu_U, \varepsilon')$ such that
    \begin{equation*}
        \mathcal{W}_3(U(\gamma_1 \dots \gamma_m), V) < \varepsilon'
    \end{equation*}
    and choose $n_0 = n_0(\varepsilon, \varepsilon' , \tilde{\rho})$ such that
    \begin{equation*}
        \frac{m}{S_{n_0}^{1/2}} < \varepsilon'.
    \end{equation*}
    Note that by Lemma~\ref{lemma:centred_variance} provided $\tilde{\rho}$ is sufficiently close to $1$ (in terms of $\varepsilon'$) we may assume that $n_0 \ll_{d, \mu_U} {\varepsilon'}^{-2}$. Now inductively chose $n_\ell'$ such that $\sum_{i=0}^{n_\ell'-1} \rho^{(m+n_\ell) i} \in [\varepsilon'^{-3/2}, 2 \varepsilon'^{-3/2}]$ and define $n_{\ell+1} := n_\ell'(n_\ell + m)$. Repeat this process until we find some $\ell$ such that $\sum_{i=0}^{\infty} \rho^{(m+n_\ell) i} < \varepsilon'^{-3/2}$ and let $\ell^*$ denote this value of $\ell$. Let $W_i := \mathcal{W}_3(d^{1/2} S_{i}^{-1/2} \gamma_1 \dots \gamma_{i}(0), N(0, I))$. By Lemma \ref{lemm:w3_estimate} we have that for $i=1, \dots, \ell^*$ we have
    \begin{equation*}
        W_{n_i} \ll_d \varepsilon'{}^{1/4} (W_{n_{i-1}} + 1).
    \end{equation*}
    Clearly $W_0 \ll_d n_o ^3 \ll {\varepsilon'}^{-6}$. Therefore for every $\ell \geq 24$ we have
    \begin{equation*}
        W_{n_\ell} \ll \varepsilon '^ {1/4}
    \end{equation*}
    We also have that
    \begin{equation*}
        \frac{1 - \rho^{n_{\ell+1}}}{1 - \rho^{m+n_\ell}} \in [\varepsilon'^{-3/2}, 2 \varepsilon^{-3/2}]
    \end{equation*}
    and so providing we choose $\tilde{\rho}$ to be sufficiently large we have
    \begin{equation*}
        \frac{1 - \rho^{n_{\ell+1}}}{1 - \rho^{n_\ell}} \leq 4 \varepsilon'^{-3/2}.
    \end{equation*}
    The result follows. When we have a spectral gap, all of these constants can be chosen to be uniform.
\end{proof}

Now we have enough to prove Proposition \ref{prop:homo_close_to_normal}.

\begin{proof}[Proof of Proposition \ref{prop:homo_close_to_normal}]
	Without loss of generality we may assume that $\mu$ is centred at zero and that $\max_{i=1}^k |b_i| = 1$.
 
    Let $\varepsilon' > 0$. By Lemma \ref{lemma:dense_tends_to_haar} there is some $m \in \Z_{>0}$ depending only on $\varepsilon$ and $\varepsilon'$ such that
	\begin{equation*}
		\mathcal{W}_3(U(\gamma_1 \dots \gamma_m), V) < \varepsilon'.
	\end{equation*}
	
	By Lemma \ref{lemma:centred_variance} there is some $N$ depending only on $\mu_U$ and $\varepsilon'$ such that for any $n \geq N$ we have
	\begin{equation*}
		\frac{m}{S_{n}^{1/2}} < \varepsilon'.
	\end{equation*}

    Let $C$ be as in Corollary \ref{coro:many_bounded} and choose $n$ such that 
    \begin{equation*}
        \frac{1}{1 - \rho^{m+n}} \in [C^{-1} \varepsilon'^{-3/2}, C \varepsilon'^{-3/2}].
    \end{equation*}
    Providing we choose $\tilde{\rho}$ sufficiently close to $1$ we will also have $n \geq N$. By letting $n' \to \infty$ in Lemma \ref{lemm:w3_estimate} we deduce that
    \begin{equation*}
		\mathcal{W}_3(A^{-1}\nu, N(0, I)) \ll_d C \varepsilon'^{1/4}
	\end{equation*}
    where $A = d^{1/2}(1 - \rho^2)^{1/2} = \lim_{\ell \to \infty} d^{1/2}S_{\ell}^{-1/2} $. In the presence of a spectral gap, all of these bounds are easily seen to be uniform.
\end{proof}

    \section{Construction of Examples}

    \label{section:examples}

    Throughout this section we denote as usual by $G = \mathrm{Sim}(\R^d)$. We first study random walk entropy in section~\ref{subsection:Entropy} and then the separation rate in section~\ref{HeightSeparationSection}.  We then prove all of the corollaries stated in the introduction.

\subsection{Bounding Random Walk Entropy}\label{subsection:Entropy}

The techniques from \cite[Section 6.3]{HochmanSolomyak2017} or \cite[Section 9.2]{Kittle2023} follow through to our setting. In particular we have the following using Breuillard's strong Tits alternative. 

\begin{proposition} (\cite[Section 6.3]{HochmanSolomyak2017}) \label{EntropyFromTits}
    Let $d\geq 1$. Then for every $p_0 > 0$ there exists $\rho = \rho(p_0,d)$ such that if $\mu = \sum_{i = 1}^k p_i \delta_{g_i}$ is a finitely supported probability measure on $G$ with $p_i \geq p_0$ and $\mathrm{supp}(\mu)$ generates a non-virtually solvable subgroup, then $h_{\mu} \geq \rho$.
\end{proposition}

We will also use the following version of the ping-pong lemma for which we provide a full proof for the convenience of the reader.

\begin{lemma}(Ping-Pong) \label{lemm:ping_pong}
	Let $G$ be a group acting on a set $X$ and let $g_1, g_2 \in G$. Assume there exist disjoint non-empty sets $A_1, A_2 \subset X$ such $$g_1(A_1 \cup A_2) \subset A_1 \quad \text{ and } \quad g_2(A_1 \cup A_2) \subset A_2.$$ Then $g_1$ and $g_2$ generate a free semigroup.
\end{lemma}

When this happens we say that $g_1$ and $g_2$ \emph{play ping pong}.

\begin{proof}
	Let $w_1 = h_1h_2\cdots h_{\ell_1}$ and $w_2 = f_1f_2\cdots f_{\ell_2}$ with distinct sequences $h_i, f_j \in \{ g_1, g_2 \}$. Assume without loss of generality that $\ell_1 \leq \ell_2$. First assume that there is some $1 \leq k \leq \ell_1$ such that $h_k \neq f_k$. Choose the smallest such $k$ and note that it suffices to show that $h_k\cdots h_{\ell_1} \neq f_k\cdots f_{\ell_2}$, which follows by applying the resulting maps to any $x \in A_1 \cup A_2$ and noting that $h_k\cdots h_{\ell_1}x \neq f_k\cdots f_{\ell_2}x$. On the other hand assume that $h_i = f_i$ for all $1 \leq i \leq \ell_1$, in which case we need to show that $w' = f_{\ell_1 + 1}\cdots f_{\ell_2}$ is not the identity. Without loss of generality assume that $f_{\ell_1 + 1} = g_1$. Then for $x \in A_2$ we have that $w'x \in A_1$ and thus $w'$ is not the identity. We note that in particular it follows by the assumptions that $g_1$ and $g_2$ have infinite order.
\end{proof}

\begin{lemma}\label{lemm:FreeEntBound}
    Let $\mu$ be a finitely supported probability measure on $G$ such that $g_1, g_2 \in \mathrm{supp}(\mu)$ generate a free semigroup. Then $$h_{\mu} \gg  \min\{ \mu(g_1), \mu(g_2) \}.$$
\end{lemma}

\begin{proof}
    Denote $\mu' = \frac{1}{2}\delta_e + \frac{1}{2}\mu$. Then by \cite[Lemma 6.8]{HochmanSolomyak2017} we have $h_{\mu'} = h_{\mu}/2$. Thus the claim follows from \cite[Proposition 9.7]{Kittle2023} (generalised to $G$ and applied to $K = \min\{ \mu(g_1), \mu(g_2) \}/2$).
\end{proof}

\subsubsection{$p$-adic Ping-Pong}

We first use ping-pong in a $p$-adic setting. For a number field $K$ with ring of integers $O_K$. Let $\mathfrak{p} \subset O_K$ be a prime ideal and we denote by $R_{\mathfrak{p}}$ the localization of $O_K$ at $P$ defined as $$R_{\mathfrak{p}} = \left\{ \frac{a}{b} : a \in O_K, b \in O_K \backslash \mathfrak{p} \right\}.$$

\begin{lemma}($p$-adic Ping-Pong) \label{lemm:p_adic_ping_pong}
	Let $K$ be a number field and let $O_K$ be its ring of integers. Let $\mathfrak{p} \subset O_K$ be a prime ideal and let $M_{\mathfrak{p}}$ be the ideal of $R_{\mathfrak{p}}$ defined by $$M_{\mathfrak{p}} = \left\{ \frac{a}{b} : a \in \mathfrak{p}, b \in O_K \backslash \mathfrak{p} \right\}.$$ Let $g_1, g_2 \in G$ be such that all of the entries of $\rho(g_1) U(g_1)$ and $\rho(g_2) U(g_2)$ are in $M_{\mathfrak{p}}$ and all components of $b_1$ and $b_2$ are in $R_{\mathfrak{p}}$. Suppose that $$M_{\mathfrak{p}} \times \cdots \times M_{\mathfrak{p}} + b_1 \neq M_{\mathfrak{p}} \times \cdots \times M_{\mathfrak{p}} + b_2.$$ Then $g_1$ and $g_2$ generate a free semigroup.
\end{lemma}

\begin{proof}
	This follows immediately from Lemma \ref{lemm:ping_pong} with $X = R_{\mathfrak{p}} \times \cdots \times R_{\mathfrak{p}}$ and $A_i = M_{\mathfrak{p}} \times \cdots \times M_{\mathfrak{p}} + b_i$ for $i=1, 2$.
\end{proof}

\subsubsection{Ping-Pong under a Galois transform}

We can also apply the ping-pong lemma using field automorphisms. Recall that given a number field $K$, the automorphism group $\mathrm{Aut}(K/\Q)$ consists of field automorphisms that fix $\Q$.

\begin{lemma}(Galois Ping-Pong)\label{lemm:GaloisPingPong}
    Let $g_1$ and $g_2$ be two elements in $G$ whose coefficients lie in a real number field $K$ and without a common fixed point. Let $\Phi \in \mathrm{Aut}(K/\Q)$ be such that for $i = 1,2$ we have $$|\rho(\Phi(g_i))| < 1/3.$$ Then $g_1$ and $g_2$ generate a free semigroup.
\end{lemma}

\begin{proof}
    For $i = 1,2$ write $h_i = \Phi(g_i)$ and let $p_i$ be the fixed point of $h_i$, which has coefficients in $K$ since it arises from a linear equation over $K$. Then $h_1 \neq h_2$ as $g_1$ and $g_2$ have no common fixed point. Consider $A_i = B_{d(h_1,h_2)/2}(h_i)$ (the open ball around $h_i$ of radius $d(h_1,h_2)/2$) and note further that $h_1(A_1 \cup A_2)\subset A_1$ and $h_2(A_1 \cup A_2)\subset A_2$. So the claim follows by Lemma~\ref{lemm:ping_pong}. 
\end{proof}

\subsubsection{Height Entropy Bound in Dimension One}

In dimension one we also have the following tool for bounding the random walk entropy. We use the absolute height $\mathcal{H}(\alpha)$ and the logarithmic heigh $h(\alpha)$ of an algebraic number $\alpha$ as defined in \eqref{def:height} and \eqref{def:logheight}.

\begin{proposition} \label{proposition:height_to_entropy}
    Suppose that $\mu$ is a finitely supported probability measure on $G$ and that there exist $f, g \in \supp \mu$ which are of the form $f: x \mapsto \lambda_1x + 1$ and $g : x \mapsto \lambda_2 x$ with $\lambda_1$ and $\lambda_2$ real algebraic and $\lambda_2 \neq 1$. Let $n =  \lceil\frac{\log 3}{\max \{ h(\lambda_1), h(\lambda_2) \}}\rceil + 2$. Then
    \begin{equation*}
        h_{\mu} \gg \frac{1}{n} \min \{ \mu(f), \mu(g) \}^n.
    \end{equation*}
\end{proposition}

This is a simple consequence of the following lemma.

\begin{lemma} \label{lemma:height_to_free_semigroup}
    Suppose that $\lambda$ is algebraic and in some number field $K$. Let $f, g \in G$ be defined by $f : x \mapsto \lambda(x - a) + a$ and $g : x \mapsto \lambda(x - b) + b$ for some $a, b \in K$ with $a \neq b$. Suppose that $\mathcal{H}(\lambda) > 3$. Then $f$ and $g$ freely generate a free semi-group.
\end{lemma}

\begin{proof}
    First note that 
    \begin{equation*}
        \mathcal{H}(\lambda) = \mathcal{H}(\lambda^{-1}) = \prod_{v \in M_K} \min(1, |\lambda|_v)^{-\frac{n_v}{[K:\mathbb{Q}]}}.
    \end{equation*}
    This means that either there is some Archimedean  place $v$ such that $| \lambda|_v < 1/3$ or there is some non-Archimedean place $v$ such that $|\lambda |_v <1$.

    In the Archimedean case there is some Galois transform $\rho$ such that $| \rho(\lambda)| < 1/3$ and the result follows from Lemma~\ref{lemm:GaloisPingPong}. In the non-Archimedean case there is some prime ideal $\mathfrak{p} \subset \mathcal{O}_K$ with $\lambda \in \mathfrak{p}$ and Lemma~\ref{lemm:p_adic_ping_pong} applies.
\end{proof}

We now deduce Proposition \ref{proposition:height_to_entropy}.

\begin{proof}[Proof of Proposition \ref{proposition:height_to_entropy}]
    For $n =  \lceil\frac{\log 3}{\max \{ h(\lambda_1), h(\lambda_2) \}}\rceil + 2$, using that $h(\alpha^n) = |n|h(\alpha)$ and $h(\alpha \beta) \geq h(\alpha)- h(\beta)$ for all $\alpha,\beta$ algebraic and $n \in \Z$, there exists $f', g' \in \{f, g \}^n$ satisfying the conditions of Lemma \ref{lemma:height_to_free_semigroup}. Therefore by Lemma \ref{lemm:FreeEntBound} we deduce that
    \begin{equation*}
        h_{\mu^{*n}} \gg \min \{ \mu(f), \mu(g) \}^n \quad\quad \text{ and so } \quad\quad h_{\mu} \gg \frac{1}{n} \min \{ \mu(f), \mu(g) \}^n
    \end{equation*}
    as required.
\end{proof}

\subsection{Heights and Separation}\label{HeightSeparationSection}

In this subsection we will review some techniques for bounding $S_{\mu}$ using heights as defined in \eqref{def:height} and \eqref{def:logheight}. We wish to bound the size of polynomials of algebraic numbers. To do this we need the following way of measuring the complexity of a polynomial.

\begin{definition}
Given some polynomial $P \in \Z [X_1, X_2, \dots, X_n]$ we define the \emph{length} of $P$, which we denote by $\mathcal{L}(P)$, to be the sum of the absolute values of the coefficients of $P$.
\end{definition}

We recall the following basic facts about heights.

\begin{lemma} \label{lemma:height_basic}
The following properties hold:
\begin{enumerate}[label = (\roman*)]
    \item $\mathcal{H}(\alpha^{-1}) = \mathcal{H}(\alpha)$ for any non-zero algebraic number $\alpha$.
    \item If $\alpha$ is a non-zero algebraic number of degree $d$, $$\mathcal{H}(\alpha)^{-d} \leq |\alpha| \leq \mathcal{H}(\alpha)^d.$$
    \item Given $P \in \Z[X_1, X_2, \dots, X_n]$ of degree at most $L_1 \geq 0$ in $X_1$, $\dots$, $L_n \geq 0$ in $X_n$ and algebraic numbers $\xi_1, \xi_2, \dots, \xi_n$ we have
\begin{equation*}
\mathcal{H}(P(\xi_1, \xi_2, \dots, \xi_n)) \leq \mathcal{L}(P) \mathcal{H}(\xi_1)^{L_1} \dots \mathcal{H}(\xi_n)^{L_n}
\end{equation*}
\end{enumerate}
\end{lemma}
\begin{proof}
(i) and (ii) are well-known and (iii) is \cite[Proposition 14.7]{MasserPolynomialsBook}.
\end{proof}

\begin{proposition} \label{HeightSeparationBound}
    Suppose that $\mu$ is a finitely supported measure on $G = \mathrm{Sim}(\R^d)$. Let $S$ be the set of coefficients of $\rho(g), U(g)$ and $b(g)$ with $g \in \mathrm{supp}(\mu)$ supported on a finite set of points. Suppose that all of the elements of $S$ are algebraic and let $K$ be the number field generated by $S$. Then $$S_{\mu} \ll_d [K:\mathbb{Q}] \max (\{h(y) : y \in S \} \cup \{1 \} ).$$
\end{proposition}

\begin{proof}
    We let $m, n \in \mathbb{Z}_{>0}$ and we consider an expression of the from $$a_1^{-1} a_2^{-1} \dots a_n^{-1} b_1 b_2 \dots b_m$$ for $a_1, \ldots , a_n$ and $b_1, \ldots , b_m$ elements in the support of $\mu$. We wish to show that this is either the identity or at least some distance away from the identity. Let $C := \max\{\mathcal{H}(y) : y \in S \}$. First note that $$\rho(a_1^{-1} a_2^{-1} \dots a_n^{-1} b_1 b_2 \dots b_m) - 1$$ is a polynomial in elements of $S$ and their inverses with length $2$ and total degree at most $n+m$. Therefore by Lemma~\ref{lemma:height_basic} $$H(\rho(a_1^{-1} a_2^{-1} \dots a_n^{-1} b_1 b_2 \dots b_m) - 1) \leq 2 C^{m+n}$$
		and so either $\rho(a_1^{-1} a_2^{-1} \dots a_n^{-1} b_1 b_2 \dots b_m) = 1$ or $$|\rho(a_1^{-1} a_2^{-1} \dots a_n^{-1} b_1 b_2 \dots b_m) - 1| \geq 2^{-[K:\mathbb{Q}]} C^{-(m+n)[K:\mathbb{Q}]}.$$
		By a similar argument, using that the coefficients of the inverse matrix of a matrix are polynomial in the coefficients of the given matrix, we see that either $$U(a_1^{-1} a_2^{-1} \dots a_n^{-1} b_1 b_2 \dots b_m) = I$$ or $$||U(a_1^{-1} a_2^{-1} \dots a_n^{-1} b_1 b_2 \dots b_m) - I|| \geq  (d^{m+n} + 1)^{-[K:\mathbb{Q}]} C^{-O_d(m+n)[K:\mathbb{Q}]}$$
		and that either $b(a_1^{-1} a_2^{-1} \dots a_n^{-1} b_1 b_2 \dots b_m) = 0$ or $$|b(a_1^{-1} a_2^{-1} \dots a_n^{-1} b_1 b_2 \dots b_m)| \geq  (d^{m+n} + 1)^{-[K:\mathbb{Q}]} C^{-O_d(m+n)[K:\mathbb{Q}]}.$$ Overall this means that either $a_1^{-1} a_2^{-1} \dots a_n^{-1} b_1 b_2 \dots b_m = \id$ or $$\log d(a_1^{-1} a_2^{-1} \dots a_n^{-1} b_1 b_2 \dots b_m, \id) \gg_d -(m + n)(\log C + 1)[K:\mathbb{Q}].$$ The result follows.
\end{proof}

\subsection{Inhomogeneous examples in $\R$}

In this section, we prove Corollary~\ref{Dim1Imhomogen}, Corollary~\ref{Dim1ImhomogenEasy} and Corollary~\ref{Dim1ContAverageEasy}, which we all recall for convenience of the reader.

\begin{corollary*}[Restatement of Corollary \ref{Dim1Imhomogen}]
		For every $\eps > 0$ there exists a small constant $c = c(\eps) > 0$ such that the following holds. Let $K$ be a number field and $\lambda_1,\lambda_2 \in K \cap (0,1)$ and write $h(\lambda_1,\lambda_2) = \max\{ h(\lambda_1), h(\lambda_2) \}$. Consider the similarities given for $x \in \R$ as $$g_1(x) = \lambda_1 x \quad\quad \text{ and } \quad\quad g_2(x) = \lambda_2 x + 1.$$ Then the self-similar measure of $\mu = \frac{1}{2}\delta_{g_1} + \frac{1}{2}\delta_{g_2}$ is absolutely continuous if $$ h(\lambda_1,\lambda_2) \geq \eps \quad\quad \text{ and } \quad\quad  |\chi_{\mu}| \max\{ 1, \log ([K:\Q] h(\lambda_1,\lambda_2) ) \} ^2 < c.$$
	\end{corollary*}

\begin{proof}(of Corollary~\ref{Dim1Imhomogen})
    Write $\mu = \frac{1}{2}\delta_{g_1} + \frac{1}{2}\delta_{g_2}$ By Proposition~\ref{proposition:height_to_entropy} for every $\eps > 0$ there exists a $\delta > 0$ such that if $h(\lambda_1,\lambda_2) \geq \eps$ then it follows that $h_{\mu} \geq \delta$. Therefore by Theorem~\ref{MainResultFixU} and  using that $S_{\mu} \ll h(\lambda_1,\lambda_2) [K:\Q]$ as shown in Proposition~\ref{HeightSeparationBound}, it follows that $\mu$ is absolutely continuous if for absolute constants $C_1,C_2$ it holds that $$\frac{\delta}{|\chi_{\mu}|} \geq C_1 \max \{ 1, \log (C_2\delta^{-1}h(\lambda_1,\lambda_2) [K:\Q] ) \}^2,$$ which easily implies the claim. 
\end{proof}

\begin{corollary*}[Restatement of Corollary \ref{Dim1ImhomogenEasy}]
		There exists an absolute constant $c > 0$ such that the following holds. Let $\lambda_i = 1 - p_i/q_i \in (0,1)$ be rational for $i \in \{1,2 \}$ with coprime integers $p_i, q_i \geq 1$ and let $g_1$ and $g_2$ be the similarities from \eqref{EasySimilarities}. Then the self-similar measure of $\frac{1}{2}\delta_{g_1} + \frac{1}{2}\delta_{g_2}$ is absolutely continuous if for $i \in \{ 1,2 \}$, $$\frac{p_i}{q_i}\leq \frac{c}{(\log \log \max\{q_1, q_2\})^2}.$$
	\end{corollary*}

\begin{proof}(of Corollary~\ref{Dim1ImhomogenEasy})
This follows easily from Corollary~\ref{Dim1Imhomogen} by taking $K = \Q$ and noting that then $h(\lambda_1, \lambda_2) = \max\{ \log q_1, \log q_2\} \geq \log 2$ as $q_i \geq 2$ since $\lambda_i \in (0,1)$. So by Corollary~\ref{Dim1Imhomogen} there is an absolute constant $c >  0$ such that the self-similar measure of $\mu = \frac{1}{2}\delta_{g_1} + \frac{1}{2}\delta_{g_2}$ is absolutely continuous if $$|\chi_{\mu}| \leq \frac{c}{(\log \log \max\{ q_1, q_2\})^2}.$$ As $-x \leq \log (1-x)$ for $x \in [0,\frac{1}{2}]$, it follows that $|\chi_{\mu}| \leq \frac{1}{2}\frac{p_1}{q_1} + \frac{1}{2}\frac{p_2}{q_2}$. Thus, to prove that the self-similar measure is absolutely continuous, it suffices to show that $$\frac{1}{2}\frac{p_1}{q_1} + \frac{1}{2}\frac{p_2}{q_2} \leq\frac{c}{(\log \log \max\{ q_1, q_2\})^2}$$ as follows directly from the assumption.
\end{proof}

    \begin{corollary*}[Restatement of Corollary \ref{Dim1ContAverageEasy}]
		For $n\geq 1$ consider the similarities given for $x\in \R$ as $$g_1(x) = \frac{n + 1}{n} x \quad\quad \text{ and } \quad\quad g_2(x) = \frac{n}{n+ 1} x + 1.$$ Then the self-similar measure of $\mu = \frac{1}{5}\delta_{g_1} + \frac{4}{5}\delta_{g_2}$ is absolutely continuous for sufficiently large $n$. 
	\end{corollary*}

\begin{proof}(of Corollary~\ref{Dim1ContAverageEasy})
    By Proposition~\ref{proposition:height_to_entropy} and Proposition~\ref{HeightSeparationBound} it holds that $h_{\mu} \gg_{\mu} \log n$ and $S_{\mu} \ll  \log n$. Thus by Theorem~\ref{MainResultContAvFixU} it suffices to show that there is $\hat{\rho} \in (0,1)$ sufficiently close to $1$ such that for sufficiently large $N$, $$\sup_{n \geq N}\frac{\mathbb{E}_{\gamma \sim \mu}[|\hat{\rho} - \rho(\gamma)|]}{1 - \mathbb{E}_{\gamma \sim \mu}[\rho(\gamma)]} < 1.$$ We note that $$1 - \mathbb{E}_{\gamma \sim \mu}[\rho(\gamma)] = 1 - \left(\frac{1}{5}\frac{n + 1}{n} + \frac{4}{5}\frac{n}{n + 1}\right) = \frac{3n - 1}{5n(n + 1)}.$$
    We then set $\hat{\rho} = \frac{n}{n + 1}$ (which actually is the optional $\hat{\rho}$) and then have $$\mathbb{E}_{\gamma \sim \mu}[|\hat{\rho} - \rho(\gamma)|] = \frac{1}{5}\left(\frac{n + 1}{n} - \frac{n}{n+1}\right) = \frac{2n + 1}{5n(n + 1)}.$$ Thus it follows with this choice of $\hat{\rho}$ that $$\lim_{n \to \infty} \frac{\mathbb{E}_{\gamma \sim \mu}[|\hat{\rho} - \rho(\gamma)|]}{1 - \mathbb{E}_{\gamma \sim \mu}[\rho(\gamma)]} = \lim_{n \to \infty} \frac{2n + 1}{3n - 1} = \frac{2}{3} < 1,$$ concluding the proof.
\end{proof}

\subsection{Examples in $\R^d$}
\label{section:GenEx}

In this section we prove Corollary~\ref{QuiteGeneralExample}, Corollary~\ref{QuiteGeneralContAvExample} and Corollary~\ref{MostGeneralExample} on general examples with absolutely continuous self-similar measures, which we all again recall for convenience of the reader.

\begin{corollary*}[Restatement of Corollary \ref{QuiteGeneralExample}]
		Let $d \geq 1$ and $\eps > 0$, let $\mu_U = \sum_{i = 1}^k p_i\delta_{U_i}$ be an irreducible probability measure on $O(d)$ with $p_i \geq \eps$ and let $b_1, \ldots , b_k \in \R^d$ with $b_1 \neq b_2$. Assume that $U_1, \ldots , U_k$ and $b_1, \ldots , b_k$ have algebraic coefficients. Let $q$ be a prime number and for $1 \leq i \leq k$ consider $$g_i(x) = \frac{q}{q+a_{i,q}}U_ix + b_i \quad \quad \text{ for any integer } \quad \quad a_{i,q} \in [1, q^{1-\eps}].$$ Assume that $g_1, \ldots , g_k$ do not have a common fixed point and consider $\mu = \sum_{i = 1}^k p_i\delta_{g_i}$. Then the self-similar measure of $\mu$ is absolutely continuous for $q$ a sufficiently large prime depending on $d,\eps,U_1, \ldots , U_k$ and $b_1, \ldots , b_k$. 
	\end{corollary*}

\begin{proof}[Proof of Corollary~\ref{QuiteGeneralExample}.]
    We first show that $g_1$ and $g_2$ generate a free semigroup for sufficiently large $q$ by using Lemma~\ref{lemm:p_adic_ping_pong}. For simplicity we first treat the case when all of the entries are rational. Then consider the $q$-adic numbers $\Q_{q}$ and the $q$-adic integers $\Z_{q}$. As the $U_1, \ldots , U_k$ and the $b_1, \ldots , b_k$ are fixed, for a sufficiently large prime $q$ all of their entries are in $\Z_{q} \backslash q\Z_q$. On the other hand, by construction $\rho(g_i) \in q\Z_q$ for $1\leq i \leq k$ and as $q\Z_q$ is an ideal therefore also all of the entries of $\rho(g_i)U_i$ are in $q\Z_q$. By Lemma~\ref{lemm:p_adic_ping_pong} it therefore suffices to check that $(q\Z_q)^d + b_1 \neq (q\Z_q)^{d} + b_2$ or equivalently $b_1 - b_2 \not\in (q\Z_q)^d$, which is clearly the case for sufficiently large $q$. Thus $g_1$ and $g_2$ generate a free semigroup. The same argument applies in the general case for $K$ the number field generated by the coefficients of the entries of $g_i$ and by choosing any prime ideal that factors $(q)$.

    Thus it follows by Lemma~\ref{lemm:FreeEntBound} that $h_{\mu} \gg \eps$ and note that by Lemma~\ref{lemma:height_basic} it holds that $S_{\mu} \ll_{K,d} \log q$. Hence there exists a constant $C$ depending on all the relevant parameters such that the self-similar measure of $\mu$ is absolutely continuous if $$C|\chi_{\mu}| \leq \frac{1}{(\log \log q)^2}.$$ Therefore it remains to estimate the Lyapunov exponent. Indeed, note that 
    \begin{align*}
        \log\left( \frac{q}{q + a_{i,q}} \right) &= \log\left(1 - \frac{a_{i,q}}{q + a_{i,q}}\right) \geq \log\left(1 - \frac{q^{1 - \eps}}{q}\right) \gg -q^{-\eps}.
    \end{align*} Therefore $|\chi_{\mu}| \ll q^{-\eps}$ and the claim follows for sufficiently large $q$.
\end{proof}

\begin{corollary*}[Restatement of Corollary \ref{QuiteGeneralContAvExample}]
		Let $d, \eps$ and $\mu_U = \sum_{i = 1}^k p_i \delta_{U_i}$  as well as $b_1, \ldots , b_k$ be as in Proposition~\ref{QuiteGeneralExample}. Let $q$ be a prime number and  consider for $1 \leq i \leq k-1$ $$g_i(x) = \frac{q}{q+3}U_ix + b_i \quad \quad \text{ and } \quad \quad g_k(x) = \frac{q}{q-1}U_kx + b_k.$$ Assume that $g_1, \ldots , g_k$ do not have a common fixed point and further that $$p_k \leq \frac{1}{3}.$$ Then the self-similar measure of $\mu = \sum_{i = 1}^k p_i\delta_{g_i}$ is absolutely continuous for $q$ a sufficiently large prime depending on $d,\eps,U_1, \ldots , U_k$ and $b_1, \ldots , b_k$. 
	\end{corollary*}

\begin{proof}[Proof of Corollary~\ref{QuiteGeneralContAvExample}.]
    As in the proof of Corollary~\ref{QuiteGeneralExample}, $g_1$ and $g_2$ generate a free semigroup for sufficiently large $q$ and therefore $h_{\mu} \gg \eps $. Write $\alpha_1 = p_1 + \ldots + p_{k-1}$ and $\alpha_2 = p_k$. Then we have as $\alpha_1 + \alpha_2 = 1$,
    \begin{align*}
        \E_{\gamma\sim\mu}[\rho(\gamma)] &= \alpha_1\frac{q}{q + 3} + \alpha_2\frac{q}{q-1} = \frac{q^2 + (4\alpha_2-1)q}{(q + 3)(q-1)}
    \end{align*} and thus $$1 - \E_{\gamma\sim\mu}[\rho(\gamma)] = \frac{(q + 3)(q-1) - (q^2 + (4\alpha_2-1)q)}{(q + 3)(q-1)} = \frac{(3-4\alpha_2)q - 3}{(q + 3)(q-1)}.$$ On the other hand, choosing $\hat{\rho} = \frac{q}{q + 3}$ we have $$\E_{\gamma\sim\mu}[|\hat{\rho} - \rho(\gamma)|] = \alpha_2 \left( \frac{q}{q-1} - \frac{q}{q + 3} \right) = \frac{4q\alpha_2}{k(q + 3)(q-1)}.$$ Thus it follows that 
    \begin{equation}\label{ContAvLimit}
        \lim_{q \to \infty}  \frac{\E_{\gamma\sim\mu}[|\hat{\rho} - \rho(\gamma)|]}{1 - \E_{\gamma\sim\mu}[\rho(\gamma)]} = \frac{4\alpha_2}{3-4\alpha_2} < 1
    \end{equation}
    provided that $\alpha_2 = p_k < \frac{3}{8}$. If we assume that $p_k \leq \frac{1}{3}$ then we have that the limit in \eqref{ContAvLimit} is uniformly away from $1$. As in Corollary~\ref{QuiteGeneralExample}, we have that $S_{\mu} \ll_{K,d} \log q$.  Therefore by Theorem~\ref{MainResultContAvFixU} there exists a constant $C$ depending on all of the parameters such that $\mu$ is absolutely continuous if $$C|\chi_{\mu}| \leq \frac{1}{(\log \log q)^2}.$$ As in Corollary~\ref{QuiteGeneralExample} it follows that $|\chi_{\mu}| \ll q^{-1}$ and hence the claim follows. 
\end{proof}

We next prove Corollary~\ref{MostGeneralExample} and first show the following basic lemma.

\begin{lemma}\label{BasicGaloisLemma}
    Let $K$ be a real algebraic number field satisfying $\Q(\sqrt{q}) \subset K$ for a prime $q$. Then there exists a field automorphism $\Phi \in \mathrm{Aut}(K/\Q)$ such that $\Phi(\sqrt{q}) = -\sqrt{q}$.
\end{lemma}

\begin{proof}
    Write $K_0 = \Q(\sqrt{q})$ and assume that $K = K_0(\alpha_1, \ldots , \alpha_{\ell})$ for some $\alpha_1, \ldots , \alpha_{\ell} \in K$. Denote by $\Theta \in \mathrm{Aut}(K_0/\Q)$ the automorphism with $\Theta(\sqrt{q}) = -\sqrt{q}$. When $\ell = 1$ we consider the surjective map $K_0[X] \to K_0(\alpha)$ with $P \mapsto \Theta(P)(\alpha_1)$ for $\Theta(P)$ the polynomial to which all coefficients we have applied $\Theta$. This map induces a field automorphism of $K_0(\alpha)$ with the required properties and our proof is concluded by an induction on $\ell$ with the same argument.
\end{proof}

\begin{corollary*}[Restatement of Corollary \ref{MostGeneralExample}]
		Let $d \geq 1$ and $\eps \in (0,1)$ and $\mu_U = \sum_{i = 1}^k p_i\delta_{U_i}$ an irreducible probability measure on $O(d)$ with $p_i \geq \eps$ for all $1\leq i \leq k$. Assume furthermore that $U_1, \ldots , U_k$ have algebraic entries. Let $\tilde{\rho} \in (0,1)$ be sufficiently close to $1$ in terms of $d,\eps$ and $\mu_U$ and let $C > 1$ be sufficiently large depending on the same parameters.  
		
		Suppose that $g_i(x) = \frac{a_i + b_i\sqrt{q}}{c_i}U_ix + d_i$ with $a_i, b_i, c_i \in \Z$ and $d_i \in \Z^d$ for $1 \leq i \leq k$ and a prime number $q$ do not have a common fixed point. Then the self-similar measure associated to $\mu = \sum_{i = 1}^k p_i\delta_{g_i}$ is absolutely continuous if the following properties are satisfied: 
		\begin{enumerate}[label = (\roman*)]
			\item $\frac{a_i + b_i\sqrt{q}}{c_i} \in (\tilde{\rho},1)$ for $1\leq i \leq k$,
			\item for $j = 1$ and for $j = 2$ we have $$\bigg|\frac{a_j - b_j\sqrt{q}}{c_j}\bigg| < \frac{1}{3},$$
			\item For $L = \max(\sqrt{q}, |a_i|, |b_i|, |c_i|, |d_i|_{\infty})$ we have $$C|\chi_{\mu}| \leq \frac{1}{(\log(\log L))^2}.$$
		\end{enumerate}
	\end{corollary*}

\begin{proof}[Proof of Corollary~\ref{MostGeneralExample}.]
    By Theorem~\ref{MainResultFixU} there exists $\tilde{\rho} \in (0,1)$ and $C \geq 1$ depending on $d,\eps$ and $\mu_U$ such that $\mu$ is absolutely continuous if $p_i \geq \eps$ as well as $\frac{a_i + b_i\sqrt{q}}{c_i} \in (\tilde{\rho},1)$ for all $1\leq i \leq k$ as well as $$\frac{h_{\mu}}{|\chi_{\mu}|} \geq  C \left( \max\left\{ 1, \log \frac{S_{\mu}}{h_{\mu}} \right\} \right)^2.$$ Let $K$ be the number field generated by all the coefficients of elements in $\mathrm{supp}(\mu)$. Then by Lemma~\ref{BasicGaloisLemma} there is a field automorphism $\Phi \in \mathrm{Aut}(K/\Q)$ such that $\Phi(\sqrt{q}) = -\sqrt{q}$ and therefore we have that $|\rho(\Phi(g_j))| < \frac{1}{3}$ for $j = 1,2$. Thus by Lemma~\ref{lemm:GaloisPingPong} and Lemma~\ref{lemm:FreeEntBound} we have that $h_{\mu} \gg \eps$. We also have $h_{\mu} \leq \log \eps^{-1}$. On the other hand, it follows by Lemma~\ref{lemma:height_basic} (iii) and Proposition~\ref{HeightSeparationBound} that $S_{\mu} \ll_{d,\mu_U} \log L$, which readily implied the claim upon changing the constant $C$.
\end{proof}

\begin{lemma}\label{MostGeneralParticualr}
    In the setting of Corollary~\ref{MostGeneralExample}, for $\eps > 0$ choose $$a_i = \lceil \sqrt{q} \rceil - m_{i,q}, \quad \quad b_i = 2 \quad\quad c_i = 3\lceil \sqrt{q} \rceil$$ for  $m_{i,q}$ an integer satisfying $m_{i,q} \in [0,q^{1/2 - \eps}]$ and any $d_i \in \Z^d$ with $|d_i|_{\infty} \leq \exp(\exp(q^{\eps/3}))$. Then $\mu$ is absolutely continuous for sufficiently large $q$ depending on $d,p_0, \eps$ and $U_1, \ldots , U_k$, provided $g_1, \ldots , g_k$ does not have a common fixed point.
\end{lemma}

\begin{proof}
    It holds that $\frac{a_i + b_i\sqrt{q}}{c_i} \in (0,1)$ converges to $1$ as $q \to \infty$ and that $|\frac{a_i - b_i\sqrt{q}}{c_i}| < \frac{1}{3}$. We next estimate the Lyapunov exponent of $\mu$. Indeed, note that for $q$ large enough,
    \begin{align*}
        \log\left( \frac{a_i + b_i\sqrt{q}}{c_i} \right) &\geq \log\left( \frac{\lceil \sqrt{q} \rceil - q^{1/2-\eps} + 2\sqrt{q}}{3\lceil \sqrt{q} \rceil }\right) \\
        &\geq \log\left( 1 - \frac{2(\lceil \sqrt{q}\rceil - \sqrt{q}) + q^{1/2-\eps}}{3\lceil \sqrt{q} \rceil }\right) \gg -q^{-\eps}
    \end{align*} and therefore $|\chi_{\mu}| \ll q^{-\eps}$. In our case, for large $q$ we have  $L = |d_i|_{\infty} = \exp(\exp(q^{\eps/3}))$ and therefore $\log(\log L) = q^{\eps/3}$. Thus for sufficiently large $q$ we have that $C|\chi_{\mu}| \leq  (\log\log L)^{-2} = q^{-2\eps/3}$ and the claim follows. 
\end{proof}

\subsection{Real Bernoulli Convolutions}
\label{section:RealBern}

In this section we prove Corollary~\ref{BernoulliCorollary}.

\begin{corollary*}[Restatement of Corollary \ref{BernoulliCorollary}]
		There is an absolute constant $C > 1$ such that the following holds. Let $\lambda \in (1/2,1)$ be a real algebraic number. Then the Bernoulli convolution $\nu_{\lambda}$ is absolutely continuous on $\R$ if 
		\begin{equation}\label{BernoulliCond}
		\lambda > 1 - C^{-1}\min\{\log M_{\lambda},(\log \log M_{\lambda})^{-2} \} .
		\end{equation}
\end{corollary*}

\begin{proof}[Proof of Corollary~\ref{BernoulliCorollary}]
        As in the paragraph before Proposition~\ref{prop:non_degen}, Bernoulli convolutions are uniformly non-degenerate. Since we are in $d = 1$ they are $(1,0)$-well-mixing and therefore Theorem~\ref{MainResult} applies. For convenience write $\eta = \log M_{\lambda}$ and $h_{\lambda} = h_{\nu_{\lambda}}$. We don't keep track of possible enlargements of $C$. That Bernoulli convolutions are uniformly non-degenerate follows from Proposition~\ref{UnifNonDeg}. Then Theorem~\ref{MainResult} implies that if 
        \begin{equation}\label{BernoulliMain}
            (1-\lambda)^{-1}h_{\lambda} > C \left(\max\left\{ 1, \log \eta/h_{\lambda} \right\} \right)^2,
        \end{equation}
        then $\nu_{\lambda}$ is absolutely continuous. Recall that by \cite[Theorem 5]{BreuillardVarju2020} (which is stated with logarithms base 2) there is an absolute $c_0 \in (0,1)$ such that $c_0 \min(\log 2,\eta) \leq h_{\lambda} \leq \min(\log 2,\eta).$ 
        
        We proceed with a case distinction. First assume that $\eta \leq \log 2$. Then $c_0^{-1} \geq \eta/h_{\lambda} \geq 1$ and therefore by \eqref{BernoulliMain} the condition $(1-\lambda)^{-1}c_0 \eta > C$ is sufficient for absolute continuity, which is equivalent to 
        \begin{equation}\label{BernoulliCondPart1}
            \lambda > 1 - C^{-1}\eta.
        \end{equation}
        
        Next assume that $\eta \geq \log 2$. Then $c_0\log 2 \leq h_{\lambda} \leq \log 2$ and so \eqref{BernoulliMain} gives $$(1-\lambda) \max\{ 1, \log \eta + \log (c_0\log 2)^{-1} \}^2 < C^{-1}.$$ Note that $\max\{ 1, \log \eta + \log (c_0\log 2)^{-1} \} \leq 2\log (c_0\log 2)^{-1}\max\{ 1, \log \eta \}$. Therefore we get the condition 
        \begin{equation}\label{BernoulliCondPart2}
            \lambda > 1 - C^{-1}\max\{ 1, \log \eta \}^{-2} = 1 - C^{-1}\min\{ 1, (\log \eta)^{-2} \}.
        \end{equation}

        To deduce \eqref{BernoulliCond}, we note that there is a unique $\eta' > 0$ with $\eta' = (\log \eta')^{-2}$ and this $\eta'$ satisfies $2 \leq \eta' \leq 5/2$. Moreover $\log \eta < (\log \eta)^{-2}$ for $0 < \eta < \eta'$ and $\log \eta > (\log \eta)^{-2}$ for $\eta > \eta'$. Therefore \eqref{BernoulliCond} holds for $\eta < \log(2)$ and $\eta > 2\eta'$ by \eqref{BernoulliCondPart1} and \eqref{BernoulliCondPart2}. In the range $\log(2) < \eta < 2\eta'$, we enlarge $C$ to ensure that \eqref{BernoulliCond} holds. 
    \end{proof}

    We note that if $\lambda$ is algebraic and not the root of any non-zero polynomial with coefficients $0,\pm 1$, then $h_{\lambda} = 2$ and also as mentioned in Remark 5.10 of \cite{Kittle2021}, $M_{\lambda} \geq 2$. Therefore for such a $\lambda$, $\nu_{\lambda}$ is absolutely continuous if 
    \begin{equation}\label{BernoulliCondSimplified}
            \lambda > 1 - C^{-1}\min\{1,(\log \log M_{\lambda})^{-2} \}.
    \end{equation}

\subsection{Complex Bernoulli Convolutions}
\label{section:ComplexBern}

\begin{corollary*}[Restatement of Corollary \ref{ComplexBernoulliCorollary}]
		For every $\varepsilon > 0$ there is a constant $C \geq 1$ such that the following holds. Let $\lambda \in \mathbb{C}$ be a complex algebraic number such that $|\lambda| \in (2^{-1/2}, 1)$ and 
		\begin{equation}
		|\mathrm{Im}(\lambda)| \geq \varepsilon. \label{eq:SomeImaginaryPart}
		\end{equation} Then the Bernoulli convolution $\nu_{\lambda}$ is absolutely continuous on $\C$ if 
		\begin{equation*}
		|\lambda| > 1 - C^{-1}\min\{\log M_{\lambda},(\log \log M_{\lambda})^{-2} \} .
		\end{equation*}
\end{corollary*}

\begin{proof}[Proof of Corollary~\ref{ComplexBernoulliCorollary}]
    We can't directly apply Proposition~\ref{UnifNonDeg} so we give a direct proof of mixing and non-degeneracy. First note that \eqref{eq:SomeImaginaryPart} ensures that there is some $c > 0$ and $T \geq 1$ depending only on $\varepsilon$ such that the $(c, T)$-well-mixing property is satisfied. 
    
    To deal with non-degeneracy, we distinguish the case when $|\lambda| \leq \lambda_0$ and $|\lambda| \geq \lambda_0$ for some $\lambda_0$ sufficiently close to $1$. As in the case of real Bernoulli convolution, for any given $\lambda_0$, the family of Bernoulli convolutions with $|\lambda| \leq \lambda_0$ are easily seen to be uniformly non-degenerate depending on $\lambda_0$. To deal with the case $\lambda \geq \lambda_0$, we rescale our measure to the one given by the law of $B_{\lambda} = \sqrt{1 - |\lambda|^2} \sum_{i=0}^{\infty} \pm \lambda^i$ and denote the resulting measure by $\nu_{\lambda}'$. Now let $\Sigma$ be the covariance matrix of $\nu_{\lambda}'$ under the natural identification of $\mathbb{C}$ with $\R^2$. Note that the trace of $\Sigma$ is $1$ and we claim that the smallest eigenvalue of $\Sigma$ is  $\gg_{\eps} 1$. Indeed, for a unit vector $x \in \R^2$ we want to estimate $x^T\Sigma x$, which is by identifying $\C$ with $\R^2$ equal to $$\E[|B_{\lambda}\cdot x|^2] = (1 - |\lambda|^2) \sum_{i = 0}^{\infty} |\lambda^i \cdot x|^2 \gg_{\eps} 1,$$ which follows as $|\lambda^i\cdot x|^2 \gg |\lambda|^2$ unless $\lambda^i$ and $x$ are almost colinear, which is only the case for a very small proportion of $i$'s. It follows that $$\inf_{p\in P(2)} \E_{x\sim \mathcal{N}(0,\Sigma)}[|px|]  \gg_{\eps} 1$$ for $p$ ranging in the orthogonal projections of $\R^2$ as in section~\ref{section:non-deg}. By for example Lemma~\ref{BerryEssenType} we know that $\mathcal{W}_1\left( \nu_{\lambda}', N(0, \Sigma) \right) \ll \sqrt{1 - |\lambda|^2}.$ Therefore for $\lambda_0$ sufficiently close to $1$ in terms of $\varepsilon$, uniform non-degeneracy follows as in Lemma~\ref{CloseToNormalNonDegen}. Having establish uniform well-mixing and non-degeneracy, Corollary~\ref{ComplexBernoulliCorollary} is established by the same argument as the proof of Corollary \ref{BernoulliCorollary}.
\end{proof}

\bibliography{referencesgeneral.bib}
\end{document}